\documentclass{article}
\usepackage{amsmath,hyperref}
\usepackage{amsfonts,amssymb,mathrsfs}
\usepackage{theorem}
\usepackage{tikz}
\usepgflibrary{arrows}
\usetikzlibrary{matrix}

\input amssym.def

\numberwithin{equation}{section}

\numberwithin{figure}{section}

\makeatletter
\newcommand\qedsymbol{\hbox{$\Box$}}
\newcommand\qed{\relax\ifmmode\Box\else
  {\unskip\nobreak\hfil\penalty50\hskip1em\null\nobreak\hfil\qedsymbol
  \parfillskip=\z@\finalhyphendemerits=0\endgraf}\fi}

\newenvironment{proof}[1][{}]{\par\noindent Proof{#1}. }{\qed}

\DeclareMathOperator*{\colim}{colim}

\newcommand{\bfzero}{{\bf 0}}

\newcommand{\Isom}{\mathsf{Isom}}
\newcommand{\Tree}{\mathsf{Tree}}

\newcommand{\lp}{{\circlearrowleft}}
\newcommand{\nl}{\not{\circlearrowright}}

\newcommand{\grt}{{\mathfrak{grt}}}

\newcommand{\Cobar}{\mathrm{Cobar}}

\newcommand{\dGra}{\mathsf{dGra}}
\newcommand{\KGra}{\mathsf{KGra}}
\newcommand{\SGra}{\mathsf{SGra}}

\newcommand{\dfGC}{{\mathsf{dfGC}}}
\newcommand{\dgra}{\mathrm{dgra}}

\newcommand{\Hg}{{\mathsf{Hg}}}
\newcommand{\Cg}{{\mathrm{Cg}}}

\newcommand{\Lie}{{\mathsf{Lie}}}
\newcommand{\Com}{{\mathsf{Com}}}
\newcommand{\As}{{\mathsf{As}}}

\newcommand{\coCom}{{\mathsf{coCom}}}
\newcommand{\coAs}{{\mathsf{coAs}}}
\newcommand{\Ger}{{\mathsf{Ger}}}

\newcommand{\OC}{{\mathsf{OC}}}
\newcommand{\br}{\mathrm{br}} 

\newcommand{\Ch}{\mathsf{Ch}} 
\newcommand{\grVect}{\mathsf{grVect}}

\newcommand{\C}{{\mathscr{C}}}

\newcommand{\Op}{{\mathbb{OP}}}
\newcommand{\psop}{\Psi{\mathbb{OP}}}

\newcommand{\oc}{{\mathfrak{o} \mathfrak{c}}}
\newcommand{\GRT}{{\mathsf{GRT}}}

\newcommand{\Conv}{{\mathrm{Conv}}}

\newcommand{\Coder}{{\mathrm{Coder}}}

\newcommand{\mult}{{\mathrm{mult}}}

\newcommand{\CH}{\mathrm{CH}}
\newcommand{\Sh}{\mathrm{Sh}}

\newcommand{\End}{\mathsf{End}}
\newcommand{\Hom}{\mathrm{Hom}}
\newcommand{\nod}{{\mathrm{nod}}}
\renewcommand{\max}{{\mathrm{max}}}

\newcommand{\inv}{{\mathrm {inv} }}
\newcommand{\Hoch}{{\mathrm{Hoch}}}

\newcommand{\id}{\mathrm{id}}
\newcommand{\ad}{{\mathrm{ad}}}

\newcommand{\ti}[1]{{\tilde{#1}}}
\newcommand{\wt}[1]{{\widetilde{#1}}}

\newcommand{\und}[1]{{\underline{#1}}}
% ...

\newcommand{\sh}{\sharp}
\newcommand{\dia}{\diamond}

\newcommand{\Cbu}{C^{\bullet}}

\newcommand{\tG}{\widetilde{\Gamma}}

\newcommand{\al}{{\alpha}}
\newcommand{\la}{{\lambda}}

\newcommand{\bul}{{\bullet}}

\newcommand{\ww}{{\circ\, \circ}}
\newcommand{\ed}{{\bullet\hspace{-0.05cm}-\hspace{-0.05cm}\bullet}}
\newcommand{\pike}{{-\hspace{-0.10cm}\small{\bullet}}}

\newcommand{\mC}{{\mathfrak{C}}}

\newcommand{\mc}{\mathfrak{c}}
\newcommand{\md}{\mathfrak{d}}
\newcommand{\me}{\mathfrak{e}}
\newcommand{\mg}{\mathfrak{g}}

\newcommand{\ml}{\mathfrak{l}}
\newcommand{\mo}{{\mathfrak{o}}}

\newcommand{\om}{{\omega}}

\newcommand{\si}{{\sigma}}
\newcommand{\ga}{{\gamma}}

\newcommand{\ve}{{\varepsilon}}
\newcommand{\vs}{{\varsigma}}
\newcommand{\ka}{{\kappa}}
\newcommand{\G}{{\Gamma}}
\newcommand{\pa}{{\partial}}

\newcommand{\st}{{\mathsf{t}}}

\newcommand{\bs}{{\mathbf{s}}}
\newcommand{\bsi}{{\mathbf{s}^{-1}\,}}
\newcommand{\bt}{{\mathbf{t}}}
\newcommand{\bq}{{\mathbf{q}}}
\newcommand{\bu}{{\mathbf{u}}}

\newcommand{\cC}{{\cal C}}

\newcommand{\cN}{{\cal N}}
\newcommand{\cL}{{\cal L}}
\newcommand{\cD}{{\cal D}}
\newcommand{\cA}{\mathcal{A}}
\newcommand{\cF}{\mathcal{F}}

\newcommand{\cZ}{{\mathcal{Z}}}

\newcommand{\cW}{{\cal W}}

\newcommand{\cO}{{\cal O}}
\newcommand{\cV}{{\cal V}}

\newcommand{\bbK}{{\mathbb K}}

\newcommand{\bbR}{{\mathbb R}}
\newcommand{\bbZ}{{\mathbb Z}}
\newcommand{\bbQ}{{\mathbb Q}}

\newcommand{\La}{{\Lambda}}
\newcommand{\Ups}{{\Upsilon}}

\newcommand{\te}{\theta}

\newcommand{\de}{{\delta}}
\newcommand{\D}{{\Delta}}
\newcommand{\Lap}{\underline{\Delta}}

\newcommand{\sgn}{\mathrm{sgn}}

\newcommand{\hotimes}{{\,\hat{\otimes}\,}}

\date{}

\newtheorem{thm}{Theorem}[section]

\newtheorem{defi}[thm]{Definition}

\newtheorem{lem}[thm]{Lemma}
\newtheorem{cor}[thm]{Corollary}

\newtheorem{prop}[thm]{Proposition}
\newtheorem{claim}[thm]{Claim}
\newtheorem{cond}[thm]{Condition}

\theorembodyfont{\rm}
\newtheorem{example}[thm]{Example}
\newtheorem{pty}[thm]{Property}
\newtheorem{remark}[thm]{Remark}

\title{Stable Formality Quasi-isomorphisms for Hochschild Cochains}

\author{V. A. Dolgushev}
\date{}

\begin{document}

\large

\maketitle

\begin{flushright}
{\it To my mother}
\end{flushright}

~\\[-0.3cm]

\begin{abstract}
We consider $L_{\infty}$-quasi-isomorphisms for 
Hochschild cochains whose structure maps admit 
``graphical expansion''. We introduce the notion 
of stable formality quasi-isomorphism which formalizes 
such an $L_{\infty}$-quasi-isomorphism. We define a homotopy 
equivalence on the set of  stable formality quasi-isomorphisms and 
prove that the set of homotopy classes of 
stable formality quasi-isomorphisms form a torsor 
for the group corresponding to the zeroth cohomology
of the full (directed) graph complex.  
This result may be interpreted as a complete description 
of homotopy classes of formality quasi-isomorphisms 
for Hochschild cochains in the ``stable setting''.  
\end{abstract}

~\\
{\small
{\bf Keywords:} Operads, formality morphisms, graph complexes.\\
{\bf MSC 2010:} 18D50, 18G55, 55U15.}

\bigskip 
\bigskip

\begin{center}
{\bf Quasi-isomorphismes stables de formalit\'e pour les cocha\^{i}nes de Hochschild}\\[0.3cm]
{\small\bf  R\'esum\'e}\\[0.3cm]
\begin{minipage}[c]{0.9\linewidth}
{\small 
Nous consid\'erons $L_{\infty}$-quasi-isomorphismes pour les cocha\^{i}nes de 
Hochschild dont les applications structurelles admettent une ``expansion graphique''.
Nous introduisons la notion de quasi-isomorphisme stable de formalit\'e qui 
formalise les $L_{\infty}$-quasi-isomorphismes de ce genre. 
Nous d\'efinissons une \'equivalence homotopique sur l'ensemble 
des quasi-isomorphismes stables de formalit\'e. Nous prouvons que 
l'ensemble des classes homotopique de quasi-isomorphismes stables de 
formalit\'e est un torseur pour le groupe correspondant \`a la cohomologie de degr\'e z\'ero du graphe-complexe complet (direct).
Ce r\'esultat peut-\^{e}tre interpr\'et\'e comme une description compl\`ete des classes homotopiques de 
quasi-isomorphismes de formalit\'e pour les cocha\^{i}nes de 
Hochschild dans le ``cadre stable''.  
}
\end{minipage}
\end{center}

~\\
{\small
{\bf Mots-cl\'es:} Op\'erades, morphismes de formalit\'e, graphe-complexes.\\
{\bf MSC 2010:} 18D50, 18G55, 55U15.}

\newpage

\tableofcontents

\section{Introduction}

When a difficult problem is solved, it becomes even more 
challenging to describe all possible solutions to that problem.
In this paper we propose a framework in which this 
interesting question can be answered completely 
for Kontsevich's formality conjecture \cite{K-conj}
on Hochschild cochain complex. 

Kontsevich's formality conjecture \cite{K-conj}
states that there exists an  $L_{\infty}$ quasi-isomorphism 
from the graded Lie algebra $V_A$ of polyvector fields on an affine space  
to the dg Lie algebra of Hochschild cochains $\Cbu(A)$ of 
the algebra of functions $A$ on this affine space. 

In plain English the question was to find an infinite collection of 
maps 
\begin{equation}
\label{U_n}
U_n : \big( V_A \big)^{\otimes\, n} \to \Cbu(A)\,, \qquad n \ge 1
\end{equation}
compatible with the action of symmetric groups and
satisfying an intricate sequence of relations. The first relation 
says that $U_1$ is a map of complexes, the second relation
says that $U_1$ is compatible with the Lie brackets up to 
homotopy with $U_2$ serving as a chain homotopy and so on.

In his groundbreaking paper \cite{K} M. Kontsevich 
proposed a construction of such an 
$L_{\infty}$ quasi-isomorphism over reals. His construction is 
``natural'' in the following sense. Given polyvector fields
$v_1, v_2, \dots, v_n \in V_A$, the $n$-th component 
$U_n$ produces  a Hochschild cochain via contracting 
indices of  derivatives of various orders of polyvector fields and 
of functions which enter as arguments for this cochain. 

Thus each term in $U_n$ can encoded by a directed graph with 
two types of vertices: vertices of one type are reserved for polyvector fields 
and vertices of another type are reserved for functions. 

In this paper  we formalize the notion of $L_{\infty}$
quasi-isomorphism for Hochschild cochains which
are ``natural'' in the above sense. In other words, each term in 
$U_n$ is encoded by a graph with two types of vertices
and all the desired identities hold universally, i.e. on 
the level of linear combinations of graphs.
  
Such formality quasi-isomorphisms are defined for 
affine spaces of all\footnote{In fact they are also defined 
for any $\bbZ$-graded affine space.} (finite) dimensions simultaneously. 
This is why we refer to them as {\it stable formality 
quasi-isomorphisms} (SFQs). We show that the notion of 
homotopy equivalence of formality quasi-isomorphisms 
can also be formulated in this ``stable setting''. 
Thus we can talk about homotopy classes of stable 
formality quasi-isomorphisms.

In this paper we show (see Theorem \ref{thm:main}) that the 
set of homotopy classes of SFQs form a torsor 
for a pro-unipotent group which is obtained by 
exponentiating the Lie algebra $H^0(\dfGC)$, where 
$\dfGC$ denotes the full (directed) version of  
Kontsevich's graph complex \cite[Section 5]{K-conj}.

Following\footnote{See \cite[Corollary 3.6]{dgraphs} for more precise statement.} 
T. Willwacher \cite{Thomas} the group
$\exp(H^0(\dfGC))$ is isomorphic to the  
Grothendieck-Teichm\"uller group $\GRT_1$ introduced by 
V. Drinfeld in \cite{Drinfeld}.
Thus combining  Theorem \ref{thm:main} with the result 
of T. Willwacher \cite{Thomas}, we conclude that the 
set of homotopy classes of SFQs is a $\GRT_1$-torsor. 

Since a formality quasi-isomorphism for Hochschild cochains
provides us with a bijection between equivalence classes of 
star products and equivalence classes of formal Poisson structures, 
the result may be interpreted as a complete description of all 
(deformation) quantization procedures.

To give a precise definition of an SFQ, we recall \cite{HL-OCHA}, \cite{OCHA} that an open-closed
homotopy algebra ($\OC$-algebra) is a pair
$(\cV, \cA)$ of cochain complexes  
with the following data: 
\begin{itemize}

\item an $L_{\infty}$-structure on $\cV$, 

\item an $A_{\infty}$-structure on $\cA$ and 

\item an $L_{\infty}$-morphism from $\cV$ to the 
Hochschild cochain complex $\Cbu(\cA)$ of $\cA$\,.

\end{itemize}
We denote by $\OC$ the 2-colored dg operad which 
governs open-closed homotopy algebras. It is known that 
$\OC$ is free as an operad in the category of graded vector 
spaces. Furthermore, $\OC$ is the cobar construction of 
a (2-colored) cooperad closely connected with the homology of 
Voronov's Swiss Cheese operad \cite{Sasha-SC}.   

Let us also denote by $\KGra$ the 2-colored operad which 
is ``assembled'' from graphs used in Kontsevich's paper \cite{K}. 
This 2-colored operad extends the operad $\dGra$ of directed labeled graphs 
and acts naturally on the pair ``polyvector fields $V_A$ and polynomials $A$''. 
(See Section \ref{sec:dGra-KGra} for more details.)  

Let us observe that any map of (dg) operads from $\OC$ to $\KGra$ induces 
an open-closed homotopy algebra on the pair $(V_A, A)$\,. 
So we define an SFQ as 
a map of (dg) operads from $\OC$ to $\KGra$ subject to 
a few ``boundary conditions''. These conditions guarantee that
\begin{itemize}

\item  the  $L_{\infty}$-structure on polyvector fields coincides with the 
Lie algebra structure given by the Schouten-Nijenhuis bracket, 

\item the $A_{\infty}$-structure on $A$ coincides with the usual associative
(and commutative) algebra structure on polynomials, and 

\item the $L_{\infty}$-morphism from polyvector fields $V_A$ to Hochschild 
cochains  $\Cbu(A)$ starts with the Hochschild-Kostant-Rosenberg \cite{HKR}
embedding  
$$
V_A  \hookrightarrow \Cbu(A)\,.
$$

\end{itemize}

This operadic definition allows us to introduce a natural notion of homotopy equivalence
on the set of SFQs. We give this 
definition using an interpretation of SFQs
as Maurer-Cartan (MC) elements of an auxiliary dg Lie algebra.  

Let us denote by $\cZ^0(\dfGC)$ the Lie algebra of 
degree zero cocycles of the full directed version $\dfGC$ of  
Kontsevich's graph complex \cite[Section 5]{K-conj}. 
It is not hard to see that $\cZ^0(\dfGC)$ is a pro-nilpotent 
Lie algebra. Hence it can be exponentiated to the 
group   $\exp\big( \cZ^0(\dfGC) \big)$\,.

We show that the group $\exp\big( \cZ^0(\dfGC) \big)$ acts on SFQs
and this action descends to an action of the group
$\exp\big(H^0(\dfGC) \big)$ on homotopy classes of SFQs.

Finally, we prove that this action of  $\exp\big(H^0(\dfGC) \big)$ 
on homotopy classes is simply transitive. 

Specialists can probably start reading this paper 
with Section \ref{sec:dGra-KGra}. 
The goal of Section \ref{sec:prelim} is mostly to fix conventions
and remind a few constructions for colored (co)operads.
In Section \ref{sec:dGra-KGra}, we define the operad of graded vector spaces 
$\dGra$ and its 2-colored extension $\KGra$\,. In this section 
we also introduce a natural action of $\KGra$ on the pair 
 ``polyvector fields $V_A$ and polynomials $A$''. 
In Section \ref{sec:OC}, we remind the (dg) operad $\OC$ which 
governs open-closed homotopy algebras \cite{OCHA}.  
In Section \ref{sec:stable}, we introduce SFQs
and define a notion of homotopy equivalence between them.
Section \ref{sec:dfGC} is devoted to the full graph complex $\dfGC$ and
its ``action'' on SFQs. The main result of this paper (Theorem \ref{thm:main})
is stated at the end of Section \ref{sec:dfGC}. Its proof occupies Section \ref{sec:transit} 
and \ref{sec:free} and it depends on a few technical statements which 
are proved in Appendices at the end of the paper.

~\\
{\bf Acknowledgment.}
In many respects this work was inspired by 
papers \cite{K}, \cite{K-conj}, \cite{K-nc-symp}, \cite{K-mot}
and \cite{Thomas} and I would like to thank Thomas Willwacher 
for numerous discussions of his work  \cite{Thomas} and for his 
comments on the original draft of this paper.  The ideas of paper \cite{Thomas-stable}
by Thomas Willwacher were used to streamline the proof of the fact that the action of 
$\exp\big(H^0(\dfGC)\big)$ on the set of homotopy classes of SFQs is free.
In this respect, the current version of this manuscript benefited from paper \cite{Thomas-stable}.
I would like to thank Chris Rogers for collaboration
on \cite{notes}, \cite{dgraphs}, and for  numerous discussions.
The results of this paper were presented at multiple seminars, 
at  XXX Workshop on Geometric Methods in Physics in June of 2011
(Bialowieza, Poland) and at Geometry and Physics (GAP) XI in 
August of 2013 (Pittsburgh, PA).
I would like to thank Murray Gerstenhaber and Jim Stasheff for   
the comments they made during my talk on this paper at 
their Deformation Theory Seminar.  I would like to thank an anonymous referee 
for carefully reading my manuscript and for her/his remarks.
I would like to thank Elena Roubtsov and 
Volodya Roubtsov for their help with the French version of the abstract for this paper. 
I would like to thank my brother-in-law Igal Vainer
for his help and encouragement. I would also like to acknowledge 
the following NSF grants DMS 0856196, DMS-1124929,  
DMS-1161867, and DMS-1501001. 

Many years ago, my mother abandoned her 
unfinished PhD thesis in mathematics to be able to devote more time to my brother and me when we 
were kids. I would like to thank my mother for her devotion to us and {\it humbly devote this paper to her.}

\section{Preliminaries}
\label{sec:prelim}

We denote by $\bbK$ a field of characteristic zero. 
Our underlying symmetric monoidal category $\mC$ is 
either the category $\grVect_{\bbK}$ of $\bbZ$-graded
$\bbK$-vector spaces or the category 
$\Ch_{\bbK}$ of unbounded {\it cochain} complexes of $\bbK$-vector spaces.  
In this paper, we use exclusively cohomological conventions. 
The notation $\ad_{\xi}$ is reserved for the adjoint 
action $[\xi, ~]$ of a vector $\xi$ in a Lie algebra and 
the expression $\CH(x,y)$ denotes the Campbell-Hausdorff 
series in variables $x$ and $y$.  

The notation $S_n$ is reserved for the group of permutations 
of the set $\{1,2, \dots, n\}$ and $\Sh_{p_1,p_2, \dots, p_k} $,
with $p_i \ge 0$ and $p_1 + p_2 + \dots  + p_k =n$,
denotes the subset 
of $(p_1,p_2, \dots, p_k)$-shuffles in $S_n$, i.e.
\begin{equation}
\label{Sh}
\Sh_{p_1,p_2, \dots, p_k} = 
\Big\{\si \in S_n ~~\Big|~~  \si(1) < \dots < \si(p_1), \end{equation}
$$
\si(p_1+1) < \dots < \si(p_1+p_2),~~ \dots,~~  \si(n-p_k+1) < \dots < \si(n)
  \Big\}\,.
$$
We often denote by $\id$ the identity element of $S_n$ without 
specifying the number $n$.

We denote by $\Com$ (resp. $\As$) the operad 
which governs commutative (and associative) algebras without unit 
(resp. associative algebras without unit). The notation $\Lie$ is 
reserved for the operad which governs Lie algebras. 
Dually, we denote by $\coCom$ (resp. $\coAs$) the cooperad 
which governs cocommutative (and coassociative) coalgebras without counit 
(resp. coassociative coalgebras without counit). 

The notation $\La$ is reserved for the following collection in $\grVect_{\bbK}$
\begin{equation}
\label{La}
\La(n) = \begin{cases}
\bs^{1-n} \sgn_{n}  \qquad {\rm if} ~~ n \ge 1\,,  \\
 \bfzero  \qquad  \qquad  \quad {\rm if} ~~ n = 0\,,
\end{cases}
\end{equation}
where $\sgn_n$ is the sign representation of $S_n$\,.

The collection \eqref{La} is equipped with a natural 
structure of an operad and a natural structure of a cooperad.
Namely,  the $i$-th elementary insertion and the 
$i$-th elementary co-insertion are given by the formula
\begin{equation}
\label{La-ins}
1_n \circ_i 1_k  = (-1)^{(1-k) (n-i)} 1_{n+k-1}
\end{equation}
and the formula 
\begin{equation}
\label{La-co-ins}
\D_i (1_{n+k-1}) =  (-1)^{(1-k) (n-i)} 1_n \otimes 1_k \,,
\end{equation}
respectively. Here $1_m$ denotes the generator 
$\bs^{1-m} 1 \in \bs^{1-m} \sgn_{m}$\,.

For an operad $\cO$ (resp. a cooperad $\cC$) we denote 
by $\La \cO$ (resp. $\La\cC$) the operad (resp. the cooperad) which 
is obtained from $\cO$ (resp. $\cC$)
by tensoring with $\La$\,. For example, a $\La\Lie$-algebra in $\grVect_{\bbK}$ is
a graded vector space $\cV$ equipped with the binary operation: 
$$
\{\,,\, \} : \cV \otimes  \cV \to \cV 
$$
of degree $-1$ satisfying the identities:  
$$
\{v_1,v_2\} = (-1)^{|v_1| |v_2|} \{v_2, v_1\}\,,
$$
$$
\{\{v_1, v_2\} , v_3\} +  
(-1)^{|v_1|(|v_2|+|v_3|)} \{\{v_2, v_3\} , v_1\} +
(-1)^{|v_3|(|v_1|+|v_2|)} \{\{v_3, v_1\} , v_2\}  = 0\,,
$$
where $v_1, v_2, v_3 $ are homogeneous vectors in $\cV$\,.

The operad $\La\Lie$ has the following free resolution 
\begin{equation}
\label{LaLie-infty}
\La\Lie_{\infty} = \Cobar (\La^2 \coCom) 
\end{equation}
which we use to define an $\infty$-version of $\La\Lie$-algebra
structure. Thus a $\La\Lie_{\infty}$-{\it structure} on a cochain complex 
$\cV$ is a MC element $Q$ in the Lie algebra 
$$
\Coder\big( \La^2 \coCom(\cV) \big) 
$$
of coderivations of the cofree coalgebra 
$\La^2 \coCom(\cV)$ subject to the auxiliary technical condition 
$$
Q \Big |_{\cV} = 0\,.
$$
A $\La\Lie_{\infty}$-{\it morphism} between  $\La\Lie$-algebras 
$(\cV, Q)$ and $(\cW, \wt{Q} )$ is a homomorphism of the 
cofree coalgebras 
$$
\La^2 \coCom(\cV) \qquad \textrm{and} \qquad \La^2 \coCom(\cW)
$$
compatible with the differentials $\pa_{\cV}  + \ad_Q$ and 
$\pa_{\cW} + \ad_{\wt{Q}}$ on $\La^2 \coCom(\cV) $ and
$\La^2 \coCom(\cW)$, respectively.  

It is not hard to see that $\La\Lie_{\infty}$-algebra
structures on a cochain complex $\cV$ are in a natural bijection with 
$L_{\infty}$-algebra structures on $\bsi \cV$\,. Moreover, it is very easy to switch back and forth 
between these algebra structures. However, for our purposes, it is much more convenient to 
work with the operad \eqref{LaLie-infty} versus the operad 
$\Cobar(\La\coCom)$ which governs $L_{\infty}$-algebras.
So, in the bulk of the paper, we adhere to the former choice.

A directed graph $\G$ consists of two finite sets $V(\G)$, $E(\G)$ and
a map $\me : E(\G) \to V(\G) \times  V(\G)$. Elements of $V(\G)$ are called 
vertices and elements of $E(\G)$ are called edges.  In this paper, we 
consider exclusively graphs without loops (i.e. cycles of length one).
In other words, the image of the map $\me$ has the empty intersection 
with the diagonal in $V(\G) \times  V(\G)$. Although we do consider graphs with the 
empty set of edges, we will tacitly assume that the set of vertices is always non-empty.
 
For example, the  graph $\G$ shown in figure \ref{fig:G-dir} has $V(\G) = \{1,2,3,4,5 \}$ and
$E(\G) = \{a, b, c, d\}$ with $\me(a)=(3,1),$  $\me(b) = (3,2)$,  and $\me(c)=  \me(d) = (2,3)$\,. 
\begin{figure}[htp] 
\centering 
\begin{minipage}[t]{0.45\linewidth} 
\centering 
\begin{tikzpicture}[scale=0.5, >=stealth']
\tikzstyle{w}=[circle, draw, minimum size=3, inner sep=1]
\tikzstyle{b}=[circle, draw, fill, minimum size=3, inner sep=1]
\node [b] (b1) at (0,0) {};
\draw (-0.4,0) node[anchor=center] {{\small $1$}};
\node [b] (b3) at (2,0) {};
\draw (2, 0.5) node[anchor=center] {{\small $3$}};
\node [b] (b2) at (6,0) {};
\draw (5.9, 0.6) node[anchor=center] {{\small $2$}};
\node [b] (b4) at (2,2) {};
\draw (2, 2.6) node[anchor=center] {{\small $4$}};
\node [b] (b5) at (5,2) {};
\draw (5, 2.6) node[anchor=center] {{\small $5$}};
\draw [->] (b3) edge (b1);
\draw (1,0.4) node[anchor=center] {{\small $a$}};
\draw [->] (b3) ..controls (3,0.5) and (5,0.5) .. (b2);
\draw (4,0.8) node[anchor=center] {{\small $b$}}; 
\draw [<-] (b3) ..controls (3,-0.5) and (5,-0.5) .. (b2);
\draw (4,-0.7) node[anchor=center] {{\small $c$}}; 
\draw [<-] (b3) ..controls (3,-2) and (5,-2) .. (b2);
\draw (4,-1.9) node[anchor=center] {{\small $d$}}; 
\end{tikzpicture}
\caption{An example of a directed graph} \label{fig:G-dir}
\end{minipage}
\hspace{0.03\linewidth}
\begin{minipage}[t]{0.45\linewidth} 
\centering 
\begin{tikzpicture}[scale=0.5, >=stealth']
\tikzstyle{w}=[circle, draw, minimum size=3, inner sep=1]
\tikzstyle{b}=[circle, draw, fill, minimum size=3, inner sep=1]
\node [b] (b1) at (0,0) {};
\draw (-0.4,0) node[anchor=center] {{\small $1$}};
\node [b] (b2) at (1.5,1.5) {};
\draw (1.5,2.1) node[anchor=center] {{\small $2$}};
\node [b] (b3) at (1.5,-1.5) {};
\draw (1.5,-2.1) node[anchor=center] {{\small $3$}};
\node [b] (b4) at (3,0) {};
\draw (3,0.6) node[anchor=center] {{\small $4$}};
\draw (b1) edge (b2);
\draw (0.7,1.2) node[anchor=center] {{\small $a$}};
\draw (b1) edge (b3);
\draw (0.7,-1.2) node[anchor=center] {{\small $b$}};
\end{tikzpicture}
\caption{An undirected graph $\G'$} \label{fig:G-undir}
\end{minipage}
\end{figure}

An undirected graph (or simply a graph) $\G$ consists of two finite sets $V(\G)$, $E(\G)$ and
a map $\displaystyle \me : E(\G) \to V(\G)^{[2]}$, where $ V(\G)^{[2]}$ is the set of 
all unordered pairs of (distinct) elements of $V(\G)$. 
For example, the graph $\G'$ shown in figure \ref{fig:G-undir} has 
$V(\G') = \{1,2,3,4\}$, $E(\G') = \{a,b\}$, $\me(a) = \{1,2\} =\{2,1\}$, 
and $\me(b) = \{1,3\} = \{3,1\}$. 

A valency of a vertex $v$ in a (directed) graph $\G$ is the number of 
edges incident to $v$. For example, the valency 
of vertex $2$ in the graph in figure \ref{fig:G-dir} is $3$ and the valency 
of vertex $1$ in the graph in figure \ref{fig:G-undir} is $2$.

In this paper, we mostly deal with directed graphs which do not have 
multiple edges with the same direction. For such graphs $\G$, we will identify 
$E(\G)$ with the corresponding subset of ordered pairs of 
vertices.  
Furthermore, if a graph $\G'$ is undirected and has no 
multiple edges then we will identify $E(\G')$ with the corresponding 
subset of unordered pairs of vertices. For example, for the graph $\G'$ 
in figure \ref{fig:G-undir}, $E(\G')$ can be identified with the set of 
unordered pairs $\{\{1,2\}, \{1,3\} \}$.

\subsection{Trees} 
\label{sec:trees-ord}
A connected graph without cycles is called a tree.
In this paper, we tacitly assume that all trees are rooted and the root vertex 
has always valency $1$. (Such trees are sometimes called 
{\it planted}). The remaining vertices of valency $1$ are 
called {\it leaves}. A vertex is called {\it internal} if it is 
neither a root nor a leaf. We always orient trees in the 
direction towards the root. Thus every internal vertex 
has at least $1$ incoming edge and exactly $1$ outgoing edge. 
An edge adjacent to a leaf is called {\it external}.

Let us recall that for every  planar tree $\bt$ the set of 
its vertices is equipped with a natural total order.
To define this total order on the set $V(\bt)$ of all vertices 
of $\bt$ we introduce the function 
\begin{equation}
\label{cN}
\cN:   V(\bt) \to V(\bt)\,.
\end{equation}
To a non-root vertex $v$ the function 
$\cN$ assigns the next vertex along the (unique) path connecting $v$ to the 
root vertex. Furthermore $\cN$ sends the root vertex to the root vertex. 

Let $v_1, v_2$ be two distinct vertices of $\bt$\,. 
If $v_1$ lies on the path which connects $v_2$ to the root 
vertex then we declare that $v_1 < v_2$.
Similarly, if $v_2$ lies on the path which connects $v_1$ to the root 
vertex then we declare that $v_2 < v_1$.
If neither of the above options realize then there exist 
numbers $k_1$ and $k_2$ such that
\begin{equation}
\label{same-vertex}
\cN^{k_1}(v_1) = \cN^{k_2}(v_2)
\end{equation}
but 
$$
\cN^{k_1-1}(v_1) \neq \cN^{k_2-1}(v_2)\,. 
$$
Since the tree $\bt$ is planar the set of $\cN^{-1} (\cN^{k_1}(v_1))$
is equipped with a total order. Furthermore, since both 
vertices $\cN^{k_1-1}(v_1) $ and  $\cN^{k_2-1}(v_2)$ belong to 
the set  $\cN^{-1}(\cN^{k_1}(v_1))$, we may compare them with 
respect to this order. We declare that, if  $\cN^{k_1-1}(v_1) < \cN^{k_2-1}(v_2)$, then $v_1 <  v_2$.
Otherwise we set $v_2 < v_1$\,. 

It is not hard to see that the resulting relation $<$ on
$V(\bt)$ is indeed a total order. 

We have an obvious bijection between the set of 
edges $E(\bt)$ of a tree $\bt$ and the subset of vertices: 
\begin{equation}
\label{no-root}
V(\bt) \setminus \{{\rm root~vertex}\}\,.
\end{equation}
This bijection assigns to a vertex $v$ in (\ref{no-root}) its 
outgoing edge. 

Thus the canonical total order on the set (\ref{no-root}) gives 
us a natural total order on the set of edges $E(\bt)$\,. 

For our purposes we also extend the total orders 
on the sets $V(\bt) \setminus \{{\rm root~vertex}\}$ 
and $E(\bt)$ to the disjoint union
\begin{equation}
\label{vertices-edges}
\Big(V(\bt) \setminus \{{\rm root~vertex}\} \Big) \sqcup  E(\bt)
\end{equation}
by declaring that a vertex is bigger than its outgoing edge. 
For example, the root edge is the minimal element 
in the set (\ref{vertices-edges}).

%
%

\subsubsection{Colored trees, labeled colored trees} 
Let $\Xi$ be a non-empty finite totally ordered set.   
We will call elements of $\Xi$ colors. 

Let $\bt$ be a tree and $v$ be an internal vertex of $\bt$\,. 
Let us denote by $E_v(\bt)$ the set of edges terminating at 
$v$\,. Recall that a planar structure on a tree $\bt$ is 
nothing but a choice of total orders on the sets $E_v(\bt)$
for all internal vertices $v$\,.

A $\Xi$-{\it colored planar tree} is a planar tree $\bt$ equipped 
with a map 
$$
c_{\bt}: E(\bt) \to \Xi 
$$
which satisfies the following condition
\begin{cond}
\label{cond:color}
The restriction of the map $c_{\bt}$ to the subset 
$E_v(\bt) \subset E(\bt)$
$$
c_{\bt} \Big|_{E_v(\bt)} ~ : ~  E_v(\bt) \to \Xi
$$
is a monotonous function for every internal vertex $v$\,.
\end{cond}
We refer to the value $c_{\bt}(e)$ of $c_{\bt}$ at $e$ as 
{\it the color of the edge} $e$\,.

Using the obvious bijection between the leaves and 
the external edges we assign to each leaf the color 
of its adjacent edge. We denote the resulting color 
function by $c_{\bt, l}$
\begin{equation}
\label{color-leaves}
c_{\bt, l} : L(\bt) \to \Xi\,,
\end{equation}    
where $L(\bt)$ is the set of leaves of $\bt$\,.

Using the function \eqref{color-leaves} we 
split the set  $L(\bt)$ into the disjoint union  
\begin{equation}
\label{leaves-union}
L(\bt) =  \bigsqcup_{\chi \in \Xi} c_{\bt, l}^{-1}(\chi)\,.
\end{equation}

We now define a {\it labeled $\Xi$-colored planar tree} as a
$\Xi$-colored planar tree $\bt$ equipped with (not necessarily 
monotonous) injective maps
\begin{equation}
\label{labeling}
\ml_{\chi} :   \{1, 2, \dots, n_{\chi} \} \to c_{\bt, l}^{-1}(\chi)\,,
\end{equation}
where $n_{\chi}$ are non-negative integers satisfying the 
obvious condition $n_{\chi} \le |c_{\bt, l}^{-1}(\chi)|$\,.
The collection of numbers $\{n_{\chi}\}_{\chi}$ is considered as 
a part of the data incorporated in a labeling of a tree.  

Leaves belonging to the union
$$
\bigsqcup_{\chi \in \Xi} \ml_{\chi}(\{1,2, \dots, n_{\chi}\})
$$
are called {\it labeled}. Furthermore,  
a vertex $x$ of a labeled colored planar tree $\bt$ is called 
{\it nodal} if it is neither a root vertex, nor a labeled leaf. 
We denote by $V_{\nod}(\bt)$ the set of all nodal vertices of 
$\bt$. Keeping in mind the canonical total order on 
the set of all vertices of $\bt$ we say things like
``the first nodal vertex'', ``the second nodal vertex'', and
``the $i$-th nodal vertex''.

\begin{example}
In this paper, the set $\Xi$ is often the two-element 
set\footnote{The notation for colors comes from string theory 
\cite{Zwiebach}. $\mo$ refers to open strings and $\mc$ refers to 
closed strings.}  $\{\mc, \mo\}$ with $\mc < \mo$\,. 
Figure \ref{fig:exam-color} gives us an example 
of a labeled $\{\mc, \mo\}$-colored (or simply $2$-colored)
planar tree. Throughout this paper edges of color $\mc$ 
are drawn solid and edges of color $\mo$ are drawn dashed. 
In addition, we use small white circles for 
nodal vertices and small black circles for labeled 
leaves and the root vertex.  
\begin{figure}[htp]
\centering
\begin{tikzpicture}[scale=0.5, >=stealth']
\tikzstyle{w}=[circle, draw, minimum size=3, inner sep=1]
\tikzstyle{b}=[circle, draw, fill, minimum size=3, inner sep=1]
\node [b] (mc3) at (1,8) {};
\draw (1,8.6) node[anchor=center] {{\small $3_{\mc}$}};
\node [b] (mc1) at (3,8) {};
\draw (3,8.6) node[anchor=center] {{\small $1_{\mc}$}};
\node [b] (mc5) at (5,8) {};
\draw (5,8.6) node[anchor=center] {{\small $5_{\mc}$}};
\node [b] (mo2) at (7,8) {};
\draw (7,8.6) node[anchor=center] {{\small $2_{\mo}$}};
\node [b] (mo3) at (9,8) {};
\draw (9,8.6) node[anchor=center] {{\small $3_{\mo}$}};
\node [b] (mc2) at (11,8) {};
\draw (11,8.6) node[anchor=center] {{\small $2_{\mc}$}};
\node [b] (mc4) at (13,8) {};
\draw (13,8.6) node[anchor=center] {{\small $4_{\mc}$}};
\node [b] (mo1) at (15,8) {};
\draw (15,8.6) node[anchor=center] {{\small $1_{\mo}$}};
\node [w] (v2) at (2,6) {};
\node [w] (v3) at (7,6) {};
\node [w] (v4) at (13,6) {};
\node [w] (v1) at (7,3) {};
\node [b] (r) at (7,1) {};
\draw (v2) edge (mc3);
\draw (v2) edge (mc1);
\draw (v3) edge (mc5);
\draw [dashed] (v3) edge (mo2);
\draw [dashed] (v3) edge (mo3);
\draw (v4) edge (mc2);
\draw (v4) edge (mc4);
\draw [dashed] (v4) edge (mo1);
\draw (v1) edge (v2);
\draw [dashed] (v1) edge (v3);
\draw [dashed] (v1) edge (v4);
\draw [dashed] (r) edge (v1);
\end{tikzpicture}
\caption{\label{fig:exam-color} Solid edges carry the color $\mc$ and 
dashed edges carry the color $\mo$}
\end{figure}
Figure \ref{fig:exam-nod0} shows an example of a labeled 
2-colored planar tree which has two unlabeled leaves 
(a.k.a. two univalent nodal vertices).
\begin{figure}[htp]
\centering
\begin{tikzpicture}[scale=0.5, >=stealth']
\tikzstyle{w}=[circle, draw, minimum size=3, inner sep=1]
\tikzstyle{b}=[circle, draw, fill, minimum size=3, inner sep=1]
\node [b] (mc3) at (1,8) {};
\draw (1,8.6) node[anchor=center] {{\small $3_{\mc}$}};
\node [b] (mc1) at (3,8) {};
\draw (3,8.6) node[anchor=center] {{\small $1_{\mc}$}};
\node [b] (mc4) at (5,8) {};
\draw (5,8.6) node[anchor=center] {{\small $4_{\mc}$}};
\node [b] (mo2) at (7,8) {};
\draw (7,8.6) node[anchor=center] {{\small $2_{\mo}$}};
\node [w] (v4) at (9,8) {};

\node [b] (mc2) at (11,8) {};
\draw (11,8.6) node[anchor=center] {{\small $2_{\mc}$}};
\node [w] (v6) at (13,8) {};

\node [b] (mo1) at (15,8) {};
\draw (15,8.6) node[anchor=center] {{\small $1_{\mo}$}};
\node [w] (v2) at (2,6) {};
\node [w] (v3) at (7,6) {};
\node [w] (v5) at (13,6) {};
\node [w] (v1) at (7,3) {};
\node [b] (r) at (7,1) {};
\draw (v2) edge (mc3);
\draw (v2) edge (mc1);
\draw (v3) edge (mc4);
\draw [dashed] (v3) edge (mo2);
\draw [dashed] (v3) edge (v4);
\draw (v5) edge (mc2);
\draw (v5) edge (v6);
\draw [dashed] (v5) edge (mo1);
\draw (v1) edge (v2);
\draw [dashed] (v1) edge (v3);
\draw [dashed] (v1) edge (v5);
\draw [dashed] (r) edge (v1);
\end{tikzpicture}
\caption{\label{fig:exam-nod0} The 4-th and the 6-th nodal vertices are univalent}
\end{figure} 
\end{example}

$\Xi$-colored planar corollas will play an important role. 
In particular, we will need a map which assigns 
a $\Xi$-colored planar corolla $\ka(\bt)$ to a labeled $\Xi$-colored planar tree $\bt$\,.
To define this map we observe that  $\Xi$-colored planar corollas are in bijection 
with the arrays $\{n_{\chi}; \chi_{root}\}_{\chi \in \Xi}$ where 
$n_{\chi}$ are  non-negative integers and $\chi_{root}$ is an element in $\Xi$\,.
More precisely, the array  $\{n_{\chi}; \chi_{root}\}_{\chi \in \Xi}$ corresponding 
to a $\Xi$-colored planar corolla $\bq$ has $\chi_{root}$ equal to the color of 
the root edge of $\bq$ and
\begin{equation}
\label{ka-array}
n_{\chi} = |c^{-1}_{\bq,l}(\chi)|\,. 
\end{equation} 
For example, the 2-colored corolla depicted on 
figure \ref{fig:corolla} corresponds to the array 
$\{ 2, 1; \mo \}$\,.
\begin{figure}[htp]
\centering
\begin{tikzpicture}[scale=0.5, >=stealth']
\tikzstyle{w}=[circle, draw, minimum size=3, inner sep=1]
\tikzstyle{b}=[circle, draw, fill, minimum size=3, inner sep=1]
\node [b] (mc1) at (0,4) {};
\node [b] (mc2) at (2,4) {};
\node [b] (mo1) at (4,4) {};
\node [w] (v) at (2,2) {};
\node [b] (r) at (2,0) {};
\draw (v) edge (mc1);
\draw  (v) edge (mc2);
\draw [dashed] (v) edge (mo1);
\draw [dashed] (r) edge (v);
\end{tikzpicture}
\caption{\label{fig:corolla}  The corolla corresponding to the array $\{2,1; \mo\}$}
\end{figure}

The degenerate array  $\{n_{\chi} = 0; \chi_{root}\}_{\chi \in \Xi}$ {\it is} allowed and 
it corresponds to the corolla depicted in figure \ref{fig:corolla-0}.
\begin{figure}[htp]
\centering
\begin{tikzpicture}[scale=0.5, >=stealth']
\tikzstyle{w}=[circle, draw, minimum size=3, inner sep=1]
\tikzstyle{b}=[circle, draw, fill, minimum size=3, inner sep=1]
\node [w] (v) at (2,3) {};
\node [b] (r) at (2,0) {};
\draw (2.2,1.5) node[anchor=center] {{\small $\chi_{root}$}};
\draw (r) edge (2,1.1) (2,1.9) edge (v);
\end{tikzpicture}
\caption{\label{fig:corolla-0}  The corolla corresponding to the degenerate array 
$\{n_{\chi} = 0; \chi_{root}\}_{\chi \in \Xi}$ }
\end{figure}

We now notice that every labeled $\Xi$-colored planar tree
$\bt$ gives us the array $\{n_{\chi}; \chi_{root}\}_{\chi \in \Xi}$ with 
$\chi_{root}$ being the color of the root edge of $\bt$ and 
$n_{\chi}$ being the numbers which enter the labeling 
\eqref{labeling} of the tree $\bt$\,.
We denote by $\ka(\bt)$ the $\Xi$-colored planar corolla 
corresponding to this array.

For example, the corolla $\ka(\bt)$
corresponding to the labeled $2$-colored 
planar tree $\bt$ in figure \ref{fig:exam-color} is shown 
in figure \ref{fig:ka-bt}. 
Similarly,  the corolla $\ka(\bt')$
corresponding to the labeled $2$-colored 
planar tree $\bt'$ in figure \ref{fig:exam-nod0} is shown 
in figure \ref{fig:ka-bt-pr}.
\begin{figure}[htp]
\centering
\begin{minipage}[t]{0.45\linewidth}
\centering
\begin{tikzpicture}[scale=0.5, >=stealth']
\tikzstyle{w}=[circle, draw, minimum size=3, inner sep=1]
\tikzstyle{b}=[circle, draw, fill, minimum size=3, inner sep=1]
\node [b] (mc1) at (0,4) {};
\node [b] (mc2) at (1,4) {};
\node [b] (mc3) at (2,4) {};
\node [b] (mc4) at (3,4) {};
\node [b] (mc5) at (4,4) {};
\node [b] (mo1) at (5,4) {};
\node [b] (mo2) at (6,4) {};
\node [b] (mo3) at (7,4) {};
\node [w] (v) at (3.5,2) {};
\node [b] (r) at (3.5,0.5) {};
\draw (v) edge (mc1);
\draw  (v) edge (mc2);
\draw  (v) edge (mc3);
\draw  (v) edge (mc4);
\draw  (v) edge (mc5);
\draw [dashed] (v) edge (mo1);
\draw [dashed] (v) edge (mo2);
\draw [dashed] (v) edge (mo3);
\draw [dashed] (r) edge (v);
\end{tikzpicture}
\caption{\label{fig:ka-bt}  The corolla $\ka(\bt)$.}
\end{minipage}
\begin{minipage}[t]{0.45\linewidth}
\centering
\begin{tikzpicture}[scale=0.5, >=stealth']
\tikzstyle{w}=[circle, draw, minimum size=3, inner sep=1]
\tikzstyle{b}=[circle, draw, fill, minimum size=3, inner sep=1]
\node [b] (mc1) at (0,4) {};
\node [b] (mc2) at (1,4) {};
\node [b] (mc3) at (2,4) {};
\node [b] (mc4) at (3,4) {};
\node [b] (mo1) at (4,4) {};
\node [b] (mo2) at (5,4) {};
\node [w] (v) at (2.5,2) {};
\node [b] (r) at (2.5,0.5) {};
\draw (v) edge (mc1);
\draw  (v) edge (mc2);
\draw  (v) edge (mc3);
\draw  (v) edge (mc4);
\draw [dashed] (v) edge (mo1);
\draw [dashed] (v) edge (mo2);
\draw [dashed] (r) edge (v);
\end{tikzpicture}
\caption{\label{fig:ka-bt-pr}  The corolla $\ka(\bt')$.}
\end{minipage}
\end{figure}

\begin{remark}
\label{rem:non-color}
It is clear that, if $\Xi$ is a one-point set, then 
$\Xi$-colored planar trees are exactly non-colored 
planar trees and $\Xi$-colored corollas are in bijection 
with non-negative integers. 
\end{remark}

\subsubsection{Groupoid of labeled (colored) planar trees}

For our purposes we need to upgrade the set 
of labeled $\Xi$-colored planar trees 
to a groupoid $\Tree^{\Xi}$\,.  Objects of 
$\Tree^{\Xi}$ are  labeled $\Xi$-colored planar trees
and morphisms are  \und{non-planar} 
isomorphisms of the corresponding 
trees compatible with labeling and coloring in 
the following sense: an isomorphism $\phi$ from $\bt$ 
to $\bt'$ sends the leaf of $\bt$ with label $i$ to the 
leaf\footnote{In particular, a nodal vertex can only be sent a nodal 
vertex.} of $\bt'$ with label $i$; furthermore, if the 
edge originating at $v \in V(\bt)$ carries the color $\chi$
then the edge originating at $\phi(v)\in V(\bt')$ carries
the same color $\chi$.

\begin{example}
\label{exam:isom-sms}
Let us denote by $\bt$ the labeled $2$-colored planar 
tree depicted in figure \ref{fig:exam-color}. The tree 
$\bt_1$ in figure \ref{fig:bt1} is isomorphic to $\bt$
while the tree $\bt_2$ in figure \ref{fig:bt2} is 
not isomorphic to $\bt$\,. 
\begin{figure}[htp]
\begin{minipage}[t]{0.45\linewidth}
\centering
\begin{tikzpicture}[scale=0.5, >=stealth']
\tikzstyle{w}=[circle, draw, minimum size=3, inner sep=1]
\tikzstyle{b}=[circle, draw, fill, minimum size=3, inner sep=1]
\node [b] (mc3) at (1,8) {};
\draw (1,8.6) node[anchor=center] {{\small $3_{\mc}$}};
\node [b] (mc1) at (3,8) {};
\draw (3,8.6) node[anchor=center] {{\small $1_{\mc}$}};
\node [b] (mc2) at (5,8) {};
\draw (5,8.6) node[anchor=center] {{\small $2_{\mc}$}};
\node [b] (mc4) at (7,8) {};
\draw (7,8.6) node[anchor=center] {{\small $4_{\mc}$}};
\node [b] (mo1) at (9,8) {};
\draw (9,8.6) node[anchor=center] {{\small $1_{\mo}$}};
\node [b] (mc5) at (11,8) {};
\draw (11,8.6) node[anchor=center] {{\small $5_{\mc}$}};
\node [b] (mo2) at (13,8) {};
\draw (13,8.6) node[anchor=center] {{\small $2_{\mo}$}};
\node [b] (mo3) at (15,8) {};
\draw (15,8.6) node[anchor=center] {{\small $3_{\mo}$}};
\node [w] (v2) at (2,6) {};
\node [w] (v3) at (7,6) {};
\node [w] (v4) at (13,6) {};
\node [w] (v1) at (7,3) {};
\node [b] (r) at (7,1) {};
\draw (v2) edge (mc3);
\draw (v2) edge (mc1);
\draw (v3) edge (mc2);
\draw (v3) edge (mc4);
\draw [dashed] (v3) edge (mo1);
\draw (v4) edge (mc5);
\draw [dashed] (v4) edge (mo2);
\draw [dashed] (v4) edge (mo3);
\draw (v1) edge (v2);
\draw [dashed] (v1) edge (v3);
\draw [dashed] (v1) edge (v4);
\draw [dashed] (r) edge (v1);
\end{tikzpicture}
\caption{\label{fig:bt1} The labeled $2$-colored tree $\bt_1$}
\end{minipage}
\hspace{0.07\linewidth}
\begin{minipage}[t]{0.45\linewidth}
\centering
\begin{tikzpicture}[scale=0.5, >=stealth']
\tikzstyle{w}=[circle, draw, minimum size=3, inner sep=1]
\tikzstyle{b}=[circle, draw, fill, minimum size=3, inner sep=1]
\node [b] (mc3) at (1,8) {};
\draw (1,8.6) node[anchor=center] {{\small $3_{\mc}$}};
\node [b] (mc1) at (3,8) {};
\draw (3,8.6) node[anchor=center] {{\small $1_{\mc}$}};
\node [b] (mc5) at (5,8) {};
\draw (5,8.6) node[anchor=center] {{\small $5_{\mc}$}};
\node [b] (mo2) at (7,8) {};
\draw (7,8.6) node[anchor=center] {{\small $2_{\mo}$}};
\node [b] (mo3) at (9,8) {};
\draw (9,8.6) node[anchor=center] {{\small $3_{\mo}$}};
\node [b] (mc2) at (11,8) {};
\draw (11,8.6) node[anchor=center] {{\small $2_{\mc}$}};
\node [b] (mc4) at (13,8) {};
\draw (13,8.6) node[anchor=center] {{\small $4_{\mc}$}};
\node [b] (mo1) at (15,8) {};
\draw (15,8.6) node[anchor=center] {{\small $1_{\mo}$}};
\node [w] (v2) at (2,6) {};
\node [w] (v3) at (7,6) {};
\node [w] (v4) at (13,6) {};
\node [w] (v1) at (7,3) {};
\node [b] (r) at (7,1) {};
\draw (v2) edge (mc3);
\draw (v2) edge (mc1);
\draw (v3) edge (mc5);
\draw [dashed] (v3) edge (mo2);
\draw [dashed] (v3) edge (mo3);
\draw (v4) edge (mc2);
\draw (v4) edge (mc4);
\draw [dashed] (v4) edge (mo1);
\draw (v1) edge (v2);
\draw (v1) edge (v3);
\draw [dashed] (v1) edge (v4);
\draw [dashed] (r) edge (v1);
\end{tikzpicture}
\caption{\label{fig:bt2}  The labeled $2$-colored tree $\bt_2$}
\end{minipage}
\end{figure} 

An object of $\Tree^{\Xi}$ may have a non-trivial automorphism. 
For example, the tree shown in figure \ref{fig:two-nod0} has a non-identity 
automorphism which switches the two univalent nodal vertices.
\begin{figure}[htp]
\centering
\begin{tikzpicture}[scale=0.5, >=stealth']
\tikzstyle{w}=[circle, draw, minimum size=3, inner sep=1]
\tikzstyle{b}=[circle, draw, fill, minimum size=3, inner sep=1]
\node [b] (mc1) at (0,4) {};
\draw (0,4.6) node[anchor=center] {{\small $1_{\mc}$}};
\node [w] (v2) at (2,4) {};
\node [w] (v3) at (4,4) {};
\node [w] (v1) at (2,2) {};
\node [b] (r) at (2,0) {};
\draw (v1) edge (mc1);
\draw [dashed] (v1) edge (v2);
\draw [dashed] (v1) edge (v3);
\draw [dashed] (r) edge (v1);
\end{tikzpicture}
\caption{\label{fig:two-nod0} This labeled tree has a non-trivial automorphism which 
switches the two univalent nodal vertices}
\end{figure}
\end{example}

We tacitly assume that each labeled colored planar tree has at least one 
nodal vertex. In other words, the degenerate labeled tree shown on 
figure \ref{fig:degen} is not considered as an object of 
$\Tree^{\Xi}$\,.
\begin{figure}[htp]
\centering
\begin{tikzpicture}[scale=0.5, >=stealth']
\tikzstyle{w}=[circle, draw, minimum size=3, inner sep=1]
\tikzstyle{b}=[circle, draw, fill, minimum size=3, inner sep=1]
\node [b] (mc1) at (0,4) {};
\draw (0,4.6) node[anchor=center] {{\small $1_{\chi}$}};
\node [b] (r) at (0,0) {};
\draw (0,2) node[anchor=center] {{\small $\chi$}};
\draw (r) edge (0,1.5)  (0, 2.5) edge (mc1);
\end{tikzpicture}
\caption{\label{fig:degen} This tree is not considered as an object of 
$\Tree^{\Xi}$}
\end{figure}

It is easy to see that, if the corollas $\ka(\bt)$ and
$\ka(\bt')$ corresponding to  labeled $\Xi$-colored planar
trees $\bt$ and $\bt'$ are different, then there are no 
morphisms between $\bt$ and $\bt'$\,.
Therefore the groupoid $\Tree^{\Xi}$ splits into the 
disjoint union
\begin{equation}
\label{Tree-Xi-union}
\Tree^{\Xi} = \bigsqcup_{\bq} \Tree^{\Xi}(\bq)\,,
\end{equation}
where  $\Tree^{\Xi}(\bq)$ is the full subcategory 
of labeled $\Xi$-colored planar trees $\bt$ satisfying 
the condition 
\begin{equation}
\label{ka-bt-bq}
\ka(\bt) = \bq
\end{equation}
and the union \eqref{Tree-Xi-union} is taken over 
all $\Xi$-colored planar corollas. 

For every $\Xi$-colored planar corolla $\bq$\,, 
we introduce the group 
\begin{equation}
\label{S-bq}
S_{\bq} = \prod_{\chi \in \Xi} S_{n_{\chi}}\,,
\end{equation}
where $n_{\chi}= c_{\bq, l}^{-1}(\chi)\,.$
This group acts in the obvious way  on 
the groupoid $\Tree^{\Xi}(\bq)$ by permuting 
labels of leaves with the same colors.  
 
We reserve the notation $\Tree^{\Xi}_2(\bq)$ for the 
full subcategory of $\Tree^{\Xi}(\bq)$ whose objects are 
labeled $\Xi$-colored planar trees with exactly two nodal 
vertices. For example, if $\Xi = \{\mc< \mo \}$ and $\bq$ is 
the corolla corresponding the array $(n,k; \chi)$
then the set 
of isomorphism classes of objects in $\Tree^{\Xi}_2(\bq)$  
is in bijection with the set
\begin{equation}
\label{double-Sh}
\bigsqcup_{0 \le p \le n} \bigsqcup_{0 \le q \le k} 
\big\{ (\si, \tau, \chi_1) ~|~ \si \in \Sh_{p, n-p},~ \tau \in \Sh_{q, k-q},~
 \chi_1 \in \{\mc, \mo\}  \big\},
\end{equation}
where $\Sh_{p,q}$ is the set of $(p,q)$-shuffles (see the beginning of Section \ref{sec:prelim}).  

Namely, if the root edge of the corolla $\bq$ carries the color $\mc$ 
then the bijection assigns to an element $(\si, \tau, \mc)$  
(resp. $(\si, \tau, \mo)$) the isomorphism class of 
the labeled $2$-colored planar 
tree depicted in figure \ref{fig:si-tau-mc}
(resp. \ref{fig:si-tau-mo}).
\begin{figure}[htp] 
\begin{minipage}[t]{0.48\linewidth}
\centering 
\begin{tikzpicture}[scale=0.5]
\tikzstyle{w}=[circle, draw, minimum size=3, inner sep=1]
\tikzstyle{b}=[circle, draw, fill, minimum size=3, inner sep=1]
\node[b] (mc1) at (0, 3) {};
\draw (0,3.6) node[anchor=center] {{\small $\si(1)$}};
\draw (1.3,3) node[anchor=center] {{\small $\dots$}};
\node[b] (mcp) at (2, 3) {};
\draw (2,3.6) node[anchor=center] {{\small $\si(p)$}};
\node[b] (mo1) at (3.5, 3) {};
\draw (3.5,3.6) node[anchor=center] {{\small $\tau(1)$}};
\draw (4.3,3) node[anchor=center] {{\small $\dots$}};
\node[b] (moq) at (5.5, 3) {};
\draw (5.5,3.6) node[anchor=center] {{\small $\tau(q)$}};
\node[w] (vv) at (3, 1.5) {};
\node[b] (mcp1) at (5.5, 1) {};
\draw (5.5,1.6) node[anchor=center] {{\small $\si(p+1)$}};
\draw (7,1) node[anchor=center] {{\small $\dots$}};
\node[b] (mcn) at (8, 1) {};
\draw (8,1.6) node[anchor=center] {{\small $\si(n)$}};
\node[b] (moq1) at (10.5, 1) {};
\draw (10.5,1.6) node[anchor=center] {{\small $\tau(q+1)$}};
\draw (11.8,1) node[anchor=center] {{\small $\dots$}};
\node[b] (mok) at (13, 1) {};
\draw (13,1.6) node[anchor=center] {{\small $\tau(k)$}};
\node[w] (v) at (7, -1) {};
\node[b] (r) at (7, -2) {};
\draw (vv) edge (mc1);
\draw (vv) edge (mcp);
\draw [dashed] (vv) edge (mo1);
\draw [dashed] (vv) edge (moq);
\draw (v) edge (vv);
\draw (v) edge (mcp1);
\draw (v) edge (mcn);
\draw  [dashed] (v) edge (moq1);
\draw [dashed]  (v) edge (mok);
\draw  (r) edge (v);
\end{tikzpicture}
\caption{Here  $\si \in \Sh_{p, n-p}$ and $\tau \in \Sh_{q, k-q}$} \label{fig:si-tau-mc}
\end{minipage}
\begin{minipage}[t]{0.48\linewidth}
\centering  
\begin{tikzpicture}[scale=0.5]
\tikzstyle{w}=[circle, draw, minimum size=3, inner sep=1]
\tikzstyle{b}=[circle, draw, fill, minimum size=3, inner sep=1]
\node[b] (mc1) at (0, 1) {};
\draw (0,1.6) node[anchor=center] {{\small $\si(1)$}};
\draw (1.3,1) node[anchor=center] {{\small $\dots$}};
\node[b] (mcp) at (2, 1) {};
\draw (2,1.6) node[anchor=center] {{\small $\si(p)$}};
\node[w] (vv) at (5, 1.5) {};
\node[b] (mcp1) at (2, 3) {};
\draw (2,3.6) node[anchor=center] {{\small $\si(p+1)$}};
\draw (3.3,3) node[anchor=center] {{\small $\dots$}};
\node[b] (mcn) at (4.5, 3) {};
\draw (4.5,3.6) node[anchor=center] {{\small $\si(n)$}};
\node[b] (mo1) at (6, 3) {};
\draw (6,3.6) node[anchor=center] {{\small $\tau(1)$}};
\draw (7,3) node[anchor=center] {{\small $\dots$}};
\node[b] (moq) at (8, 3) {};
\draw (8,3.6) node[anchor=center] {{\small $\tau(q)$}};
\node[b] (moq1) at (7.5, 1) {};
\draw (7.5,1.6) node[anchor=center] {{\small $\tau(q+1)$}};
\draw (8.7,1) node[anchor=center] {{\small $\dots$}};
\node[b] (mok) at (10, 1) {};
\draw (10,1.6) node[anchor=center] {{\small $\tau(k)$}};
\node[w] (v) at (5, -1) {};
\node[b] (r) at (5, -2) {};
\draw  (r) edge (v);
\draw  (vv) edge (mcp1);
\draw  (vv) edge (mcn);
\draw [dashed] (vv) edge (mo1);
\draw [dashed] (vv) edge (moq);
\draw (v) edge (mc1);
\draw (v) edge (mcp);
\draw [dashed] (v) edge (vv);
\draw [dashed] (v) edge (moq1);
\draw [dashed] (v) edge (mok);
\end{tikzpicture}
\caption{Here  $\si \in \Sh_{p, n-p}$ and $\tau \in \Sh_{q, k-q}$ } \label{fig:si-tau-mo}
\end{minipage}
\end{figure}
If the root edge of the corolla $\bq$ carries the color $\mo$ then
we need to replace the solid root edges of the trees depicted in figures
\ref{fig:si-tau-mc} and \ref{fig:si-tau-mo} by dashed edges.

As we mentioned above, the case when 
$\Xi$ is the one-point set corresponds to  
non-colored labeled planar trees. In this case, corollas 
can be identified with non-negative integers and the 
groupoid $\Tree$ of labeled planar trees 
splits into the disjoint union 
\begin{equation}
\label{Tree-union}
\Tree = \bigsqcup_{n \ge 0} \Tree(n)\,,
\end{equation}
where $\Tree(n)$ is the groupoid of labeled planar 
trees with exactly $n$ labeled leaves.  We refer to objects of $\Tree(n)$
as {\it $n$-labeled planar trees.}

By analogy with $\Tree^{\Xi}_2(\bq)$,  
we reserve the notation $\Tree_2(n)$ for the full sub-groupoid 
of $\Tree(n)$ whose objects are $n$-labeled planar trees with exactly 
$2$ nodal vertices. 
It is not hard to see that isomorphism classes of $\Tree_2(n)$ are 
in bijection with the union 
$$
\bigsqcup_{0 \le p \le n} \Sh_{p, n-p}
$$
where $\Sh_{p, n-p}$ denotes the set of $(p, n-p)$-shuffles in 
$S_n$\,. The bijection assigns to a $(p, n-p)$-shuffles $\tau$ the 
isomorphism class of the planar tree depicted in figure \ref{fig:shuffle}. 
Note that the ``degenerate case'' $p=0$ we get a labeled planar tree with 
one nodal vertex of valency $1$\,.
\begin{figure}[htp]
\centering 
\begin{tikzpicture}[scale=0.5]
\tikzstyle{w}=[circle, draw, minimum size=3, inner sep=1]
\tikzstyle{b}=[circle, draw, fill, minimum size=3, inner sep=1]
\node[b] (l1) at (0, 3) {};
\draw (0,3.6) node[anchor=center] {{\small $\tau(1)$}};
\draw (1.5,2.8) node[anchor=center] {{\small $\dots$}};
\node[b] (lp) at (2.5, 3) {};
\draw (2.5,3.6) node[anchor=center] {{\small $\tau(p)$}};
\node[b] (lp1) at (4, 2) {};
\draw (4,2.6) node[anchor=center] {{\small $\tau(p+1)$}};
\draw (5.1,1.8) node[anchor=center] {{\small $\dots$}};
\node[b] (ln) at (6.5, 2) {};
\draw (6.5,2.6) node[anchor=center] {{\small $\tau(n)$}};
\node[w] (v2) at (2, 2) {};
\node[w] (v1) at (4, 1) {};
\node[b] (r) at (4, 0) {};
\draw (r) edge (v1);
\draw (v1) edge (v2);
\draw (v2) edge (l1);
\draw (v2) edge (lp);
\draw (v1) edge (lp1);
\draw (v1) edge (ln);
\end{tikzpicture}
\caption{\label{fig:shuffle} Here $\tau$ is a $(p, n-p)$-shuffle}
\end{figure}

\subsubsection{Insertion of (colored) trees}

Let $\bt$ be a $\Xi$-colored tree and $x$ be a nodal vertex of $\bt$.
We denote by $\ka(x)$ the $\Xi$-colored corolla formed by the 
edges adjacent to $x$. 

If $\wt{\bt}$ be another $\Xi$-colored
labeled planar tree and its $i$-th nodal vertex $x_i$
satisfies the condition 
\begin{equation}
\label{ka-matching}
\ka(\bt) = \ka(x_i),
\end{equation}
then we can define the insertion $\bul_i$ of the tree $\bt$ into 
the $i$-th nodal vertex of $\wt{\bt}$. For the resulting planar 
tree $\wt{\bt} \,\bul_i\, \bt$ we have
\begin{equation}
\label{ka-output}
\ka( \, \wt{\bt}\, \bul_i\, \bt) = \ka(\, \wt{\bt}\,)\,.
\end{equation}

To build the tree $\wt{\bt}  \,\bul_i\, \bt$, we follow these steps: 
\begin{itemize}

\item first, we denote by $E_{i, \chi}(\,\wt{\bt}\,)$ the set of edges
of color $\chi$ terminating at the $i$-th nodal vertex of $\wt{\bt}$\,. 
Since $\wt{\bt}$ is planar, the set $E_{i, \chi}(\,\wt{\bt}\,)$ comes 
with a total order;

\item second, we erase the $i$-th nodal vertex of $\wt{\bt}$;

\item third, we identify the root edge of $\bt$ with the edge
of $\wt{\bt}$ which originated at the $i$-th nodal vertex;

\item finally, we identify external edges of $\bt$ which have {\it labeled leaves} 
with edges in the union 
$$
\bigsqcup_{\chi \in \Xi} E_{i, \chi}(\,\wt{\bt}\,)
$$
following this rule: the external edge with color $\chi$ and label $j$ 
gets identified with the $j$-th edge in the set  $E_{i, \chi}(\,\wt{\bt}\,)$\,.
In doing this, we keep the same planar structure on $\bt$, so, in general,  
branches of $\wt{\bt}$ move around.

\end{itemize}

\begin{example}
\label{ex:insertion}
Figure \ref{fig:output-ins} shows the result of the insertion 
$\wt{\bt}  \,\bul_1\, \bt$ of the labeled $2$-colored planar tree $\bt$
(depicted in figure \ref{fig:input-bt}) into the first nodal 
vertex of the labeled $2$-colored planar tree $\wt{\bt}$ 
(depicted in figure \ref{fig:wt-input}).
\begin{figure}[htp] 
\begin{minipage}[t]{0.3\linewidth}
\centering 
\begin{tikzpicture}[scale=0.5]
\tikzstyle{w}=[circle, draw, minimum size=3, inner sep=1]
\tikzstyle{b}=[circle, draw, fill, minimum size=3, inner sep=1]
\node[b] (mc3) at (0, 3) {};
\draw (0,3.6) node[anchor=center] {{\small $3_{\mc}$}};
\node[b] (mo1) at (1, 3) {};
\draw (1,3.6) node[anchor=center] {{\small $1_{\mo}$}};
\node[b] (mc1) at (2, 2) {};
\draw (2,2.6) node[anchor=center] {{\small $1_{\mc}$}};
\node[b] (mc2) at (3.5, 3) {};
\draw (3.5,3.6) node[anchor=center] {{\small $2_{\mc}$}};
\node[w] (w2) at (0.5, 2) {};
\node[w] (w1) at (2, 1) {};
\node[w] (w3) at (3.5, 2) {};
\node[b] (r) at (2, 0) {};
\draw (w2) edge (mc3);
\draw [dashed] (w2) edge (mo1);
\draw (w1) edge (mc1);
\draw (w3) edge (mc2);
\draw (w1) edge (w2);
\draw [dashed] (w1) edge (w3);
\draw [dashed] (r) edge (w1);
\end{tikzpicture} 
\caption{A labeled $2$-colored planar tree~$\wt{\bt}$} \label{fig:wt-input}
\end{minipage} 
\hspace{0.03\linewidth}
\begin{minipage}[t]{0.3\linewidth}
\centering 
\begin{tikzpicture}[scale=0.5]
\tikzstyle{w}=[circle, draw, minimum size=3, inner sep=1]
\tikzstyle{b}=[circle, draw, fill, minimum size=3, inner sep=1]
\node[b] (mc2) at (0, 3) {};
\draw (0,3.6) node[anchor=center] {{\small $2_{\mc}$}};
\node[b] (mo1) at (2, 3) {};
\draw (2,3.6) node[anchor=center] {{\small $1_{\mo}$}};
\node[w] (w2) at (1, 2) {};
\node[b] (mc1) at (3, 2) {};
\draw (3,2.6) node[anchor=center] {{\small $1_{\mc}$}};
\node[w] (w1) at (2, 1) {};
\node[b] (r) at (2, 0) {};
\draw (w2) edge (mc2);
\draw [dashed] (w2) edge (mo1);
\draw (w1) edge (w2);
\draw (w1) edge (mc1);
\draw [dashed] (r) edge (w1);
\end{tikzpicture} 
\caption{A labeled 2-colored planar tree~$\bt$} \label{fig:input-bt}
\end{minipage}
\hspace{0.03\linewidth}
\begin{minipage}[t]{0.3\linewidth}
\centering 
\begin{tikzpicture}[scale=0.5]
\tikzstyle{w}=[circle, draw, minimum size=3, inner sep=1]
\tikzstyle{b}=[circle, draw, fill, minimum size=3, inner sep=1]
\node[b] (mc1) at (0, 3) {};
\draw (0,3.6) node[anchor=center] {{\small $1_{\mc}$}};
\node[w] (w3) at (1.5, 3) {};
\node[b] (mc2) at (1.5, 4) {};
\draw (1.5,4.6) node[anchor=center] {{\small $2_{\mc}$}};
\node[w] (w2) at (1, 2) {};
\node[w] (w4) at (3, 2) {};
\node[b] (mc3) at (2.5, 3) {};
\draw (2.5,3.6) node[anchor=center] {{\small $3_{\mc}$}};
\node[b] (mo1) at (4, 3) {};
\draw (4,3.6) node[anchor=center] {{\small $1_{\mo}$}};
\node[w] (w1) at (2, 1) {};
\node[b] (r) at (2, 0) {};
\draw (w2) edge (mc1);
\draw [dashed] (w2) edge (w3);
\draw (w3) edge (mc2);
\draw (w1) edge (w2);
\draw (w1) edge (w4);
\draw (w4) edge (mc3);
\draw [dashed] (w4) edge (mo1);
\draw [dashed] (r) edge (w1);
\end{tikzpicture} 
\caption{The result of the insertion $\wt{\bt}  \,\bul_1\, \bt$} \label{fig:output-ins}
\end{minipage} 
\end{figure}

The algorithm for constructing $\wt{\bt}  \,\bul_1\, \bt$ is illustrated in figure \ref{fig:steps}

\begin{figure}[htp] 
\centering 
\begin{minipage}[t]{0.3\linewidth}
\centering 
\begin{tikzpicture}[scale=0.5]
\tikzstyle{w}=[circle, draw, minimum size=3, inner sep=1]
\tikzstyle{b}=[circle, draw, fill, minimum size=3, inner sep=1]
\node[b] (mc3) at (-0.5, 3) {};
\draw (-0.5,3.6) node[anchor=center] {{\small $3_{\mc}$}};
\node[b] (mo1) at (0.5, 3) {};
\draw (0.5,3.6) node[anchor=center] {{\small $1_{\mo}$}};
\node[b] (mc1) at (2, 2) {};
\draw (2,2.6) node[anchor=center] {{\small $1_{\mc}$}};
\node[b] (mc2) at (4, 3) {};
\draw (4,3.6) node[anchor=center] {{\small $2_{\mc}$}};
\node[w] (w2) at (0, 2) {};
\node[w] (w1) at (2, 0) {};
\node[w] (w3) at (4, 2) {};
\node[b] (r) at (2, -1) {};
\draw (w2) edge (mc3);
\draw [dashed] (w2) edge (mo1);
\draw (w1) edge (2,0.6) (2,1.3) edge (mc1);
\draw (2,1) node[anchor=center] {{\small $2^{\mc}$}};
\draw (w3) edge (mc2);
\draw (w1) edge (1.3,0.7) (0.7,1.3) edge (w2);
\draw (1,1) node[anchor=center] {{\small $1^{\mc}$}};
\draw [dashed] (w1) edge (2.7,0.7) (3.3,1.3) edge (w3);
\draw (3,1) node[anchor=center] {{\small $1^{\mo}$}};
\draw [dashed] (r) edge (w1);
\draw (6,2) node[anchor=center] {{$\longrightarrow$}};
\end{tikzpicture} 
\end{minipage} 
\hspace{0.01\linewidth}
\begin{minipage}[t]{0.3\linewidth}
\centering 
\begin{tikzpicture}[scale=0.5]
\tikzstyle{w}=[circle, draw, minimum size=3, inner sep=1]
\tikzstyle{b}=[circle, draw, fill, minimum size=3, inner sep=1]
\node[b] (mc3) at (-0.5, 3) {};
\draw (-0.5,3.6) node[anchor=center] {{\small $3_{\mc}$}};
\node[b] (mo1) at (0.5, 3) {};
\draw (0.5,3.6) node[anchor=center] {{\small $1_{\mo}$}};
\node[b] (mc1) at (2, 2) {};
\draw (2,2.6) node[anchor=center] {{\small $1_{\mc}$}};
\node[b] (mc2) at (4, 3) {};
\draw (4,3.6) node[anchor=center] {{\small $2_{\mc}$}};
\node[w] (w2) at (0, 2) {};

\node[w] (w3) at (4, 2) {};
\node[b] (r) at (2, -1.5) {};
\draw (w2) edge (mc3);
\draw [dashed] (w2) edge (mo1);
\draw (2,1) edge (mc1);
\draw (2,0.6) node[anchor=center] {{\small $2^{\mc}$}};
\draw (w3) edge (mc2);

\draw (0,1) edge (w2);
\draw (0,0.6) node[anchor=center] {{\small $1^{\mc}$}};

\draw [dashed] (4,1) edge (w3);
\draw (4,0.6) node[anchor=center] {{\small $1^{\mo}$}};
\draw [dashed] (r) edge (2,-0.5);
\draw (6,1.5) node[anchor=center] {{$\longrightarrow$}};
\end{tikzpicture} 
\end{minipage} 
\hspace{0.01\linewidth}
\begin{minipage}[t]{0.3\linewidth}
\centering 
\begin{tikzpicture}[scale=0.5]
\tikzstyle{w}=[circle, draw, minimum size=3, inner sep=1]
\tikzstyle{b}=[circle, draw, fill, minimum size=3, inner sep=1]
\node[b] (mc3) at (-0.5, 3) {};
\draw (-0.5,3.6) node[anchor=center] {{\small $3_{\mc}$}};
\node[b] (mo1) at (0.5, 3) {};
\draw (0.5,3.6) node[anchor=center] {{\small $1_{\mo}$}};
\node[b] (mc1) at (2, 2) {};
\draw (2,2.6) node[anchor=center] {{\small $1_{\mc}$}};
\node[b] (mc2) at (4, 3) {};
\draw (4,3.6) node[anchor=center] {{\small $2_{\mc}$}};
\node[w] (w2) at (0, 2) {};
\node[w] (w3) at (4, 2) {};
\draw (w2) edge (mc3);
\draw [dashed] (w2) edge (mo1);
\draw (2,1) edge (mc1);
\draw (2,0.6) node[anchor=center] {{\small $2^{\mc}$}};
\draw (w3) edge (mc2);
\draw (0,1) edge (w2);
\draw (0,0.6) node[anchor=center] {{\small $1^{\mc}$}};
\draw [dashed] (4,1) edge (w3);
\draw (4,0.6) node[anchor=center] {{\small $1^{\mo}$}};
\node[b] (mc2) at (0.5, -1) {};
\draw (0.5,-0.4) node[anchor=center] {{\small $2_{\mc}$}};
\node[b] (mo1) at (2.5, -1) {};
\draw (2.5,-0.4) node[anchor=center] {{\small $1_{\mo}$}};
\node[w] (w2) at (1.5, -2) {};
\node[b] (mc1) at (3.5, -2) {};
\draw (3.5,-1.4) node[anchor=center] {{\small $1_{\mc}$}};
\node[w] (w1) at (2.5, -3) {};
\node[b] (r) at (2.5, -4) {};
\draw (w2) edge (mc2);
\draw [dashed] (w2) edge (mo1);
\draw (w1) edge (w2);
\draw (w1) edge (mc1);
\draw [dashed] (r) edge (w1);
\end{tikzpicture} 
\end{minipage}
\vspace{0.5cm}

\begin{minipage}[t]{\linewidth}
\centering 
\begin{tikzpicture}[scale=0.5]
\tikzstyle{w}=[circle, draw, minimum size=3, inner sep=1]
\tikzstyle{b}=[circle, draw, fill, minimum size=3, inner sep=1]
\draw (-2,2) node[anchor=center] {{$\longrightarrow$}};
\node[b] (mc1) at (0, 3) {};
\draw (0,3.6) node[anchor=center] {{\small $1_{\mc}$}};
\node[w] (w3) at (1.5, 3) {};
\node[b] (mc2) at (1.5, 4) {};
\draw (1.5,4.6) node[anchor=center] {{\small $2_{\mc}$}};
\node[w] (w2) at (1, 2) {};
\node[w] (w4) at (3, 2) {};
\node[b] (mc3) at (2.5, 3) {};
\draw (2.5,3.6) node[anchor=center] {{\small $3_{\mc}$}};
\node[b] (mo1) at (4, 3) {};
\draw (4,3.6) node[anchor=center] {{\small $1_{\mo}$}};
\node[w] (w1) at (2, 1) {};
\node[b] (r) at (2, 0) {};
\draw (w2) edge (mc1);
\draw [dashed] (w2) edge (w3);
\draw (w3) edge (mc2);
\draw (w1) edge (w2);
\draw (w1) edge (w4);
\draw (w4) edge (mc3);
\draw [dashed] (w4) edge (mo1);
\draw [dashed] (r) edge (w1);
\end{tikzpicture} 
\end{minipage} 
\caption{Algorithm for constructing  $\wt{\bt}  \,\bul_1\, \bt$} \label{fig:steps}
\end{figure}
\end{example}

\subsection{Colored operads and their dual versions}

\subsubsection{Colored collections}

Let us recall that a $\Xi$-colored collection in a symmetric monoidal 
category $\mC$ is given by the data: 

--- For each $\Xi$-colored planar corolla $\bq$
we have an object 
$$
P(\bq) \in \mC
$$ 
equipped with a left action of the group $S_{\bq}$
(\ref{S-bq}).

Morphisms of $\Xi$-colored collections are defined 
in the obvious way.  

In the case $\Xi = \{\mc <\mo\} $ we will denote
the object corresponding to a corolla $\bq$ by 
$$
P(n,k)^{\chi}\,, 
$$
where $n= |c^{-1}_{\bq,l}(\mc)|$,  $k= |c^{-1}_{\bq,l}(\mo)|$,
and $\chi$ is the color of the root edge.

Given a $\Xi$-colored collection $P$ in $\mC$
we introduce a covariant functor
\begin{equation}
\label{und-P}
\und{P} :  \Tree^{\Xi} \to  \mC
\end{equation}
from the groupoid $\Tree^{\Xi}$ of labeled $\Xi$-colored 
planar trees to $\mC$\,.

To a labeled $\Xi$-colored planar tree $\bt$, the functor 
$\und{P}$ assigns  the object
\begin{equation}
\label{und-P-bt}
\und{P}(\bt) =  \bigotimes_{x \in V_{\nod}(\bt)}  P(\ka(x))
\end{equation}
where $V_{\nod}(\bt)$ is the set of all nodal vertices
of $\bt$, $\ka(x)$ is the $\Xi$-colored planar corolla formed 
by all edges of $\bt$ adjacent to $x$, and
the order of the factors agrees with the total order on 
the set  $V_{\nod}(\bt)$.  

To define $\und{P}$ on the level of morphisms, we 
use the action of the group \eqref{S-bq} on $P(\bq)$ and
the braiding of the symmetric monoidal category in the 
obvious way. 
For example, let $\bt$ and $\bt_1$ be $2$-colored trees 
depicted in figures \ref{fig:exam-color} and \ref{fig:bt1}, 
respectively. For these trees we have 
$$
\und{P}(\bt) = P(1,2)^{\mo} \otimes P(2,0)^{\mc} \otimes 
P(1,2)^{\mo} \otimes P(2,1)^{\mo}\,, 
$$
$$
\und{P}(\bt_1) = P(1,2)^{\mo} \otimes P(2,0)^{\mc} \otimes 
 P(2,1)^{\mo} \otimes P(1,2)^{\mo}\,. 
$$
The functor $\und{P}$ sends the unique morphism $\phi: \bt \to \bt_1$
to  
$$
\und{P}(\phi) = (\id, \si_{12}) \otimes 1 \otimes \beta,
$$ 
where $(\id, \si_{12})$ is the non-identity element of the group 
$S_1 \times S_2$ and $\beta$ is the braiding 
$$
\beta: P(1,2)^{\mo} \otimes P(2,1)^{\mo}  \to
P(2,1)^{\mo} \otimes P(1,2)^{\mo} \,.
$$

\subsubsection{Colored (pseudo)operads}
Let $\bq$ be a $\Xi$-colored planar corolla. 
We say that the corolla $\bq$ is {\it naturally labeled} if the map 
$$
\ml_{\chi} :  \{1,2, \dots, n_{\chi} \}  \to c^{-1}_{\bt, l} (\chi) 
$$
is a monotonous bijection for every $\chi\in \Xi$\,. 
An example of a naturally labeled corolla is depicted 
in figure \ref{fig:nat-labeled}. 
\begin{figure}[htp]
\centering
\begin{tikzpicture}[scale=0.5, >=stealth']
\tikzstyle{w}=[circle, draw, minimum size=3, inner sep=1]
\tikzstyle{b}=[circle, draw, fill, minimum size=3, inner sep=1]
\node [b] (mc1) at (0,4) {};
\draw (0,4.6) node[anchor=center] {{\small $1_{\mc}$}};
\node [b] (mc2) at (1.5,4) {};
\draw (1.5,4.6) node[anchor=center] {{\small $2_{\mc}$}};
\node [b] (mc3) at (3,4) {};
\draw (3,4.6) node[anchor=center] {{\small $3_{\mc}$}};
\node [b] (mc4) at (4.5,4) {};
\draw (4.5,4.6) node[anchor=center] {{\small $4_{\mc}$}};
\node [b] (mc5) at (6,4) {};
\draw (6,4.6) node[anchor=center] {{\small $5_{\mc}$}};
\node [b] (mo1) at (7.5,4) {};
\draw (7.5,4.6) node[anchor=center] {{\small $1_{\mo}$}};
\node [b] (mo2) at (9,4) {};
\draw (9,4.6) node[anchor=center] {{\small $2_{\mo}$}};
\node [b] (mo3) at (10.5,4) {};
\draw (10.5,4.6) node[anchor=center] {{\small $3_{\mo}$}};
\node [w] (v) at (5,2) {};
\node [b] (r) at (5,0) {};
\draw (v) edge (mc1);
\draw  (v) edge (mc2);
\draw  (v) edge (mc3);
\draw  (v) edge (mc4);
\draw  (v) edge (mc5);
\draw [dashed] (v) edge (mo1);
\draw [dashed] (v) edge (mo2);
\draw [dashed] (v) edge (mo3);
\draw [dashed] (r) edge (v);
\end{tikzpicture}
\caption{\label{fig:nat-labeled}  An example of a naturally labeled $2$-colored corolla}
\end{figure}
The degenerate corolla shown in figure \ref{fig:corolla-0} is considered as 
a naturally labeled corolla by convention. 

For our purposes it is convenient to use the following 
definition of a colored pseudo-operad.
\begin{defi}
\label{dfn:psi-oper}
A  $\Xi$-colored pseudo-operad is a $\Xi$-colored 
collection $P$ equipped with multiplication maps 
\begin{equation}
\label{mu-bt}
\mu_{\bt} : \und{P}(\bt) \to P(\ka(\bt))
\end{equation}
defined for every labeled $\Xi$-colored planar trees $\bt$
and subject to the following axioms: 
\begin{itemize}

\item If $\bq$ is a naturally labeled $\Xi$-colored planar corolla then 
\begin{equation}
\label{corolla-iden}
\mu_{\bq} = \id_{P(\bq)}\,.
\end{equation}

\item The operation $\mu_{\bt}$ is $S_{\ka(\bt)}$-equivariant. 
Namely, for every labeled $\Xi$-colored planar tree $\bt$ we have  
\begin{equation}
\label{mu-S-equiv}
\mu_{\si(\bt)} = \si \circ \mu_{\bt}\,, \qquad 
\forall \quad \si \in S_{\ka(\bt)}\,. 
\end{equation}

\item For every morphism $\la : \bt \to \bt'$ in $\Tree^{\Xi}$ we have
\begin{equation}
\label{la-equiv}
 \mu_{\bt'} \circ  \und{P}(\la) =  \mu_{\bt}\,.
\end{equation}

\item To formulate the associativity axiom, we 
consider a triple $(\wt{\bt}, x, \bt)$ where $\wt{\bt}$ 
is a labeled $\Xi$-colored planar tree, $x$ is 
the $i$-th nodal vertex of $\wt{\bt}$, and $\bt$ is 
a labeled $\Xi$-colored planar tree 
 such that $\ka(\bt)= \ka(x)$\,.
The associativity axiom states that for each such 
triple we have 
\begin{equation}
\label{assoc}
\mu_{\wt{\bt}} \circ  (1  \otimes \dots \otimes 1 \otimes 
\underbrace{\mu_{\bt}}_{i\textrm{-th spot}}
  \otimes 1
\otimes \dots \otimes 1)  \circ \beta_{\,\wt{\bt}, x, \bt}  = \mu_{\,\wt{\bt}  \,\bul_i\, \bt} 
\end{equation}
where $\wt{\bt}  \,\bul_i\, \bt$ is the tree obtained 
by inserting $\bt$ into the $i$-th vertex of $\wt{\bt}$ and
$\beta_{\,\wt{\bt}, x, \bt}$ is the isomorphism in $\mC$ which 
is ``responsible for putting tensor factors in the correct order''. 

\end{itemize}

Morphisms of pseudo-operads are defined in 
the obvious way. 
\end{defi}

Let $\bq_1$ and $\bq_2$ be two naturally labeled 
$\Xi$-colored planar corollas such that the root edge of $\bq_2$
carries the color $\chi$\,. Let $m_{\chi'}$ (resp. $n_{\chi'}$) be the 
number of external edges (if any) of $\bq_1$ (resp. $\bq_2$) 
of color $\chi' \in \Xi$

Given a color $\chi \in \Xi$ for which $m_{\chi} > 0$ and 
$1 \ge i \ge m_{\chi}$\,, we denote by 
$\bt_{i, \chi}$ the labeled $\Xi$-colored planar tree which 
is obtained from $\bq_1$ and $\bq_2$ in two steps. 
First, we glue $\bq_2$ with $\bq_1$ by
identifying the root edge of $\bq_2$ with the external 
edge of $\bq_1$ which carries the color $\chi$ and label $i$\,.
Second, we change the labels on the leaves of the resulting $\Xi$-colored planar 
tree following these steps: 
\begin{itemize}

\item if $\chi' < \chi$ then we shift the labels on leaves in $c^{-1}_{\bq_2, l}(\chi')$ up 
by $m_{\chi'}$; 

\item we shift the labels on leaves in $c^{-1}_{\bq_2, l}(\chi)$ up by $i-1$ 
and we shift the labels on leaves in $c^{-1}_{\bq_1, l}(\chi)$ which are $> i$
up by $n_{\chi} -1$;

\item   if $\chi' > \chi$ then we shift the labels on leaves in $c^{-1}_{\bq_1, l}(\chi')$ up 
by $n_{\chi'}$\,. 

\end{itemize}

For example, if $\bq_1$ and $\bq_2$ is the 
$2$-colored corollas depicted in figures \ref{fig:bq1} and \ref{fig:bq2}, 
respectively, then $\bt_{2,\mo}$ is the tree depicted in figure 
\ref{fig:bt-2-mo}.
\begin{figure}[htp] 
\begin{minipage}[t]{0.3\linewidth}
\centering 
\begin{tikzpicture}[scale=0.5]
\tikzstyle{w}=[circle, draw, minimum size=3, inner sep=1]
\tikzstyle{b}=[circle, draw, fill, minimum size=3, inner sep=1]
\node[b] (r) at (0, 0) {};
\node[w] (v) at (0, 1) {};
\node[b] (mc1) at (-1.5, 2) {};
\draw (-1.5,2.5) node[anchor=center] {{\small $1_{\mc}$}};
\node[b] (mo1) at (-0.5, 2) {};
\draw (-0.5,2.5) node[anchor=center] {{\small $1_{\mo}$}};
\node[b] (mo2) at (0.5, 2) {};
\draw (0.5,2.5) node[anchor=center] {{\small $2_{\mo}$}};
\node[b] (mo3) at (1.5, 2) {};
\draw (1.5,2.5) node[anchor=center] {{\small $3_{\mo}$}};
\draw [dashed] (r) edge (v);
\draw (v) edge (mc1);
\draw  [dashed] (v) edge (mo1);
\draw   [dashed] (v) edge (mo2);
\draw   [dashed] (v) edge (mo3);
\end{tikzpicture}
\caption{A 2-colored corola $\bq_1$} \label{fig:bq1}
\end{minipage}
\hspace{0.2cm}
\begin{minipage}[t]{0.3\linewidth}
\centering 
\begin{tikzpicture}[scale=0.5]
\tikzstyle{w}=[circle, draw, minimum size=3, inner sep=1]
\tikzstyle{b}=[circle, draw, fill, minimum size=3, inner sep=1]
\node[b] (r) at (0, 0) {};
\node[w] (v) at (0, 1) {};
\node[b] (mc1) at (-1, 2) {};
\draw (-1,2.5) node[anchor=center] {{\small $1_{\mc}$}};
\node[b] (mo1) at (0, 2) {};
\draw (0,2.5) node[anchor=center] {{\small $1_{\mo}$}};
\node[b] (mo2) at (1, 2) {};
\draw (1,2.5) node[anchor=center] {{\small $2_{\mo}$}};
\draw [dashed] (r) edge (v);
\draw (v) edge (mc1);
\draw [dashed]  (v) edge (mo1);
\draw [dashed]  (v) edge (mo2);
\end{tikzpicture}
\caption{A 2-colored corola $\bq_2$} \label{fig:bq2}
\end{minipage}
\hspace{0.2cm}
\begin{minipage}[t]{0.35\linewidth}
\centering 
\begin{tikzpicture}[scale=0.5]
\tikzstyle{w}=[circle, draw, minimum size=3, inner sep=1]
\tikzstyle{b}=[circle, draw, fill, minimum size=3, inner sep=1]
\node[b] (r) at (0, 0) {};
\node[w] (v) at (0, 1) {};
\node[b] (mc1) at (-2, 2) {};
\draw (-2,2.5) node[anchor=center] {{\small $1_{\mc}$}};
\node[b] (mo1) at (-1, 2) {};
\draw (-1,2.5) node[anchor=center] {{\small $1_{\mo}$}};
\node[w] (vv) at (0, 2) {};
\node[b] (mc2) at (-1, 3.5) {};
\draw (-1,4) node[anchor=center] {{\small $2_{\mc}$}};
\node[b] (mo2) at (0, 3.5) {};
\draw (0,4) node[anchor=center] {{\small $2_{\mo}$}};
\node[b] (mo3) at (1, 3.5) {};
\draw (1,4) node[anchor=center] {{\small $3_{\mo}$}};
\node[b] (mo4) at (1, 2) {};
\draw (1,2.5) node[anchor=center] {{\small $4_{\mo}$}};
\draw [dashed] (r) edge (v);
\draw (v) edge (mc1);
\draw [dashed]  (v) edge (mo1);
\draw [dashed]  (v) edge (vv);
\draw [dashed]  (v) edge (mo4);
\draw (vv) edge (mc2);
\draw [dashed]  (vv) edge (mo2);
\draw [dashed]  (vv) edge (mo3);
\end{tikzpicture}
\caption{The labeled $2$-colored tree $\bt_{2, \mo}$} \label{fig:bt-2-mo}
\end{minipage}
\end{figure}
Although the tree $\bt_{i, \chi}$ depends on the corollas $\bq_1$ and 
$\bq_2$, we suppress $\bq_1$ and $\bq_2$ from the 
notation.

To introduce a structure of a pseudo-operad on
a collection $P$ it suffices to specify the 
multiplications
\begin{equation}
\label{mu-t-i-chi}
\mu_{\bt_{i, \chi}} :  P( \bq_1) \otimes P(\bq_2) \to P (\ka(\bt_{i, \chi}))  
\end{equation}
for all tuples $(\bq_1, \bq_2, i, \chi)$\,. 
All the remaining  multiplications \eqref{mu-bt} can be deduced 
from \eqref{mu-t-i-chi} using axioms of pseudo-operad.  

The operations \eqref{mu-t-i-chi} are called  {\it elementary insertions}
and we will use for them the special notation 
$ \circ_{i, \chi} $\,. Namely, if $v \in P(\bq_1)$ and 
$w\in P(\bq_2)$ then 
\begin{equation}
\label{circ-i-chi}
v\, \circ_{i, \chi} \,w \,: = \, \mu_{\bt_{i, \chi}} (v,w)\,.
\end{equation}

Let $\chi \in \Xi$ and let $\bu_{\chi}$ be the labeled tree 
with exactly two edges: the root edge and the external
edge, both carrying the color $\chi$: 
\begin{equation}
\label{u-chi} 
\begin{tikzpicture}[scale=0.5]
\tikzstyle{w}=[circle, draw, minimum size=3, inner sep=1]
\tikzstyle{b}=[circle, draw, fill, minimum size=3, inner sep=1] 
\draw (-2,0) node[anchor=center] {{$ \bu_{\chi}  = $}};
\node[b] (r) at (0, -1) {};
\node[w] (v1) at (0, 0) {};
\node[b] (v2) at (0, 1) {};
\draw (0,1.5) node[anchor=center] {{\small $1$}};
\draw (r) edge (v1);
\draw (v1) edge (v2);
\end{tikzpicture}
\end{equation}

We say that 
\begin{defi}
\label{dfn:oper}
$P$ is a $\Xi$-colored operad if $P$ is a 
$\Xi$-colored pseudo-operad with chosen maps
(unit maps)
\begin{equation}
\label{units}
I_{\chi} : \bbK \to  P(\bu_{\chi})
\end{equation}
such that the compositions
\begin{equation}
\label{unit-axioms}
\begin{array}{c}
P(\bq) \cong  P(\bq)\otimes \bbK \stackrel{1 \otimes I_{\chi} \phantom{aa}}{~\longrightarrow~}
 P(\bq) \otimes P(\bu_{\chi})  
   \stackrel{\mu_{\bt_{i, \chi}}  \phantom{aa}}{~\longrightarrow~} P(\bq)  \\[0.3cm]
P(\bq) \cong \bbK \otimes P(\bq) \stackrel{I_{\chi} \otimes 1 \phantom{aa}}{~\longrightarrow~}
 P(\bu_{\chi}) \otimes P(\bq)  
   \stackrel{\mu_{\bt_{i, \chi}}  \phantom{aa}}{~\longrightarrow~} P(\bq)   
\end{array}   
\end{equation}
coincide with the identity map on $P(\bq)$ whenever 
they make sense. Morphisms of 
$\Xi$-colored operads are 
defined in the obvious way.  
\end{defi}
\begin{remark}
\label{rem:compare-BM}
For a conventional definition of colored operads 
we refer the reader to paper  \cite{BM-colors} by 
C. Berger and I. Moerdijk. Due to the observation
made in \cite[Remark 1.3]{BM-colors} the definition given 
here is equivalent to the conventional one.  
\end{remark}

\begin{example}
\label{ex:end}
Let $\Xi = \{\mc, \mo \}$ and $(\cV, \cA)$ be a pair 
of cochain complexes. The $2$-colored collection 
$\End_{\cV, \cA}$ with 
\begin{equation}
\label{End}
\End_{\cV, \cA}(n,k)^{\mc} = \Hom(\cV^{\otimes\, n} \otimes \cA^{\otimes \, k} , \cV)\,, 
\qquad
\End_{\cV, \cA}(n,k)^{\mo} = \Hom(\cV^{\otimes\, n} \otimes \cA^{\otimes \, k} , \cA) 
\end{equation}
is equipped with the obvious structure of a $2$-colored operad. 
$\End_{\cV, \cA}$ is called the endomorphism operad of 
the pair $(\cV, \cA)$\,.  This example can be obviously generalized 
to an arbitrary set of colors $\Xi$.
\end{example}
Example \ref{ex:end} plays an important role because an algebra
over a $\Xi$-colored operad $P$ is defined as a family 
$\{V_{\chi}\}_{\chi\in \Xi}$ of objects in $\mC$ with an operad morphism from 
$P$ to $\End_{\{V_{\chi}\}_{\chi \in \Xi}}$\,.  

\subsubsection{Augmentation of colored operads}

The $\Xi$-colored collection 
\begin{equation}
\label{ast}
\ast(\bq) = \begin{cases}
   \bbK \qquad {\rm if} \quad \bq = \bu_{\chi} \textrm{ for some }\chi \in \Xi\,,    \\
   \bfzero \qquad \qquad {\rm otherwise}
\end{cases}
\end{equation}
is equipped with a unique structure of a $\Xi$-colored operad. 
It is easy to see that $\ast$ is the initial object in
the category of $\Xi$-colored operads. 

A $\Xi$-colored operad $P$ is called augmented if $P$
comes with an operad morphism 
$$
\ve:  P \to  \ast.
$$
For every augmented operad $P$ the kernel of 
the map $P \to \ast$ is naturally a pseudo-operad.
We denote this pseudo-operad by $P_{\circ}$\,.
 
It is not hard to see that the assignment 
$$
P \leadsto P_{\circ}
$$
extends to a functor. 
According to\footnote{Although, in paper \cite{Markl} the author considers only 
non-colored operads, the line of arguments can be easily extended to the colored 
setting.} \cite[Proposition 21]{Markl} this functor gives 
us an equivalence between the category of augmented (colored) operads and 
the category of (colored) pseudo-operads.

\subsubsection{Colored (pseudo)cooperads}
Reversing all arrows in Definition \ref{dfn:psi-oper} we get 
\begin{defi}
\label{dfn:psi-cooper}
A  $\Xi$-colored pseudo-cooperad is a $\Xi$-colored 
collection $Q$ equipped with comultiplication maps 
\begin{equation}
\label{D-bt}
\D_{\bt} :  Q(\ka(\bt))  \to \und{Q}(\bt)
\end{equation}
defined for every labeled $\Xi$-colored planar trees $\bt$
and subject to the following axioms: 
\begin{itemize}

\item If $\bq$ is a naturally labeled $\Xi$-colored planar corolla then 
\begin{equation}
\label{corolla-iden-co}
\D_{\bq} = \id_{Q(\bq)}\,.
\end{equation}

\item The operation $\D_{\bt}$ is $S_{\ka(\bt)}$-equivariant. 
Namely, for every labeled $\Xi$-colored planar tree $\bt$ we have  
\begin{equation}
\label{D-S-equiv}
\D_{\si(\bt)} \circ \si = \D_{\bt} \,, \qquad 
\forall \quad \si \in S_{\ka(\bt)}\,. 
\end{equation}

\item For every morphism $\la : \bt \to \bt'$ in $\Tree^{\Xi}$ we have
\begin{equation}
\label{la-equiv-co}
 \D_{\bt'}  =  \und{Q}(\la)  \circ \D_{\bt}\,.
\end{equation}

\item To formulate the coassociativity axiom we 
consider a triple $(\wt{\bt}, x, \bt)$ where $\wt{\bt}$ 
is a labeled $\Xi$-colored planar tree, $x$ is 
the $i$-th nodal vertex of $\wt{\bt}$, and $\bt$ is 
a labeled $\Xi$-colored planar tree 
 such that $\ka(\bt)= \ka(x)$\,.
The coassociativity axiom states that for each such 
triple we have 
\begin{equation}
\label{coassoc}
 (1  \otimes \dots \otimes 1 \otimes 
\underbrace{\D_{\bt}}_{i\textrm{-th spot}}
  \otimes 1
\otimes \dots \otimes 1) 
\D_{\wt{\bt}} = \D_{\,\wt{\bt}  \,\bul_i\, \bt}   \circ \beta_{\,\wt{\bt}, x, \bt} 
\end{equation}
where $\wt{\bt}  \,\bul_i\, \bt$ is the tree obtained 
by inserting $\bt$ into the $i$-th vertex of $\wt{\bt}$ and
$\beta_{\,\wt{\bt}, x, \bt}$ is the isomorphism in $\mC$ which 
is ``responsible for putting tensor factors in the correct order''. 

\end{itemize}
Morphisms of pseudo-cooperads are defined in 
the obvious way. 
\end{defi}

Similarly, reversing arrows in \eqref{units}, \eqref{unit-axioms}, 
and Definition \ref{dfn:oper} we get the notion of counit and 
the definition of a $\Xi$-colored cooperad.  

The $\Xi$-colored collection \eqref{ast} carries a unique 
structure of a  $\Xi$-colored cooperad. Furthermore, $\ast$ is the terminal object 
in the category of  $\Xi$-colored cooperads. 

Dually to augmentation, we define a coaugmentation on a (colored) cooperad $Q$ 
as a cooperad morphism
$$
\ve' : \ast \to Q\,.
$$ 

For every coaugmented (colored) cooperad $Q$ the cokernel 
of coaugmentation naturally forms a  (colored) pseudo-cooperad.  
We denote this pseudo-cooperad by $Q_{\circ}$\,.

Just as for (colored) operads, the assignment 
$$
Q \leadsto Q_{\circ}
$$
extends to a functor which establishes an equivalence between 
the category of coaugmented   (colored) cooperads and the category of
 (colored)  pseudo-cooperads. 

\subsection{The convolution Lie algebra}
\label{sec:Conv}

Let $\cC$ (resp. $\cO$) be a $\Xi$-colored pseudo-cooperad
(resp. $\Xi$-colored pseudo-operad) in $\Ch_{\bbK}$

We consider the following cochain complex 
\begin{equation}
\label{Conv-C-O}
\Conv(\cC, \cO) := 
\prod_{\bq} \Hom_{S_{\bq}} (\cC(\bq), \cO(\bq))\,,
\end{equation}
where the product is taken over all $\Xi$-colored planar 
corollas and the differential comes solely from the ones 
on $\cC$ and $\cO$\,.

Let us denote by $\Isom^{\Xi}_2(\bq)$ the set 
of isomorphism classes in\footnote{Recall that $\Tree^{\Xi}_2(\bq)$ is 
the full subcategory of $\Tree^{\Xi}(\bq)$ whose objects are 
labeled $\Xi$-colored 
planar trees $\bt$ with exactly two nodal vertices.} 
$\Tree^{\Xi}_2(\bq)$\,. Let us choose for every class 
$z \in \Isom^{\Xi}_2(\bq)$ its representative $\bt_z$. 

Using the trees $\bt_z$ we equip the 
complex \eqref{Conv-C-O} with 
the following binary operation  
\begin{equation}
\label{bul}
f \bullet g (X)  =  \sum_{z\in \Isom^{\Xi}_2(\bq) } \mu_{\bt_z} \,
\big( f \otimes g (\D_{\bt_z} (X)) \big)
\end{equation}
where $X\in \cC(\bq)$\,.  The axioms of pseudo-(co)operad imply 
that $\bul$ is a well-defined operation. Namely, the right hand side of \eqref{bul} 
does not depend on the choice of representatives $\bt_z$ and 
$f \bullet g$ is $S_{\bq}$-equivariant. 

We claim that 
\begin{prop}
\label{prop:pre-Lie}
The operation $\bullet$ \eqref{bul} equips $\Conv(\cC, \cO)$ 
with a pre-Lie algebra structure. In other words, 
\begin{equation}
\label{pre-Lie}
(f \bullet g ) \bullet h - f \bullet (g \bullet h) 
= (-1)^{|g| |h|} (f \bullet h) \bullet g  -  (-1)^{|g| |h|} f \bullet (h \bullet g)\,, 
\end{equation}
for all homogeneous vectors  
$f,g,h \in \Conv(\cC, \cO)$\,.
\end{prop}
\begin{proof}
This statement was proved in the more general setting 
(for PROPs) in \cite[Section 2.2]{MVnado} by B. Vallette and S. Merkulov.
For non-colored (co)operads, a detailed proof 
can be found in \cite[Section 4]{notes}.
\end{proof}

Proposition \ref{prop:pre-Lie} implies that the 
operation 
\begin{equation}
\label{Conv-brack}
[f,g]  = f \bullet g - (-1)^{|f||g|} g \bullet f
\end{equation}
satisfies the Jacobi identity. Thus,  $\Conv(\cC, \cO)$
is a Lie algebra in the category $\Ch_{\bbK}$\,.
Following  \cite{MVnado}, we call  $\Conv(\cC, \cO)$
the {\it convolution Lie algebra} of a pair $(\cC, \cO)$\,.

Using ``arity'' we can equip the convolution Lie algebra $\Conv(\cC, \cO)$ 
with the natural descending filtration 
$$
\Conv(\cC, \cO) =  \cF_{-1}\, \Conv(\cC, \cO) \supset \cF_0\, \Conv(\cC, \cO) \supset \cF_1\, \Conv(\cC, \cO)
\supset  \dots\,,
$$
where
\begin{equation}
\label{Conv-filtr}
\cF_m\, \Conv(\cC, \cO)=
\end{equation} 
$$
 \big\{ f \in \Conv(\cC, \cO) \quad \big| \quad
f \Big|_{\cC(\bq)} \,=\, 0 \quad \forall~~
\textrm{corollas} ~~ \bq \quad \textrm{satisfying} ~ |\bq| \le m \big\}
$$
and $|\bq|$ is the total number of incoming edges of the corolla $\bq$\,.

It is easy to see that this filtration is compatible with
the Lie bracket and  $\Conv(\cC, \cO)$ is complete 
with respect to this filtration.  Namely,
\begin{equation}
\label{complete}
\Conv(\cC, \cO) = \lim_{m} \Conv(\cC, \cO)
~\Big/~   \cF_m\, \Conv(\cC, \cO)\,.
\end{equation}

In fact, we may introduce an additional descending 
filtration  $\cF^{\chi}_{\bul}$ on the convolution Lie algebra $\Conv(\cC, \cO)$ 
for each color $\chi \in \Xi$:
$$
\Conv(\cC, \cO) =  \cF^{\chi}_{-1}\, \Conv(\cC, \cO) \supset 
 \cF^{\chi}_0\, \Conv(\cC, \cO) \supset \cF^{\chi}_1\, \Conv(\cC, \cO)
 \supset \dots\,,
$$
where
\begin{equation}
\label{Conv-filtr-chi}
\cF^{\chi}_m\, \Conv(\cC, \cO)
\end{equation} 
consists of vectors  $f \in \Conv(\cC, \cO)$ satisfying these two conditions: 
\begin{equation}
\label{color-is-chi}
f \Big|_{\cC(\bq)} \,=\, 0  \quad
\textrm{ if the color of the root edge of }\bq \textrm{ is }\chi 
\textbf{ and } |c^{-1}_{\bq, l} (\chi)| \le m, 
\end{equation}
and 
\begin{equation}
\label{color-is-not-chi}
f \Big|_{\cC(\bq)} \,=\, 0  \quad 
\textrm{ if the color of the root edge of }\bq \textrm{ is not }\chi 
\textbf{ and } |c^{-1}_{\bq, l} (\chi)| \le m - 1.
\end{equation}

It is not hard to see that this filtration is compatible with the
pre-Lie multiplication $\bullet$ \eqref{bul} on $\Conv(\cC, \cO)$
and  $\Conv(\cC, \cO)$ is complete with respect to 
this filtration.

\subsection{Free $\Xi$-colored operad}
\label{sec:free-op}
 
Let $Q$ be a $\Xi$-colored collection. Following \cite{BM-colors}, 
the spaces $\psop(Q)(\bq)$ of the free $\Xi$-colored pseudo-operad 
generated by the collection $Q$ are 
\begin{equation}
\label{psop-bq}
\psop(Q)(\bq) = \colim \und{Q} \Big|_{\Tree^{\Xi}(\bq)}\,,
\end{equation}
where $\Tree^{\Xi}(\bq)$ is the full subcategory of 
$\Tree^{\Xi}$ whose objects are labeled $\Xi$-colored planar 
trees $\bt$ satisfying condition \eqref{ka-bt-bq}.  

The pseudo-operad structure on $\psop(Q)$ is defined in 
the obvious way using grafting of trees. 

The free $\Xi$-colored operad $\Op(Q)$ is obtained from 
$\psop(Q)$ via adjoining the units.

Unfolding \eqref{psop-bq} we see that $\psop(Q)(\bq)$
is the quotient of the direct sum 
\begin{equation}
\label{pre-op-bq}
\bigoplus_{\bt, \ka(\bt) = \bq}  \und{Q} (\bt)
\end{equation}
by the subspace spanned by vectors of the form
$$
(\bt, X) - (\bt', \und{Q}(\la)(X))
$$
where $\la : \bt \to \bt'$ is a morphism in $\Tree^{\Xi}(\bq)$
and $X \in \und{Q}(\bt)$\,.

~\\[-0.5cm]

Thus it is convenient to represent vectors in  $\psop(Q)$ and in $\Op(Q)$
by labeled $\Xi$-colored planar trees with nodal vertices decorated 
by vectors in $Q$\,. The decoration is subject to this rule: if $\ka(x)$
is the corolla formed by all edges adjacent to a nodal vertex $x$ 
then $x$ is decorated by a vector $v_x \in Q(\ka(x))$\,.

If a decorated tree $\bt'$ is obtained from a decorated tree $\bt$
by applying an element $\si \in S_{\ka(x)}$ to incoming edges 
of a vertex $x$ and replacing the vector $v_x$ by $\si^{-1} (v_x)$
then $\bt'$ and $\bt$ represent the same vectors in  \eqref{psop-bq}.

\begin{example}
\label{ex:decor-trees}
Let $Q$ be a $2$-colored collection.
Figure \ref{fig:bt-decor} shows a labeled 2-colored 
tree $\bt$ decorated by vectors $v_1 \in Q(1,2)^{\mo}$, 
$v_2 \in Q(2,0)^{\mc}$ and $v_3 \in Q(1,0)^{\mo}$\,. 
Figure \ref{fig:wtbt-decor} shows another decorated 
tree with  $v'_1 = (\id, \si_{12})(v_1)$ and 
$v'_2 = \si_{12} (v_2)$, where $\si_{12}$ is the transposition
in $S_2$.
According to our discussion, these trees represent the 
same vector in $\Op(Q)(3,1)^{\mo}$\,.
\begin{figure}[htp] 
\begin{minipage}[t]{0.45\linewidth}
\centering 
\begin{tikzpicture}[scale=0.5]
\tikzstyle{lab}=[circle, draw, minimum size=4, inner sep=1]
\tikzstyle{n}=[circle, draw, fill, minimum size=4]
\tikzstyle{vt}=[circle, draw, fill, minimum size=0, inner sep=1]
\node[vt] (l1) at (1, 4) {};
\draw (1,4.5) node[anchor=center] {{\small $2_{\mc}$}};
\node[vt] (l2) at (2, 4) {};
\draw (2,4.5) node[anchor=center] {{\small $1_{\mc}$}};
\node[vt] (l3) at (3, 4) {};
\draw (3,4.5) node[anchor=center] {{\small $1_{\mo}$}};
\node[vt] (l4) at (5, 4) {};
\draw (5,4.5) node[anchor=center] {{\small $3_{\mc}$}};
\node[lab] (v2) at (1.5, 3) {{\small $v_2$}};
\node[lab] (v1) at (3, 2) {{\small $v_1$}};
\node[lab] (v3) at (4.5, 3) {{\small $v_3$}};
\node[vt] (r) at (3, 0.5) {};
\draw [dashed] (r) edge (v1);
\draw (v1) edge (v2);
\draw (v2) edge (l1);
\draw (v2) edge (l2);
\draw [dashed] (v1) edge (l3);
\draw [dashed] (v1) edge (v3);
\draw (v3) edge (l4);
\end{tikzpicture}
\caption{A 2-colored decorated tree $\bt$} \label{fig:bt-decor}
\end{minipage}
\begin{minipage}[t]{0.45\linewidth}
\centering 
\begin{tikzpicture}[scale=0.5]
\tikzstyle{lab}=[circle, draw, minimum size=4, inner sep=1]
\tikzstyle{n}=[circle, draw, fill, minimum size=4]
\tikzstyle{vt}=[circle, draw, fill, minimum size=0, inner sep=1]
\node[vt] (l1) at (1, 5) {};
\draw (1,5.5) node[anchor=center] {{\small $1_{\mc}$}};
\node[vt] (l2) at (2, 5) {};
\draw (2,5.5) node[anchor=center] {{\small $2_{\mc}$}};
\node[vt] (l3) at (5, 3.5) {};
\draw (5,4) node[anchor=center] {{\small $1_{\mo}$}};
\node[vt] (l4) at (3, 5) {};
\draw (3,5.5) node[anchor=center] {{\small $3_{\mc}$}};
\node[lab] (v2) at (1.5, 3.5) {{\small $v'_2$}};
\node[lab] (v1) at (3, 2) {{\small $v'_1$}};
\node[lab] (v3) at (3, 3.8) {{\small $v_3$}};
\node[vt] (r) at (3, 0.5) {};
\draw [dashed] (r) edge (v1);
\draw (v1) edge (v2);
\draw (v2) edge (l1);
\draw (v2) edge (l2);
\draw [dashed] (v1) edge (l3);
\draw [dashed] (v1) edge (v3);
\draw (v3) edge (l4);
\end{tikzpicture}
\caption{A 2-colored decorated tree $\wt{\bt}$. 
Here $v'_1 = (\id, \si_{12})(v_1)$ and $v'_2 = \si_{12}(v_2)$} \label{fig:wtbt-decor}
\end{minipage}
\end{figure}
\end{example}

\subsection{The cobar construction in the colored setting}

The cobar construction \cite{Fresse}, \cite{GJ}, \cite{GK}, \cite[Section 6.5]{LV-book}
is a functor from the category of coaugmented 
cooperads (in $\Ch_{\bbK}$) to the category of augmented operads 
(in $\Ch_{\bbK}$). It is used to construct free resolutions 
for operads. In this section, we briefly describe the cobar construction 
in the colored setting. 

Let $\cC$ be a coaugmented $\Xi$-colored cooperad in 
the category $\Ch_{\bbK}$
and $\cC_{\circ}$ be the cokernel of coaugmentation. As an 
operad in the category $\grVect_{\bbK}$, $\Cobar(\cC)$
is freely generated by the collection $\bs\, \cC_{\circ}$
\begin{equation}
\label{cobar-free}
\Cobar(\cC) = \Op(\bs\, \cC_{\circ})\,.
\end{equation}
Thus, it suffices to define the differential $\pa^{\Cobar}$ 
on generators $X \in \bs\, \cC_{\circ}$\,.

The differential $\pa^{\Cobar}$ on $\Cobar(\cC)$ can be written as the 
sum  
$$
\pa^{\Cobar} = \pa' + \pa''\,,
$$
with 
\begin{equation}
\label{pa-pr}
\pa' (X) =  - \bs\,  \pa_{\cC} \, \bs^{-1} X 
\end{equation}
and
\begin{equation}
\label{pa-prpr}
\pa'' (X) =  -  \bigoplus_{z \in \Isom^{\Xi}_2(\bq)} (\bs \otimes \bs) \,
\big(\bt_z; \D_{\bt_z} (\bs^{-1} X )\big),  
\end{equation}
where $X \in \bs \, \cC_{\circ}(\bq)$, 
$ \Isom^{\Xi}_2(\bq)$ is the set of isomorphism 
classes in $\Tree^{\Xi}_2(\bq)$, the tree $\bt_z$ is any representative 
of the class $z$, and $\pa_{\cC}$ is the differential on $\cC$\,. 

Properties of comultiplications $\D_{\bt}$ imply that 
the right hand side of \eqref{pa-prpr} does not depend on
the choice of representatives $\bt_z$\,.
Furthermore, using the identity  $\big(\pa_{\cC}\big)^2 = 0$
and the compatibility of $\pa_{\cC}$ with 
comultiplications $\D_{\bt}$ one easily deduces that 
$$
\pa' \circ \pa' = 0\,,
$$
and 
$$
\pa' \circ \pa'' + \pa'' \circ \pa' = 0\,.
$$
Finally  the coassociativity law \eqref{coassoc}
implies that 
\begin{equation}
\label{pa-cobar-sq}
\pa'' \circ \pa'' = 0\,.
\end{equation}

Let $\cO$ be a $\Xi$-colored operad in $\Ch_{\bbK}$\,. 
We claim that 
\begin{prop}
\label{prop:cobar-conv}
For every coaugmented $\Xi$-colored cooperad $\cC$ 
(in $\Ch_{\bbK}$), operad morphisms from $\Cobar(\cC)$ to $\cO$ are 
in bijection with MC elements of the Lie algebra 
\begin{equation}
\label{Conv-cCc-cO}
\Conv(\cC_{\circ}, \cO)\,,
\end{equation}
where $\cO$ is viewed as a $\Xi$-colored pseudo-operad 
via the forgetful functor. 
\end{prop}
\begin{proof}
Since $\Cobar(\cC)$ is freely generated by the 
$\Xi$-colored collection $\bs\, \cC_{\circ}$ any operad morphism 
$$
F : \Cobar (\cC) \to \cO
$$ 
is uniquely determined by its restriction to $\bs\,\cC_{\circ}$\,.

Let us denote by $\al_{F}$ the degree $1$ element 
\begin{equation}
\label{al-F}
\al_{F} : \Conv(\cC_{\circ}, \cO)
\end{equation}
corresponding to the restriction 
$$
F \Big|_{\bs\,\cC_{\circ}} : \bs\,\cC_{\circ} \to \cO\,.
$$

A direct computation shows that the compatibility of 
$F$ with the differentials is equivalent to the MC 
equation on $\al_F$ in the Lie algebra \eqref{Conv-cCc-cO}.
\end{proof}

\begin{remark}
\label{rem:brack-pa-cobar}
It is possible to express the Lie bracket on $\Conv(\cC_{\circ}, \cO)$ in terms of  the 
portion $\pa''$ \eqref{pa-prpr} of the cobar differential $\pa^{\Cobar}$\,.
More precisely, for $f, g \in \Conv(\cC_{\circ}, \cO)$ and $X \in \cC_{\circ}$
we have  
\begin{equation}
\label{brack-pa-cobar}
[f,g] (X) =(-1)^{|g|} \mu \big( f \bs^{-1} \otimes  g \bs^{-1} (\pa''(\bs X)) \big) 
- (-1)^{|f| |g|} (f \leftrightarrow g)\,,
\end{equation}
where $f \bs^{-1}$ and $g \bs^{-1}$ act in the obvious way 
on the tensor factors of $\pa''(\bs X) \in \Op (\bs\, \cC_{\circ})$ 
and $\mu$ denotes the  multiplication map
$$
\mu : \Op(\cO) \to \cO\,.
$$
\end{remark}

\section{Operad $\dGra$ and its 2-colored extension $\KGra$}
\label{sec:dGra-KGra}

Let us remind from \cite{Thomas} the operad (in $\grVect_{\bbK}$) of 
directed labeled graphs $\dGra$\,. 

To define the space $\dGra(n)$ we introduce an auxiliary set $\dgra_{n}$.
An element of $\dgra_{n}$ is a directed graph $\G$ with 
the set of vertices $\{1,2, \dots, n\}$ and with a total order on the set of edges.   
We require that each directed graph $\G$ in $\dgra_n$ has no 
multiple edges with the same direction\footnote{This allows us to identify 
elements of $\dgra_n$ with ordered subsets of ordered pairs 
of numbers $1,2, \dots, n$. Let us also recall that we do not consider 
graphs with loops (i.e. cycles of length $1$).}. 
An example of an element in $\dgra_5$ is 
shown in figure \ref{fig:exam}. 
\begin{figure}[htp] 
\centering 
\begin{tikzpicture}[scale=0.5, >=stealth']
\tikzstyle{w}=[circle, draw, minimum size=3, inner sep=1]
\tikzstyle{b}=[circle, draw, fill, minimum size=3, inner sep=1]
\node [b] (b1) at (0,0) {};
\draw (-0.4,0) node[anchor=center] {{\small $1$}};
\node [b] (b3) at (2,0) {};
\draw (2, 0.5) node[anchor=center] {{\small $3$}};
\node [b] (b2) at (6,0) {};
\draw (5.9, 0.6) node[anchor=center] {{\small $2$}};
\node [b] (b4) at (2,2) {};
\draw (2, 2.6) node[anchor=center] {{\small $4$}};
\node [b] (b5) at (5,2) {};
\draw (5, 2.6) node[anchor=center] {{\small $5$}};
\draw [->] (b3) edge (b1);
\draw (1,0.4) node[anchor=center] {{\small $i$}};
\draw [->] (b3) ..controls (3,0.5) and (5,0.5) .. (b2);
\draw (4,0.8) node[anchor=center] {{\small $ii$}}; 
\draw [<-] (b3) ..controls (3,-0.5) and (5,-0.5) .. (b2);
\draw (4,-0.8) node[anchor=center] {{\small $iii$}}; 
\end{tikzpicture}
\caption{Roman numerals indicate that  
$(3,1) < (3,2) < (2,3)$ } \label{fig:exam}
\end{figure}
We will often use roman numerals to specify a total order on a set of edges. 
For example, the roman numerals in figure \ref{fig:exam} indicate that  
$(3,1) < (3,2) < (2,3)$\,.

The space $\dGra(n)$ is spanned by elements of 
$\dgra_n$, modulo the relation $\G^{\si} = (-1)^{|\si|} \G$
where the graphs $\G^{\si}$ and $\G$ correspond to the same
directed labeled graph but differ only by permutation $\si$
of edges. We also declare that 
the degree of a graph $\G$ in $\dGra(n)$ equals 
$-e(\G)$, where $e(\G)$ is the number of edges in $\G$\,.
For example, the graph $\G$ in figure \ref{fig:exam} has 
$3$ edges. Thus its degree is $-3$\,. 

\subsection{Operad structure on $\dGra$} 
Let $\G$ and $\wt{\G}$ be graphs representing vectors 
in $\dGra(n)$ and $\dGra(m)$, respectively. 

For $1\le i \le m$, the vector $\wt{\G}\, \circ_{i}\, \G \in \dGra(n+m-1)$ is represented 
by the sum of graphs $\G_{\al} \in \dgra_{n+m-1}$
 \begin{equation}
\label{circ-i-mc}
\wt{\G} \,\circ_{i}\, \G  = \sum_{\al} \G_{\al}\,,
\end{equation}
where $\G_{\al}$ is obtained by  
 ``plugging in'' the graph $\G$ into the $i$-th vertex of the graph $\tG$ and   
reconnecting the edges incident to the $i$-th vertex of $\tG$ to vertices of $\G$ in 
all possible ways. (The index $\al$ refers to a particular way of connecting the 
edges incident to the $i$-th vertex of $\tG$ to vertices of $\G$. )
After reconnecting edges we  
\begin{itemize}

\item shift all labels on vertices of $\G$ up by $i-1$, and

\item shift the labels on the last $m-i$ vertices 
of $\wt{\G}$ up by $n-1$.

\end{itemize}
To define the total order on edges of the graph $\G_{\al}$ we declare 
that all edges of $\wt{\G}$ are smaller than all edges of the graph $\G$\,. 

Note that every graph in $\{ \G_{\al} \}_{\al}$ is a legitimate element of 
$\dgra_{n+m-1}$ because $\G$ and $\wt{\G}$ have no multiple edges with the same direction 
and have no loops. 
 
\begin{example}
\label{ex:Gra-ins}
Let $\wt{\G}$ (resp. $\G$) be the graph depicted in figure 
\ref{fig:wtG} (resp. figure \ref{fig:G12})\,. The vector 
$\wt{\G} \circ_2 \G$ is shown in figure \ref{fig:Gra-ins}.
\begin{figure}[htp]
\begin{minipage}[t]{0.45\linewidth}
\centering 
\begin{tikzpicture}[scale=0.5, >=stealth']
\tikzstyle{w}=[circle, draw, minimum size=4, inner sep=1]
\tikzstyle{b}=[circle, draw, fill, minimum size=4, inner sep=1]
\node [b] (b1) at (0,0) {};
\draw (0,-0.6) node[anchor=center] {{\small $1$}};
\node [b] (b2) at (1,2) {};
\draw (1.4,2.2) node[anchor=center] {{\small $2$}};
\node [b] (b3) at (-1,2) {};
\draw (-1.4,2.2) node[anchor=center] {{\small $3$}};
\draw [->] (b1) edge (b2);
\draw [->] (b1) edge (b3);
\draw [->] (b3) edge (b2);
\end{tikzpicture}
~\\[0.3cm]
\caption{A graph $\wt{\G} \in \dgra_{3}$. The order on edges is
$(1,2) < (1,3) < (3,2)$} \label{fig:wtG}
\end{minipage}
\begin{minipage}[t]{0.45\linewidth}
\centering 
\begin{tikzpicture}[scale=0.5, >=stealth']
\tikzstyle{w}=[circle, draw, minimum size=4, inner sep=1]
\tikzstyle{b}=[circle, draw, fill, minimum size=4, inner sep=1]
\node [b] (b1) at (0,0) {};
\draw (0,-0.6) node[anchor=center] {{\small $1$}};
\node [b] (b2) at (1,2) {};
\draw (1.4,2.2) node[anchor=center] {{\small $2$}};
\draw [->] (b1) edge (b2);
\end{tikzpicture}
~\\[0.3cm]
\caption{A graph $\G \in \dgra_{2}$} \label{fig:G12}
\end{minipage}
\end{figure} 
\begin{figure}[htp]
\begin{minipage}[t]{0.3\linewidth}
\centering 
\begin{tikzpicture}[scale=0.5, >=stealth']
\tikzstyle{w}=[circle, draw, minimum size=4, inner sep=1]
\tikzstyle{b}=[circle, draw, fill, minimum size=4, inner sep=1]
\draw (-3.5,1) node[anchor=center] {{$\wt{\G} \circ_2 \G \quad = $}};
\node [b] (b1) at (0,0) {};
\draw (0,-0.6) node[anchor=center] {{\small $1$}};
\node [b] (b2) at (1,2) {};
\draw (1.3,1.6) node[anchor=center] {{\small $2$}};
\node [b] (b3) at (2.5,2.5) {};
\draw (2.9,2.6) node[anchor=center] {{\small $3$}};
\node [b] (b4) at (-1,2) {};
\draw (-1.4,2.2) node[anchor=center] {{\small $4$}};
\draw [->] (b1) edge (b2);
\draw [->] (b1) edge (b4);
\draw [->] (b4) edge (b2);
\draw [->] (b2) edge (b3);
\end{tikzpicture}
~\\[0.3cm]
\end{minipage}
\begin{minipage}[t]{0.2\linewidth}
\centering 
\begin{tikzpicture}[scale=0.5, >=stealth']
\tikzstyle{w}=[circle, draw, minimum size=4, inner sep=1]
\tikzstyle{b}=[circle, draw, fill, minimum size=4, inner sep=1]
\draw (-2.5,1) node[anchor=center] {{$ + $}};
\node [b] (b1) at (0,0) {};
\draw (0,-0.6) node[anchor=center] {{\small $1$}};
\node [b] (b2) at (1,2) {};
\draw (1.3,1.6) node[anchor=center] {{\small $3$}};
\node [b] (b3) at (2.5,2.5) {};
\draw (2.9,2.6) node[anchor=center] {{\small $2$}};
\node [b] (b4) at (-1,2) {};
\draw (-1.4,2.2) node[anchor=center] {{\small $4$}};
\draw [->] (b1) edge (b2);
\draw [->] (b1) edge (b4);
\draw [->] (b4) edge (b2);
\draw [<-] (b2) edge (b3);
\end{tikzpicture}
~\\[0.3cm]
\end{minipage}
\begin{minipage}[t]{0.2\linewidth}
\centering 
\begin{tikzpicture}[scale=0.5, >=stealth']
\tikzstyle{w}=[circle, draw, minimum size=4, inner sep=1]
\tikzstyle{b}=[circle, draw, fill, minimum size=4, inner sep=1]
\draw (-3,1) node[anchor=center] {{$ + $}};
\node [b] (b1) at (0,0) {};
\draw (0,-0.6) node[anchor=center] {{\small $1$}};
\node [b] (b2) at (1.5,1.5) {};
\draw (1.6,1) node[anchor=center] {{\small $2$}};
\node [b] (b3) at (0,3) {};
\draw (0,3.5) node[anchor=center] {{\small $3$}};
\node [b] (b4) at (-1.5,1.5) {};
\draw (-1.6,1) node[anchor=center] {{\small $4$}};
\draw [->] (b1) edge (b2);
\draw [->] (b2) edge (b3);
\draw [->] (b4) edge (b3);
\draw [->] (b1) edge (b4);
\end{tikzpicture}
~\\[0.3cm]
\end{minipage}
\begin{minipage}[t]{0.2\linewidth}
\centering 
\begin{tikzpicture}[scale=0.5, >=stealth']
\tikzstyle{w}=[circle, draw, minimum size=4, inner sep=1]
\tikzstyle{b}=[circle, draw, fill, minimum size=4, inner sep=1]
\draw (-2.8,1) node[anchor=center] {{$ + $}};
\node [b] (b1) at (0,0) {};
\draw (0,-0.6) node[anchor=center] {{\small $1$}};
\node [b] (b2) at (1.5,1.5) {};
\draw (1.6,1) node[anchor=center] {{\small $3$}};
\node [b] (b3) at (0,3) {};
\draw (0,3.5) node[anchor=center] {{\small $2$}};
\node [b] (b4) at (-1.5,1.5) {};
\draw (-1.6,1) node[anchor=center] {{\small $4$}};
\draw [->] (b1) edge (b2);
\draw [<-] (b2) edge (b3);
\draw [->] (b4) edge (b3);
\draw [->] (b1) edge (b4);
\end{tikzpicture}
~\\[0.3cm]
\end{minipage}
\caption{The vector $\wt{\G} \circ_2  \G \in \dGra(4)$} \label{fig:Gra-ins}
\end{figure} 
For the first graph in the sum 
$\wt{\G} \circ_2  \G$ we have $(1,2) < (1,4) < (4,2) < (2,3)\,.$
For the second graph in the sum 
$\wt{\G} \circ_2  \G$ we have $
(1,3) < (1,4) < (4,3) < (2,3)\,.$
For the third graph in the sum 
$\wt{\G} \circ_2  \G$ we have 
$(1,2) < (1,4) < (4,3) < (2,3)\,.$
Finally, for the last graph in the sum 
$\wt{\G} \circ_2 \G$ we have 
$(1,3) < (1,4) < (4,2) < (2,3)\,.$
\end{example}

The symmetric group $S_n$ acts on $\dGra(n)$ in the 
obvious way by rearranging the labels on vertices. 
It is not hard to see that insertions \eqref{circ-i-mc}
together with this action of $S_n$ give on $\dGra$ an 
operad structure with the identity element being the unique
graph in $\dgra_1$ with no edges.

\subsection{2-colored operad $\KGra$}
\label{sec:KGra}

To define a stable formality quasi-iso\-mor\-phism (SFQ) 
we need to upgrade the operad $\dGra$ to a $2$-colored operad $\KGra$ 
(in $\grVect_{\bbK}$). 
The additional spaces of the operad $\KGra$ are assembled from the 
graphs which were used by M. Kontsevich in his groundbreaking paper 
\cite{K}. As far as I understand, T. Willwacher is using this operad 
in \cite{Thomas-KS} under the different name: $\SGra$.    

Recall that, following our conventions, $\KGra(n,k)^{\mc}$ denotes
the space of operations with $n$ inputs of color $\mc$, $k$ inputs of color 
$\mo$, and with the color of the output being $\mc$. Similarly,  
$\KGra(n,k)^{\mo}$ is the space of
operations with $n$ inputs of color $\mc$, $k$ inputs of color 
$\mo$, and with the color of the output being $\mo$\,.

First, we declare that $\KGra(n,k)^{\mc} = \bfzero  $ whenever $k \ge 1$\,. 

Next, for the space $\KGra(n,0)^{\mc}$ ($n \ge 0$), we have 
\begin{equation}
\label{KGra-mc}
\KGra(n,0)^{\mc} = \dGra(n)\,.
\end{equation}

To define the space $\KGra(n,k)^{\mo}$ we introduce 
the auxiliary set $\dgra_{n,k}$\,.   
An element of the set $\dgra_{n,k}$ is a directed graph $\G$
with the set of vertices $\{1_{\mc}, \dots, n_{\mc}, 1_{\mo}, \dots, k_{\mo}\}$
and with a total order on the set of its edges. In addition, we require that  
\begin{itemize}

\item each $\G \in \dgra_{n,k}$ has no multiple edges with the same 
direction, and 

\item each $\G \in \dgra_{n,k}$ has no edges originating from any vertex with color $\mo$. 

\end{itemize}

\begin{example}
\label{ex:KGra-mc-mo}
Figure \ref{fig:exam-mc-mo} shows an 
example of a graph in $\dgra_{2,3}$. Black (resp. white) vertices carry
the color $\mc$ (resp. $\mo$). We use separate labels for vertices of 
color $\mc$ and vertices of color $\mo$\,. For example, $2_{\mc}$
denotes the second vertex of color $\mc$ and 
$3_{\mo}$ denotes the third vertex of color $\mo$. 
\begin{figure}[htp] 
\centering 
\begin{tikzpicture}[scale=0.5, >=stealth']
\tikzstyle{w}=[circle, draw, minimum size=4, inner sep=1]
\tikzstyle{b}=[circle, draw, fill, minimum size=4, inner sep=1]
\node [w] (w2) at (0,0) {};
\draw (0.1,-0.6) node[anchor=center] {{\small $2_{\mo}$}};
\node [w] (w1) at (1,0) {};
\draw (1.1,-0.6) node[anchor=center] {{\small $1_{\mo}$}};
\node [w] (w3) at (2,0) {};
\draw (2.1,-0.6) node[anchor=center] {{\small $3_{\mo}$}};
\node [b] (b1) at (0,2) {};
\draw (0,2.6) node[anchor=center] {{\small $1_{\mc}$}};
\node [b] (b2) at (2,2) {};
\draw (2,2.6) node[anchor=center] {{\small $2_{\mc}$}};
\draw [->] (b1) edge (w1);
\draw [->] (b1) edge (b2);
\draw [->] (b2) edge (w1);
\draw [->] (b2) edge (w3);
\end{tikzpicture}
\caption{We equip the edges with the order 
$(1_{\mc}, 2_{\mc}) < (1_{\mc}, 1_{\mo}) <  (2_{\mc}, 1_{\mo})
<  (2_{\mc}, 3_{\mo}) $ } \label{fig:exam-mc-mo}
\end{figure}
\end{example}

The space $\KGra(n,k)^{\mo}$ is spanned by elements of 
$\dgra_{n,k}$, modulo the relation $\G^{\si} = (-1)^{|\si|} \G$
where $\G^{\si}$ and $\G$ correspond to the same
directed labelled graph but differ only by permutation $\si$
of edges.  As above, we declare that 
the degree of a graph $\G$ in $\KGra(n,k)^{\mo}$ equals $-e(\G)$\,.

The elementary insertions 
$$
\KGra(m,0)^{\mc} \otimes  \KGra(n,0)^{\mc} \to 
 \KGra(m+n-1,0)^{\mc} 
$$
are defined in the same way as for $\dGra$\,.
So we proceed to the remaining insertions.

\subsection{Elementary insertions $\KGra(m,k)^{\mo} \otimes  \KGra(n,0)^{\mc} \to 
 \KGra(m+n-1,k)^{\mo} $ } 

Let $\G$ and $\wt{\G}$ be graphs representing vectors 
in $\KGra(n, 0)^{\mc}$ and $\KGra(m,k)^{\mo}$, respectively. 
Let  $1\le i \le m$\,.

The vector $\wt{\G} \circ_{i, \mc} \G \in \KGra(n+m-1, k)^{\mo}$ is 
the sum of graphs $\G_{\al} \in \dgra_{n+m-1, k}$
 \begin{equation}
\label{circ-i-mc-mo}
\wt{\G} \circ_{i, \mc} \G  = \sum_{\al} \G_{\al}\,,
\end{equation}
where $\G_{\al}$ is obtained by  
 ``plugging in'' the graph $\G$ into the $i$-th black vertex of the graph $\wt{\G}$ and   
reconnecting the edges incident to this vertex to vertices of $\G$ in 
all possible ways. (The index $\al$ refers to a particular way of connecting the 
edges incident to the $i$-th black vertex of $\wt{\G}$ to vertices of $\G$. )
After reconnecting edges we 
\begin{itemize}

\item shift all labels on vertices of $\G$ up by $i-1$, and

\item shift labels on the last $m-i$ black 
vertices of $\wt{\G}$ up by $n-1$.

\end{itemize}
To define the total order on edges of the graph $\G_{\al}$ we declare 
that all edges of $\wt{\G}$ are smaller than all edges of the graph $\G$\,. 

%
%
\begin{example}
\label{ex:ins-mc-mo}
The graphs depicted in figures \ref{fig:tG} and \ref{fig:G}
represent vectors $\tG \in \KGra(2,1)^{\mo}$ and
$\G \in \KGra(2,0)^{\mc}$, respectively. 
For the edges of $\wt{\G}$ we set
$$
(1_{\mc}, 1_{\mo}) < (1_{\mc}, 2_{\mc}) < (2_{\mc}, 1_{\mo})\,.
$$
As above, white vertices carry the color $\mo$
and black vertices carry the color $\mc$\,. 
\begin{figure}[htp] 
\begin{minipage}[t]{0.33\linewidth}
\centering 
\begin{tikzpicture}[scale=0.5, >=stealth']
\tikzstyle{w}=[circle, draw, minimum size=4, inner sep=1]
\tikzstyle{b}=[circle, draw, fill, minimum size=4, inner sep=1]
\node [w] (w1) at (1,0) {};
\draw (1.1,-0.6) node[anchor=center] {{\small $1_{\mo}$}};
\node [b] (b1) at (0,2) {};
\draw (0,2.6) node[anchor=center] {{\small $1_{\mc}$}};
\node [b] (b2) at (2,2) {};
\draw (2,2.6) node[anchor=center] {{\small $2_{\mc}$}};
\draw [->] (b1) edge (w1);
\draw [->] (b1) edge (b2);
\draw [->] (b2) edge (w1);
\end{tikzpicture}
\caption{The graph $\tG$} \label{fig:tG}
\end{minipage} 
\begin{minipage}[t]{0.33\linewidth}
\centering 
\begin{tikzpicture}[scale=0.5, >=stealth']
\tikzstyle{w}=[circle, draw, minimum size=4, inner sep=1]
\tikzstyle{b}=[circle, draw, fill, minimum size=4, inner sep=1]
\node [b] (b2) at (0,2) {};
\draw (0,2.6) node[anchor=center] {{\small $2_{\mc}$}};
\node [b] (b1) at (2,2) {};
\draw (2,2.6) node[anchor=center] {{\small $1_{\mc}$}};
\draw (1,1) node[anchor=center] {{$\phantom{a}$}};
\draw [->] (b2) edge (b1);
\end{tikzpicture}
\caption{The graph $\G$} \label{fig:G}
\end{minipage} 
\begin{minipage}[t]{0.33\linewidth}
\centering 
\begin{tikzpicture}[scale=0.5, >=stealth']
\tikzstyle{w}=[circle, draw, minimum size=4, inner sep=1]
\tikzstyle{b}=[circle, draw, fill, minimum size=4, inner sep=1]
\node [w] (w1) at (1,0) {};
\draw (1.1,-0.6) node[anchor=center] {{\small $1_{\mo}$}};
\node [b] (b1) at (0,2) {};
\draw (0,2.6) node[anchor=center] {{\small $1_{\mc}$}};
\node [b] (b2) at (0,0) {};
\draw (-0.5,0.2) node[anchor=center] {{\small $2_{\mc}$}};
\node [b] (b3) at (2,2) {};
\draw (2,2.6) node[anchor=center] {{\small $3_{\mc}$}};
\draw [->] (b1) edge (w1);
\draw [->] (b1) edge (b3);
\draw [->] (b2) edge (b1);
\draw [->] (b3) edge (w1);
\end{tikzpicture}
\caption{The graph $\G_1$} \label{fig:G1}
\end{minipage}
\end{figure}

The vector $\wt{\G} \circ_{1, \mc} \G \in \KGra(3,1)^{\mo}$ is represented by the sum of 
graphs $\G_1$, $\G_2$, $\G_3$, $\G_4$ depicted on
figures \ref{fig:G1},  \ref{fig:G2}, \ref{fig:G3}, \ref{fig:G4}, 
respectively. 
\begin{figure}[htp] 
\begin{minipage}[t]{0.33\linewidth}
\centering 
\begin{tikzpicture}[scale=0.5, >=stealth']
\tikzstyle{w}=[circle, draw, minimum size=4, inner sep=1]
\tikzstyle{b}=[circle, draw, fill, minimum size=4, inner sep=1]
\node [w] (w1) at (1,0) {};
\draw (1.1,-0.6) node[anchor=center] {{\small $1_{\mo}$}};
\node [b] (b2) at (0,2) {};
\draw (0,2.6) node[anchor=center] {{\small $2_{\mc}$}};
\node [b] (b1) at (0,0) {};
\draw (-0.5,0.2) node[anchor=center] {{\small $1_{\mc}$}};
\node [b] (b3) at (2,2) {};
\draw (2,2.6) node[anchor=center] {{\small $3_{\mc}$}};
\draw [->] (b2) edge (w1);
\draw [->] (b2) edge (b3);
\draw [->] (b2) edge (b1);
\draw [->] (b3) edge (w1);
\end{tikzpicture}
\caption{The graph $\G_2$} \label{fig:G2}
\end{minipage}
\begin{minipage}[t]{0.33\linewidth}
\centering 
\begin{tikzpicture}[scale=0.5, >=stealth']
\tikzstyle{w}=[circle, draw, minimum size=4, inner sep=1]
\tikzstyle{b}=[circle, draw, fill, minimum size=4, inner sep=1]
\node [w] (w1) at (1,0) {};
\draw (1.1,-0.6) node[anchor=center] {{\small $1_{\mo}$}};
\node [b] (b2) at (-1,2) {};
\draw (-1,2.6) node[anchor=center] {{\small $2_{\mc}$}};
\node [b] (b1) at (1,2) {};
\draw (1,2.6) node[anchor=center] {{\small $1_{\mc}$}};
\node [b] (b3) at (3,2) {};
\draw (3,2.6) node[anchor=center] {{\small $3_{\mc}$}};
\draw [->] (b2) edge (w1);
\draw [->] (b3) edge (w1);
\draw [->] (b2) edge (b1);
\draw [->] (b1) edge (b3);
\end{tikzpicture}
\caption{The graph $\G_3$} \label{fig:G3}
\end{minipage}
\begin{minipage}[t]{0.33\linewidth}
\centering 

\begin{tikzpicture}[scale=0.5, >=stealth']
\tikzstyle{w}=[circle, draw, minimum size=4, inner sep=1]
\tikzstyle{b}=[circle, draw, fill, minimum size=4, inner sep=1]
\node [w] (w1) at (1,0) {};
\draw (1.1,-0.6) node[anchor=center] {{\small $1_{\mo}$}};
\node [b] (b1) at (-1,2) {};
\draw (-1,2.6) node[anchor=center] {{\small $1_{\mc}$}};
\node [b] (b2) at (1,2) {};
\draw (1,2.6) node[anchor=center] {{\small $2_{\mc}$}};
\node [b] (b3) at (3,2) {};
\draw (3,2.6) node[anchor=center] {{\small $3_{\mc}$}};
\draw [->] (b1) edge (w1);
\draw [->] (b3) edge (w1);
\draw [->] (b2) edge (b1);
\draw [->] (b2) edge (b3);
\end{tikzpicture}
\caption{The graph $\G_4$} \label{fig:G4}
\end{minipage}
\end{figure}
Following our rule, the edges of $\G_1$ are ordered as 
follows $
(1_{\mc}, 1_{\mo}) < (1_{\mc}, 3_{\mc}) < (3_{\mc}, 1_{\mo})
< (2_{\mc}, 1_{\mc})\,. $
Similarly, edges of $\G_2$ carry the order 
$
(2_{\mc}, 1_{\mo}) < (2_{\mc}, 3_{\mc}) < (3_{\mc}, 1_{\mo})
< (2_{\mc}, 1_{\mc})\,. 
$
The edges of $\G_3$ are equipped with the order
$
(2_{\mc}, 1_{\mo}) < (1_{\mc}, 3_{\mc}) < (3_{\mc}, 1_{\mo})
< (2_{\mc}, 1_{\mc})\,. 
$
Finally, for $\G_4$ we have 
$
(1_{\mc}, 1_{\mo}) < (2_{\mc}, 3_{\mc}) < (3_{\mc}, 1_{\mo})
< (2_{\mc}, 1_{\mc})\,. 
$
\end{example}

\subsection{Elementary insertions $\KGra(m,p)^{\mo} \otimes  \KGra(n,q)^{\mo} \to 
 \KGra(m+n, p + q-1)^{\mo} $ } 
 
Let $\G$ and $\wt{\G}$ be graphs representing vectors 
in $\KGra(n, q)^{\mo}$ and $\KGra(m,p)^{\mo}$, respectively. 
Let  $1\le i \le p$\,.

The vector $\wt{\G} \circ_{i, \mo} \G \in \KGra(m+n, p+q-1)^{\mo}$ is represented 
by the sum of graphs $\G_{\al} \in \dgra_{m+n, p+q-1}$
 \begin{equation}
\label{circ-i-mo-mo}
\wt{\G} \circ_{i, \mo} \G  = \sum_{\al} \G_{\al}\,,
\end{equation}
where $\G_{\al}$ is obtained by  
 ``plugging in'' the graph $\G$ into the $i$-th white vertex of the graph $\wt{\G}$ and   
reconnecting the edges incident to this vertex to vertices of $\G$ in 
all possible ways. (The index $\al$ refers to a particular way of connecting the 
edges incident to the $i$-th white vertex of $\wt{\G}$ to vertices of $\G$. )
After reconnecting edges we
\begin{itemize}

\item shift all labels on black vertices of $\G$ up by $m$,  

\item shift all labels on white  vertices of $\G$ up by $i-1$, and finally  

\item shift all labels on the last $p-i$ white vertices of $\tG$ 
up by $q-1$.
 
\end{itemize}
To define the total order on edges of the graph $\G_{\al}$ we declare 
that all edges of $\wt{\G}$ are smaller than all edges of the graph $\G$\,. 

\begin{example}
\label{ex:ins-mo-mo}
If $\wt{\G}$ is the graph depicted in figure \ref{fig:exam-mc-mo}
and $\G$ is the graph depicted in figure \ref{fig:mc1mo1} then
\begin{figure}[htp]
\centering 
\begin{tikzpicture}[scale=0.5, >=stealth']
\tikzstyle{w}=[circle, draw, minimum size=4, inner sep=1]
\tikzstyle{b}=[circle, draw, fill, minimum size=4, inner sep=1]
\node [w] (w1) at (0,0) {};
\draw (0.1,-0.6) node[anchor=center] {{\small $1_{\mo}$}};
\node [b] (b1) at (0,2) {};
\draw (0,2.6) node[anchor=center] {{\small $1_{\mc}$}};
\draw [->] (b1) edge (w1);
\end{tikzpicture}
~\\[0.3cm]
\caption{A graph $\G \in \dgra_{1,1}$} \label{fig:mc1mo1}
\end{figure} 
the vector $\wt{\G} \circ_{3, \mo} \G \in \KGra(3,3)$ is the sum of graphs 
depicted in figure \ref{fig:exam-mo-mo}.
\begin{figure}[htp]
\centering 
\begin{tikzpicture}[scale=0.5, >=stealth']
\tikzstyle{w}=[circle, draw, minimum size=4, inner sep=1]
\tikzstyle{b}=[circle, draw, fill, minimum size=4, inner sep=1]
\draw (-3, 1) node[anchor=center] {{$\wt{\G} \circ_{3, \mo} \G \quad = $}};
\node [w] (w2) at (0,0) {};
\draw (0.1,-0.6) node[anchor=center] {{\small $2_{\mo}$}};
\node [w] (w1) at (1,0) {};
\draw (1.1,-0.6) node[anchor=center] {{\small $1_{\mo}$}};
\node [w] (w3) at (2,0) {};
\draw (2.1,-0.6) node[anchor=center] {{\small $3_{\mo}$}};
\node [b] (b1) at (0,2) {};
\draw (0,2.6) node[anchor=center] {{\small $1_{\mc}$}};
\node [b] (b2) at (2,2) {};
\draw (2,2.6) node[anchor=center] {{\small $2_{\mc}$}};
\node [b] (b3) at (3,1) {};
\draw (3,1.6) node[anchor=center] {{\small $3_{\mc}$}};
\draw [->] (b1) edge (w1);
\draw [->] (b1) edge (b2);
\draw [->] (b2) edge (w1);
\draw [->] (b2) edge (w3);
\draw [->] (b3) edge (w3);
\draw (4.5, 1) node[anchor=center] {{$+$}};
\node [w] (ww2) at (6,0) {};
\draw (6.1,-0.6) node[anchor=center] {{\small $2_{\mo}$}};
\node [w] (ww1) at (7,0) {};
\draw (7.1,-0.6) node[anchor=center] {{\small $1_{\mo}$}};
\node [w] (ww3) at (8,0) {};
\draw (8.1,-0.6) node[anchor=center] {{\small $3_{\mo}$}};
\node [b] (bb1) at (6,2) {};
\draw (6,2.6) node[anchor=center] {{\small $1_{\mc}$}};
\node [b] (bb2) at (8,2) {};
\draw (8,2.6) node[anchor=center] {{\small $2_{\mc}$}};
\node [b] (bb3) at (9,1) {};
\draw (9,1.6) node[anchor=center] {{\small $3_{\mc}$}};
\draw [->] (bb1) edge (ww1);
\draw [->] (bb1) edge (bb2);
\draw [->] (bb2) edge (ww1);
\draw [->] (bb2) edge (bb3);
\draw [->] (bb3) edge (ww3);
\end{tikzpicture}
\caption{The vector $\wt{\G} \circ_{3, \mo} \G$} \label{fig:exam-mo-mo}
\end{figure}
For the edges of the first graph in this sum we have 
$
(1_{\mc}, 2_{\mc}) < (1_{\mc}, 1_{\mo}) <  (2_{\mc}, 1_{\mo})
<  (2_{\mc}, 3_{\mo}) <  (3_{\mc}, 3_{\mo})\,.
$
For the edges of the second graph in this sum we have  
$
(1_{\mc}, 2_{\mc}) < (1_{\mc}, 1_{\mo}) <  (2_{\mc}, 1_{\mo})
<  (2_{\mc}, 3_{\mc}) <  (3_{\mc}, 3_{\mo})\,.
$
\end{example}

The identity element $\bu_{\mc} \in \KGra(1,0)^{\mc}$ (resp.  $\bu_{\mo} \in \KGra(0,1)^{\mo}$)
is represented by the graph in $\dgra_1$ (resp. the graph in $\dgra_{0,1}$) with no edges.

It is straightforward to verify that  $\bu_{\mc}$, $\bu_{\mo}$, and equations 
\eqref{circ-i-mc}, \eqref{circ-i-mc-mo}, \eqref{circ-i-mo-mo}  together with the natural 
action of $S_n \times S_k$ on $\KGra(n,k)^{\mo}$ (resp. $S_n$ on 
$\KGra(n,0)^{\mc}$)
define a structure of a $2$-colored operad on $\KGra$ in $\grVect_{\bbK}$\,.

\begin{remark}
\label{rem:mc-omitted}
When dealing with elements of $\KGra(n,0)^{\mc} = \dGra(n)$ or with 
elements of $\KGra(n,0)^{\mo}$, we will often omit the subscript $\mc$
in labels $1_{\mc}, 2_{\mc}, 3_{\mc}, \dots$.
\end{remark}
\begin{remark}
\label{rem:undir}
Let $\G$ be a graph in $\dgra_n$ (resp. $\dgra_{n,k}$) and 
$e$ be an edge of $\G$ which connects two black vertices.
We denote by $f_e(\G)$ the graph which is obtained from 
$\G$ by changing the direction of the edge $e$.

It is convenient to draw the linear combination
$\G + f_e(\G)$ as a graph which is obtained from 
$\G$ by forgetting the direction of $e$\,.
For example, 
\begin{equation}
\label{b-b}
\begin{tikzpicture}[scale=0.5, >=stealth']
\tikzstyle{w}=[circle, draw, minimum size=4, inner sep=1]
\tikzstyle{b}=[circle, draw, fill, minimum size=4, inner sep=1]
\node [b] (b1) at (0,0) {};
\draw (0,0.6) node[anchor=center] {{\small $1$}};
\node [b] (b2) at (2,0) {};
\draw (2,0.6) node[anchor=center] {{\small $2$}};
\draw (b1) edge (b2);
\end{tikzpicture}
\quad
=
\quad
\begin{tikzpicture}[scale=0.5, >=stealth']
\tikzstyle{w}=[circle, draw, minimum size=4, inner sep=1]
\tikzstyle{b}=[circle, draw, fill, minimum size=4, inner sep=1]
\node [b] (b1) at (0,0) {};
\draw (0,0.6) node[anchor=center] {{\small $1$}};
\node [b] (b2) at (2,0) {};
\draw (2,0.6) node[anchor=center] {{\small $2$}};
\draw [->](b1) edge (b2);
\end{tikzpicture}
\quad
+
\quad
\begin{tikzpicture}[scale=0.5, >=stealth']
\tikzstyle{w}=[circle, draw, minimum size=4, inner sep=1]
\tikzstyle{b}=[circle, draw, fill, minimum size=4, inner sep=1]
\node [b] (b1) at (0,0) {};
\draw (0,0.6) node[anchor=center] {{\small $1$}};
\node [b] (b2) at (2,0) {};
\draw (2,0.6) node[anchor=center] {{\small $2$}};
\draw [<-](b1) edge (b2);
\end{tikzpicture}
\quad .
\end{equation}

Similarly, if $e_1, e_2, \dots, e_p$ 
are edges of $\G$ which connect only black vertices 
and the graph $\G'$ is obtained from $\G$ by forgetting 
the directions of the edges  $e_1, e_2, \dots, e_p$,  then 
$\G'$ denotes the sum 
$$
\G' = \sum_{k_i\in  \{0,1\}} 
(f_{e_1})^{k_1} (f_{e_2})^{k_2} \dots (f_{e_p})^{k_p}
(\G)\,.
$$
For example,
\begin{equation}
\label{tri}
\begin{tikzpicture}[scale=0.5, >=stealth']
\tikzstyle{w}=[circle, draw, minimum size=4, inner sep=1]
\tikzstyle{b}=[circle, draw, fill, minimum size=4, inner sep=1]
\node [b] (b2) at (0,0) {};
\draw (0,0.6) node[anchor=center] {{\small $2$}};
\node [b] (b3) at (1,1) {};
\draw (1,1.6) node[anchor=center] {{\small $3$}};
\node [b] (b1) at (2,0) {};
\draw (2,0.6) node[anchor=center] {{\small $1$}};
\draw (b2) edge (b3);
\draw (b3) edge (b1);
\end{tikzpicture}
\quad
=
\quad
\begin{tikzpicture}[scale=0.5, >=stealth']
\tikzstyle{w}=[circle, draw, minimum size=4, inner sep=1]
\tikzstyle{b}=[circle, draw, fill, minimum size=4, inner sep=1]
\node [b] (b2) at (0,0) {};
\draw (0,0.6) node[anchor=center] {{\small $2$}};
\node [b] (b3) at (1,1) {};
\draw (1,1.6) node[anchor=center] {{\small $3$}};
\node [b] (b1) at (2,0) {};
\draw (2,0.6) node[anchor=center] {{\small $1$}};
\draw [->](b2) edge (b3);
\draw (b3) edge (b1);
\end{tikzpicture}
\quad
+
\quad
\begin{tikzpicture}[scale=0.5, >=stealth']
\tikzstyle{w}=[circle, draw, minimum size=4, inner sep=1]
\tikzstyle{b}=[circle, draw, fill, minimum size=4, inner sep=1]
\node [b] (b2) at (0,0) {};
\draw (0,0.6) node[anchor=center] {{\small $2$}};
\node [b] (b3) at (1,1) {};
\draw (1,1.6) node[anchor=center] {{\small $3$}};
\node [b] (b1) at (2,0) {};
\draw (2,0.6) node[anchor=center] {{\small $1$}};
\draw [<-] (b2) edge (b3);
\draw (b3) edge (b1);
\end{tikzpicture}
\end{equation} 
$$
= \quad 
\begin{tikzpicture}[scale=0.5, >=stealth']
\tikzstyle{w}=[circle, draw, minimum size=4, inner sep=1]
\tikzstyle{b}=[circle, draw, fill, minimum size=4, inner sep=1]
\node [b] (b2) at (0,0) {};
\draw (0,0.6) node[anchor=center] {{\small $2$}};
\node [b] (b3) at (1,1) {};
\draw (1,1.6) node[anchor=center] {{\small $3$}};
\node [b] (b1) at (2,0) {};
\draw (2,0.6) node[anchor=center] {{\small $1$}};
\draw [->] (b2) edge (b3);
\draw [->](b3) edge (b1);
\end{tikzpicture}
\quad
+
\quad
\begin{tikzpicture}[scale=0.5, >=stealth']
\tikzstyle{w}=[circle, draw, minimum size=4, inner sep=1]
\tikzstyle{b}=[circle, draw, fill, minimum size=4, inner sep=1]
\node [b] (b2) at (0,0) {};
\draw (0,0.6) node[anchor=center] {{\small $2$}};
\node [b] (b3) at (1,1) {};
\draw (1,1.6) node[anchor=center] {{\small $3$}};
\node [b] (b1) at (2,0) {};
\draw (2,0.6) node[anchor=center] {{\small $1$}};
\draw [->] (b2) edge (b3);
\draw [<-](b3) edge (b1);
\end{tikzpicture}
\quad
+
\quad
\begin{tikzpicture}[scale=0.5, >=stealth']
\tikzstyle{w}=[circle, draw, minimum size=4, inner sep=1]
\tikzstyle{b}=[circle, draw, fill, minimum size=4, inner sep=1]
\node [b] (b2) at (0,0) {};
\draw (0,0.6) node[anchor=center] {{\small $2$}};
\node [b] (b3) at (1,1) {};
\draw (1,1.6) node[anchor=center] {{\small $3$}};
\node [b] (b1) at (2,0) {};
\draw (2,0.6) node[anchor=center] {{\small $1$}};
\draw [<-] (b2) edge (b3);
\draw [->](b3) edge (b1);
\end{tikzpicture}
\quad
+
\quad
\begin{tikzpicture}[scale=0.5, >=stealth']
\tikzstyle{w}=[circle, draw, minimum size=4, inner sep=1]
\tikzstyle{b}=[circle, draw, fill, minimum size=4, inner sep=1]
\node [b] (b2) at (0,0) {};
\draw (0,0.6) node[anchor=center] {{\small $2$}};
\node [b] (b3) at (1,1) {};
\draw (1,1.6) node[anchor=center] {{\small $3$}};
\node [b] (b1) at (2,0) {};
\draw (2,0.6) node[anchor=center] {{\small $1$}};
\draw [<-] (b2) edge (b3);
\draw [<-] (b3) edge (b1);
\end{tikzpicture}
\quad .
$$
\end{remark}

\subsection{The action of the operad $\KGra$ on polyvectors and functions} 
\label{sec:acts-on-V-A}

Let $A$ be a free finitely generated commutative algebra (with unit) in $\grVect_{\bbK}$. 
We denote by 
\begin{equation}
\label{gener}
x^1, x^2, \dots, x^d
\end{equation}
generators of $A$ and by 
$|x^1|, |x^2|, \dots, |x^d|$ their corresponding degrees. 
We think of $A$ as the algebra of functions on a graded affine space. 

Let us denote by $V_A$ the free commutative algebra 
in $\grVect_{\bbK}$ generated by 
\begin{equation}
\label{x-te}
x^1, x^2, \dots, x^d, \te_1, \te_2, \dots, \te_d\,,
\end{equation}
where $\te_c$ carries the degree $1 - |x^c|$\,.
We think of $V_A$ as the algebra of polyvector fields on
the corresponding graded affine space. 

If all generators $x^c$ have degree $0$ then $A$  
(resp. $V_A$) is the algebra of functions (resp. the algebra 
of polyvector fields) on the affine space $\bbK^d$\,.
However, for our constructions 
there is no need to impose any restrictions on 
degrees of generators \eqref{gener}.

We claim that
\begin{prop}
\label{prop:action}
The pair $(V_A, A)$ is naturally 
an algebra over the $2$-colored operad $\KGra$\,.
\end{prop}
\begin{proof}
For $\G\in \dgra_{n}$ and $v_1, v_2, \dots, v_n \in V_A$ we 
set 
\begin{equation}
\label{G-acts-mc}
\G(v_1, v_2, \dots, v_n) : = 
\mult_n  
\Big(
\Big[\prod_{(i,j) \in E(\G)} \Lap_{(i,j)} \Big]
\,(v_1  \otimes v_2 \otimes \dots \otimes v_n )\,
 \Big) \,,
\end{equation}
where  $\mult_n$ is the multiplication map 
$$
\mult_n : \big( V_A\big)^{\otimes \, n}  \to  V_A,
$$ 
\begin{equation}
\label{Lap-less}
\Lap_{(i,j)} = \sum_{c=1}^d 1 \otimes \dots \otimes 1 \otimes 
\underbrace{\pa_{\te_c}}_{i\textrm{-th slot}} \otimes 1 \otimes \dots \otimes 1 
\otimes 
\underbrace{\pa_{x^c}}_{j\textrm{-th slot}}  \otimes 1  \otimes \dots  \otimes 1
\end{equation}
if $i < j$,
\begin{equation}
\label{Lap-more}
\Lap_{(i,j)} = \sum_{c=1}^d 1 \otimes \dots \otimes 1 \otimes 
\underbrace{\pa_{x^c}}_{j\textrm{-th slot}} \otimes 1 \otimes \dots \otimes 1 
\otimes 
\underbrace{\pa_{\te_c}}_{i\textrm{-th slot}}  \otimes 1  \otimes \dots  \otimes 1
\end{equation}
if $j < i$, and the order of factors in the product 
$$
\prod_{(i,j) \in E(\G)} \Lap_{(i,j)} 
$$
comes from the order on the set $E(\G)$ of edges of $\G$\,.

To define the action of a graph  $\G\in  \dgra_{n, k}$ 
we identify vertices of $\G$ with the numbers  
$1, 2, \dots, n+k$ by using the labels and declaring that 
all black vertices precede all white vertices.  Namely, 
the black vertex with label $i$ is identified with number $i$ 
and the white vertex with label $j$ is identified with number 
$n+j$\,. Then for $v_1, v_2, \dots, v_n \in V_A$, 
and $a_1, a_2, \dots, a_k \in A$ we set 
$$
\G(v_1, v_2, \dots, v_n; a_1, a_2, \dots, a_k) : =
$$
\begin{equation}
\label{G-acts-mo}
\mult_{n,k} 
\Big(
\Big[\prod_{(i,j) \in E(\G)} \Lap_{(i,j)} \Big]
\,(v_1  \otimes v_2 \otimes \dots \otimes v_n \otimes a_1 \otimes a_2 \otimes \dots \otimes a_k)\,
 \Big) \Big|_{\te_c = 0} \,,
\end{equation} 
where $\mult_{n,k}$ is the 
multiplication map 
$$
\mult_{n,k} : \big( V_A\big)^{\otimes \, n}   ~\otimes~
 A^{\otimes \, k} \to 
V_A\,,
$$ 
$\Lap_{(i,j)}$ is defined by equations \eqref{Lap-less}, 
\eqref{Lap-more}, and the order of factors in the product 
$$
\prod_{(i,j) \in E(\G)} \Lap_{(i,j)} 
$$
comes from the order on the set $E(\G)$ of edges of $\G$\,.

It is not hard to verify that equations \eqref{G-acts-mc}, 
\eqref{G-acts-mo} define an action of $\KGra$ on the 
pair $(V_A, A)$\,.
\end{proof}

\section{The $2$-colored operad $\OC$ of H. Kajiura and J. Stasheff}
\label{sec:OC}
Inspired by Zwiebach's open-closed string field theory \cite{Zwiebach}, 
H. Kajiura and J. Stasheff introduced in \cite{OCHA}
open-closed homotopy algebras (OCHA). 

An OCHA is a pair of  cochain complexes $(\cV, \cA)$ with 
the following data: 
\begin{itemize}
\item A $\La\Lie_{\infty}$-structure on $\cV$, 

\item an $A_{\infty}$-structure on $\cA$, and 

\item a $\La\Lie_{\infty}$-morphism from $\cV$ to the 
Hochschild cochain complex $\Cbu(\cA)$ of $\cA$\,.

\end{itemize}

It was shown in \cite{OCHA1}, that
OCHAs are governed by a $2$-colored operad (in $\Ch_{\bbK}$)
 which we denote by $\OC$\,. Moreover, as an operad in $\grVect$, $\OC$ 
is freely generated by the 2-colored collection $\oc$ with the following 
spaces: 
\begin{equation}
\label{oc-mc}
\oc(n,0)^{\mc} = \bs^{3-2n} \bbK\,, \qquad  n \ge 2\,, 
\end{equation}
\begin{equation}
\label{oc-mo}
\oc(0,k)^{\mo} = \bs^{2-k}\, \sgn_k \otimes \bbK[S_k]\,, \qquad  k \ge 2\,, 
\end{equation}
\begin{equation}
\label{oc-mo1}
\oc(n,k)^{\mo} = \bs^{2-2n-k}\,  \sgn_k \otimes \bbK[S_k]\,, \qquad n \ge 1\,,
\qquad k \ge 0\,, 
\end{equation}
where $\sgn_k$ is the sign representation of $S_k$\,.
The remaining spaces of the collection $\oc$ are zero. 

Following the description of free colored operads via  
decorated (and colored) trees (see Section \ref{sec:free-op}), 
we represent generators of $\OC$ in $\oc(n,0)^{\mc}$
by non-planar labeled corollas with $n$ solid incoming edges 
(see figure \ref{fig:bt-n-oc-mc}). We represent  generators of $\OC$ in 
$\oc(0,k)^{\mo}$ by planar labeled corollas with  $k$ dashed incoming edges
(see figure \ref{fig:bt-k-oc-mo}). 
Finally, we use labeled 2-colored corollas with a planar structure given only  
on the dashed edges to represent generators of $\OC$ in $\oc(n,k)$
(see figure \ref{fig:bt-nk-oc-mo}).
   
\begin{figure}[htp] 
\begin{minipage}[t]{0.3\linewidth}
\centering 
\begin{tikzpicture}[scale=0.5]
\tikzstyle{w}=[circle, draw, minimum size=3, inner sep=1]
\tikzstyle{vt}=[circle, draw, fill, minimum size=0, inner sep=1]
\node[vt] (l1) at (0, 2) {};
\draw (0,2.5) node[anchor=center] {{\small $1$}};
\node[vt] (l2) at (1, 2) {};
\draw (1,2.5) node[anchor=center] {{\small $2$}};
\draw (2,1.8) node[anchor=center] {{\small $\dots$}};
\node[vt] (ln) at (3, 2) {};
\draw (3,2.5) node[anchor=center] {{\small $n$}};
\node[w] (v) at (1.5, 1) {};
\node[vt] (r) at (1.5, 0) {};
\draw (r) edge (v);
\draw (v) edge (l1);
\draw (v) edge (l2);
\draw (v) edge (ln);
\end{tikzpicture}
\caption{The non-planar corolla $\st^{\mc}_n$
representing a generator of $\oc(n,0)^{\mc}$ } \label{fig:bt-n-oc-mc}
\end{minipage}
\hspace{0.2cm}
\begin{minipage}[t]{0.3\linewidth}
\centering 
\begin{tikzpicture}[scale=0.5]
\tikzstyle{vt}=[circle, draw, fill, minimum size=0, inner sep=1]
\tikzstyle{w}=[circle, draw, minimum size=3, inner sep=1]
\node[vt] (l1) at (0, 2) {};
\draw (0,2.5) node[anchor=center] {{\small $1$}};
\node[vt] (l2) at (1, 2) {};
\draw (1,2.5) node[anchor=center] {{\small $2$}};
\draw (2,1.8) node[anchor=center] {{\small $\dots$}};
\node[vt] (lk) at (3, 2) {};
\draw (3,2.5) node[anchor=center] {{\small $k$}};
\node[w] (v) at (1.5, 1) {};
\node[vt] (r) at (1.5, 0) {};
\draw [dashed] (r) edge (v);
\draw [dashed] (v) edge (l1);
\draw [dashed] (v) edge (l2);
\draw [dashed] (v) edge (lk);
\end{tikzpicture}
\caption{The 2-colored planar corola $\st^{\mo}_k$ representing 
a generator of $\oc(0,k)^{\mo}$ } \label{fig:bt-k-oc-mo}
\end{minipage}
\hspace{0.2cm}
\begin{minipage}[t]{0.35\linewidth}
\centering 
\begin{tikzpicture}[scale=0.5]
\tikzstyle{vt}=[circle, draw, fill, minimum size=0, inner sep=1]
\tikzstyle{w}=[circle, draw, minimum size=3, inner sep=1]
\node[vt] (l1) at (0, 2) {};
\draw (0,2.5) node[anchor=center] {{\small $1$}};
\draw (1.5,1.8) node[anchor=center] {{\small $\dots$}};
\node[vt] (ln) at (2, 2) {};
\draw (2,2.5) node[anchor=center] {{\small $n$}};
\node[vt] (l1mo) at (3, 2) {};
\draw (3,2.5) node[anchor=center] {{\small $1$}};
\draw (3.5,1.8) node[anchor=center] {{\small $\dots$}};
\node[vt] (lkmo) at (5, 2) {};
\draw (5,2.5) node[anchor=center] {{\small $k$}};
\node[w] (v) at (2.5, 1) {};
\node[vt] (r) at (2.5, 0) {};
\draw  [dashed] (r) edge (v);
\draw (v) edge (l1);
\draw (v) edge (ln);
\draw  [dashed] (v) edge (l1mo);
\draw  [dashed] (v) edge (lkmo);
\end{tikzpicture}
\caption{The 2-colored partially planar corola $\st^{\mo}_{n,k}$ representing 
a generator of $\oc(n,k)^{\mo}$ 
} \label{fig:bt-nk-oc-mo}
\end{minipage}
\end{figure}
Using the corolla $\st^{\mo}_k$ (resp. the corolla $\st^{\mo}_{n,k}$) depicted in figure 
\ref{fig:bt-k-oc-mo} (resp. \ref{fig:bt-nk-oc-mo}), we can form a basis of the vector 
space $\oc(0,k)^{\mo}$ (resp. $\oc(n,k)^{\mo}$). 
Namely, the set $\{ \si(\st^{\mo}_k) ~|~ \si \in S_k \}$ is a basis of the vector space
$\oc(0,k)^{\mo}$ and the set  $\{ (\id , \si)\big(\st^{\mo}_{n,k}\big) ~|~ \si \in S_k \}$
is a basis of the vector space $\oc(n,k)^{\mo}$.

Equations \eqref{oc-mc}, \eqref{oc-mo}, and \eqref{oc-mo1} imply 
that the corollas  $\st^{\mc}_n$, $\st^{\mo}_k$ and  $\st^{\mo}_{n,k}$
carry the following degrees: 
\begin{eqnarray}
\label{deg-mc-n}
| \st^{\mc}_n |  = 3 - 2n &  n \ge 2\,, \\[0.3cm]
\label{deg-mo-k}
| \st^{\mo}_k |  = 2 - k  &  k \ge 2\,, \\[0.3cm]
\label{deg-mo-nk}
| \st^{\mo}_{n,k} |  = 2 -2n - k &  ~~~~n \ge 1,~ k \ge 0\,. 
\end{eqnarray}

We should remark that $\OC$ comes from a Koszul operad and this 
fact was established in beautiful paper \cite{HL-OCHA} by E. Hoefel and M. Livernet.

\subsection{The differential on  $\OC$}
It is convenient to split the differential  $\cD$ on $\OC$ 
into four summands
\begin{equation}
\label{cD-OC-sum}
\cD = \cD_{\Lie} + \cD_{\As} + \cD' +  \cD''\,. 
\end{equation}

Since $\OC$ is freely generated by the 2-colored collection  $\oc$, 
it suffices to define the values of summands $\cD_{\Lie}$,
$ \cD_{\As} $, $\cD'$, and $\cD''$ 
on corollas 
$\st^{\mc}_n$, $\st^{\mo}_k$, and $\st^{\mo}_{n,k}$ depicted 
in figures \ref{fig:bt-n-oc-mc}, \ref{fig:bt-k-oc-mo}, and \ref{fig:bt-nk-oc-mo}, 
respectively. 

For the corolla $\st^{\mc}_{n}$ we have 
\begin{equation}
\label{cD-bt-n-zero}
\cD_{\As} (\st^{\mc}_n) = 0\,, 
\qquad
\cD'  (\st^{\mc}_n) = 0\,, 
\qquad 
\cD''  (\st^{\mc}_n) = 0
\end{equation}
and $\cD_{\Lie} (\st^{\mc}_n)$ is the sum 
shown in figure \ref{fig:cDLie-bt-mc}.
\begin{figure}[htp] 
\centering 
\begin{tikzpicture}[scale=0.5]
\tikzstyle{w}=[circle, draw, minimum size=3, inner sep=1]
\tikzstyle{vt}=[circle, draw, fill, minimum size=0, inner sep=1]
\draw (-4,1.5) node[anchor=center] 
{{$\displaystyle \cD_{\Lie} (\st^{\mc}_n) \quad = \quad - \quad
\sum_{p=2}^{n-1} \sum_{\tau\in \Sh_{p, n-p} } $}};
\node[vt] (l1) at (0, 3) {};
\draw (0,3.6) node[anchor=center] {{\small $\tau(1)$}};
\draw (1.5,2.8) node[anchor=center] {{\small $\dots$}};
\node[vt] (lp) at (2.5, 3) {};
\draw (2.5,3.6) node[anchor=center] {{\small $\tau(p)$}};
\node[vt] (lp1) at (4, 2) {};
\draw (4,2.6) node[anchor=center] {{\small $\tau(p+1)$}};
\draw (5.1,1.8) node[anchor=center] {{\small $\dots$}};
\node[vt] (ln) at (6.5, 2) {};
\draw (6.5,2.6) node[anchor=center] {{\small $\tau(n)$}};
\node[w] (v2) at (2, 2) {};
\node[w] (v1) at (4, 1) {};
\node[vt] (r) at (4, 0) {};
\draw (r) edge (v1);
\draw (v1) edge (v2);
\draw (v2) edge (l1);
\draw (v2) edge (lp);
\draw (v1) edge (lp1);
\draw (v1) edge (ln);
\end{tikzpicture}
\caption{ The value of $\cD_{\Lie}$ on $\st^{\mc}_{n}$} \label{fig:cDLie-bt-mc}
\end{figure}

For the corolla $\st^{\mo}_{k}$ we have 
\begin{equation}
\label{cD-bt-k-zero}
\cD_{\Lie} (\st^{\mo}_k) = 0\,, 
\qquad 
\cD' (\st^{\mo}_k) = 0\,, 
\qquad 
\cD'' (\st^{\mo}_k) = 0\,, 
\end{equation}
and $\cD_{\As} (\st^{\mo}_k)$ is the sum shown 
in figure \ref{fig:cDAs-bt-k}.
\begin{figure}[htp]
\centering 
\begin{tikzpicture}[scale=0.5]
\tikzstyle{vt}=[circle, draw, fill, minimum size=0, inner sep=1]
\tikzstyle{w}=[circle, draw, minimum size=3, inner sep=1]
\draw (-8,1.5) node[anchor=center] 
{{$\displaystyle \cD_{\As} (\st^{\mo}_k) \quad = \quad
- \sum_{p=0}^{k-2} \sum_{q = p+2}^{k}  \quad (-1)^{p + (k-q) (q-p)} $}};
\node[vt] (l1) at (0, 2) {};
\draw (0,2.5) node[anchor=center] {{\small $1$}};
\draw (1.5,1.8) node[anchor=center] {{\small $\dots$}};
\node[vt] (lp) at (2, 2) {};
\draw (2,2.5) node[anchor=center] {{\small $p$}};
\node[w] (v2) at (3, 2.3) {};
\node[vt] (lp1) at (2.2, 3.5) {};
\draw (2.2,4) node[anchor=center] {{\small $p+1$}};
\draw (3, 3.3) node[anchor=center] {{\small $\dots$}};
\node[vt] (lq) at (3.8, 3.5) {};
\draw (3.8,4) node[anchor=center] {{\small $q$}};
\node[vt] (lq1) at (4, 2) {};
\draw (4.2,2.5) node[anchor=center] {{\small $q+1$}};
\draw (4.5, 1.8) node[anchor=center] {{\small $\dots$}};
\node[vt] (lk) at (6, 2) {};
\draw (6,2.5) node[anchor=center] {{\small $k$}};
\node[w] (v1) at (3, 1) {};
\node[vt] (r) at (3, 0) {};
\draw [dashed] (r) edge (v1);
\draw [dashed] (v1) edge (v2);
\draw [dashed] (v1) edge (l1);
\draw [dashed] (v1) edge (lp);
\draw [dashed] (v2) edge (lp1);
\draw [dashed] (v2) edge (lq);
\draw [dashed] (v1) edge (lq1);
\draw [dashed] (v1) edge (lk);
\end{tikzpicture}
\caption{ The value of $\cD_{\As}$ on $\st^{\mo}_{k}$} \label{fig:cDAs-bt-k}
\end{figure}

The value of $\cD_{\Lie}$ on the corolla $\st^{\mo}_{n,k}$
is given by the sum depicted in figure \ref{fig:cDLie} and
the value of $\cD_{\As}$ on the corolla $\st^{\mo}_{n,k}$
is given by the sum depicted in figure \ref{fig:cDAs-bt-nk}.
The values $\cD'(\st^{\mo}_{n,k})$ and $\cD''(\st^{\mo}_{n,k})$
for $n \ge 2$
are defined in figures \ref{fig:cDpr}  and  \ref{fig:cDprpr}, 
respectively. Finally, for the corollas $\st^{\mo}_{1,k}$
we have 
\begin{equation}
\label{cDpr-prpr-zero}
\cD' (\st^{\mo}_{1,k}) =  
\cD''(\st^{\mo}_{1,k}) = 0\,, \qquad \forall ~ k \ge 0\,.
\end{equation}

\begin{figure}[htp] 
\centering 
\begin{tikzpicture}[scale=0.5]
\tikzstyle{vt}=[circle, draw, fill, minimum size=0, inner sep=1]
\tikzstyle{w}=[circle, draw, minimum size=3, inner sep=1]
\draw (-6,1.5) node[anchor=center] 
{{$\displaystyle \cD_{\Lie} (\st^{\mo}_{n,k}) \quad = \quad
(-1)^k \quad
 \sum_{p=2}^{n-1} \sum_{\tau\in \Sh_{p, n-p} } $}};
\node[vt] (l1) at (0, 4) {};
\draw (0,4.6) node[anchor=center] {{\small $\tau(1)$}};
\draw (1.5,3.8) node[anchor=center] {{\small $\dots$}};
\node[vt] (lp) at (2.5, 4) {};
\draw (2.5,4.6) node[anchor=center] {{\small $\tau(p)$}};
\node[w] (v2) at (1.3, 2.6) {};
\node[vt] (lp1) at (3.7, 2.5) {};
\draw (3.7,3.1) node[anchor=center] {{\small $\tau(p+1)$}};
\draw (5,2.3) node[anchor=center] {{\small $\dots$}};
\node[vt] (ln) at (6, 2.5) {};
\draw (6,3.1) node[anchor=center] {{\small $\tau(n)$}};
\node[vt] (l1mo) at (7.5, 2.5) {};
\draw (7.5,3.1) node[anchor=center] {{\small $1$}};
\draw (8.5,2.3) node[anchor=center] {{\small $\dots$}};
\node[vt] (lkmo) at (9.5, 2.5) {};
\draw (9.5,3.1) node[anchor=center] {{\small $k$}};
\node[w] (v1) at (6, 1) {};
\node[vt] (r) at (6, -0.3) {};
\draw [dashed] (r) edge (v1);
\draw (v1) edge (v2);
\draw (v2) edge (l1);
\draw (v2) edge (lp);
\draw (v1) edge (lp1);
\draw (v1) edge (ln);
\draw [dashed] (v1) edge (l1mo);
\draw [dashed] (v1) edge (lkmo);
\end{tikzpicture}
\caption{ The value of $\cD_{\Lie}$ on $\st^{\mo}_{n,k}$} \label{fig:cDLie}
\end{figure}

\begin{figure}[htp]
\begin{minipage}[t]{\linewidth}
\centering 
\begin{tikzpicture}[scale=0.5]
\tikzstyle{vt}=[circle, draw, fill, minimum size=0, inner sep=1]
\tikzstyle{w}=[circle, draw, minimum size=3, inner sep=1]
\draw (-8,1.5) node[anchor=center] 
{{$\displaystyle \cD_{\As} (\st^{\mo}_{n,k}) \quad = \quad
- \sum_{p=0}^{k-2} \sum_{q = p+2}^{k}  \quad (-1)^{p + (k-q) (q-p)} $}};
\node[vt] (lmc1) at (0, 3) {};
\draw (0,3.5) node[anchor=center] {{\small $1$}};
\draw (0.9,3) node[anchor=center] {{\small $\dots$}};
\node[vt] (lmcn) at (1.5, 3) {};
\draw (1.5, 3.5) node[anchor=center] {{\small $n$}};

\node[vt] (l1) at (2.3, 3) {};
\draw (2.3, 3.5) node[anchor=center] {{\small $1$}};
\draw (3.2,3) node[anchor=center] {{\small $\dots$}};
\node[vt] (lp) at (4, 3) {};
\draw (4, 3.5) node[anchor=center] {{\small $p$}};
\node[vt] (lp1) at (4, 4.5) {};
\draw (4, 5) node[anchor=center] {{\small $p+1$}};
\draw (5,4.5) node[anchor=center] {{\small $\dots$}};
\node[vt] (lq) at (6, 4.5) {};
\draw (6, 5) node[anchor=center] {{\small $q$}};
\node[vt] (lq1) at (6.5, 3) {};
\draw (6.7, 3.5) node[anchor=center] {{\small $q+1$}};
\draw (7.5,3) node[anchor=center] {{\small $\dots$}};
\node[vt] (lk) at (8.5, 3) {};
\draw (8.5, 3.5) node[anchor=center] {{\small $k$}};
\node[w] (v1) at (5, 1.5) {};
\node[w] (v2) at (5, 3) {};
\node[vt] (r) at (5, 0) {};
\draw [dashed] (r) edge (v1);
\draw  (v1) edge (lmc1);
\draw  (v1) edge (lmcn);
\draw  [dashed]   (v1) edge (l1);
\draw  [dashed]   (v1) edge (lp);
\draw  [dashed]   (v1) edge (v2);
\draw  [dashed]   (v2) edge (lp1);
\draw  [dashed]   (v2) edge (lq);
\draw  [dashed]   (v1) edge (lq1);
\draw  [dashed]   (v1) edge (lk);
\end{tikzpicture}
\end{minipage}
\begin{minipage}[t]{\linewidth}
~\\[0.3cm]
\end{minipage}
\begin{minipage}[t]{\linewidth}
\centering 
\begin{tikzpicture}[scale=0.5]
\tikzstyle{vt}=[circle, draw, fill, minimum size=0, inner sep=1]
\tikzstyle{w}=[circle, draw, minimum size=3, inner sep=1]
\draw (-8,1.5) node[anchor=center] 
{{$ \displaystyle  - \quad \sum_{p=0}^{k-2} \sum_{q = p+2}^{k}  \quad (-1)^{p + (k-q) (q-p)} $}};
\node[vt] (lmc1) at (2, 5) {};
\draw (2,5.5) node[anchor=center] {{\small $1$}};
\draw (3,5) node[anchor=center] {{\small $\dots$}};
\node[vt] (lmcn) at (3.7, 5) {};
\draw (3.7, 5.5) node[anchor=center] {{\small $n$}};
\node[vt] (l1) at (0, 3) {};
\draw (0, 3.5) node[anchor=center] {{\small $1$}};
\draw (1.3,3) node[anchor=center] {{\small $\dots$}};
\node[vt] (lp) at (2, 3) {};
\draw (2, 3.5) node[anchor=center] {{\small $p$}};
\node[vt] (lp1) at (5, 5) {};
\draw (5, 5.5) node[anchor=center] {{\small $p+1$}};
\draw (6,5) node[anchor=center] {{\small $\dots$}};
\node[vt] (lq) at (7, 5) {};
\draw (7, 5.5) node[anchor=center] {{\small $q$}};
\node[vt] (lq1) at (6.5, 3) {};
\draw (6.7, 3.5) node[anchor=center] {{\small $q+1$}};
\draw (7.5,3) node[anchor=center] {{\small $\dots$}};
\node[vt] (lk) at (8.5, 3) {};
\draw (8.5, 3.5) node[anchor=center] {{\small $k$}};
\node[w] (v1) at (5, 1.5) {};
\node[w] (v2) at (5, 3) {};
\node[vt] (r) at (5, 0) {};
\draw [dashed] (r) edge (v1);
\draw  (v2) edge (lmc1);
\draw  (v2) edge (lmcn);
\draw  [dashed]   (v1) edge (l1);
\draw  [dashed]   (v1) edge (lp);
\draw  [dashed]   (v1) edge (v2);
\draw  [dashed]   (v2) edge (lp1);
\draw  [dashed]   (v2) edge (lq);
\draw  [dashed]   (v1) edge (lq1);
\draw  [dashed]   (v1) edge (lk);
\end{tikzpicture}
\end{minipage}
\caption{ The value of $\cD_{\As}$ on $\st^{\mo}_{n,k}$} \label{fig:cDAs-bt-nk}
\end{figure}

\begin{figure}[htp] 
\centering 
\begin{tikzpicture}[scale=0.5]
\tikzstyle{vt}=[circle, draw, fill, minimum size=0, inner sep=1]
\tikzstyle{w}=[circle, draw, minimum size=3, inner sep=1]
\draw (-4,1.5) node[anchor=center] 
{{$\displaystyle \cD' (\st^{\mo}_{n,k}) \quad = \quad  (-1)^k \quad$}};
\node[vt] (l1) at (0, 4) {};
\draw (0,4.6) node[anchor=center] {{\small $1$}};
\node[vt] (l2) at (1, 4) {};
\draw (1,4.6) node[anchor=center] {{\small $2$}};
\draw (2,3.8) node[anchor=center] {{\small $\dots$}};
\node[vt] (ln) at (3, 4) {};
\draw (3,4.6) node[anchor=center] {{\small $n$}};

\node[w] (v2) at (1.8, 2.3) {};

\node[vt] (l1mo) at (3.5, 2.5) {};
\draw (3.5, 3.1) node[anchor=center] {{\small $1$}};
\node[vt] (l2mo) at (5, 2.5) {};
\draw (5, 3.1) node[anchor=center] {{\small $2$}};
\draw (6,2.3) node[anchor=center] {{\small $\dots$}};
\node[vt] (lkmo) at (8, 2.5) {};
\draw (8,3.1) node[anchor=center] {{\small $k$}};
\node[w] (v1) at (4, 1) {};
\node[vt] (r) at (4, -0.3) {};
\draw [dashed] (r) edge (v1);
\draw (v1) edge (v2);
\draw (v2) edge (l1);
\draw (v2) edge (l2);
\draw (v2) edge (ln);
\draw [dashed] (v1) edge (l1mo);
\draw [dashed] (v1) edge (l2mo);
\draw [dashed] (v1) edge (lkmo);
\end{tikzpicture}
\caption{ The value of $\cD'$ on $\st^{\mo}_{n,k}$ for $n \ge 2$} \label{fig:cDpr}
\end{figure}

\begin{figure}[htp] 
\centering 
\begin{tikzpicture}[scale=0.5]
\tikzstyle{vt}=[circle, draw, fill, minimum size=0, inner sep=1]
\tikzstyle{w}=[circle, draw, minimum size=3, inner sep=1]
\draw (-10,1.5) node[anchor=center] 
{{$\displaystyle \cD'' (\st^{\mo}_{n,k}) \quad = \quad
- \quad \sum_{r = 1}^{n-1} \sum_{\si \in \Sh_{r, n-r}} \sum_{0 \le p \le q  \le k} 
\quad (-1)^{p + (k-q)(q-p)} $}};
\node[vt] (mc1) at (0, 2) {};
\draw (0,2.6) node[anchor=center] {{\small $\si(1)$}};
\draw (1.2,2) node[anchor=center] {{\small $\dots$}};
\node[vt] (mcr) at (2, 2) {};
\draw (2,2.6) node[anchor=center] {{\small $\si(r)$}};
\node[vt] (mo1) at (3, 2) {};
\draw (3,2.6) node[anchor=center] {{\small $1$}};
\draw (4.1,2) node[anchor=center] {{\small $\dots$}};
\node[vt] (mop) at (5, 2) {};
\draw (5,2.6) node[anchor=center] {{\small $p$}};
\node[w] (v2) at (6, 3) {};
\node[vt] (moq1) at (7.5, 2) {};
\draw (7.5,2.6) node[anchor=center] {{\small $q+1$}};
\draw (8.3,2) node[anchor=center] {{\small $\dots$}};
\node[vt] (mok) at (9.5, 2) {};
\draw (9.5,2.6) node[anchor=center] {{\small $k$}};
\node[vt] (mcr1) at (3, 4.5) {};
\draw (3,5.1) node[anchor=center] {{\small $\si(r+1)$}};
\draw (4.3,4.5) node[anchor=center] {{\small $\dots$}};
\node[vt] (mcn) at (5.5, 4.5) {};
\draw (5.3,5.1) node[anchor=center] {{\small $\si(n)$}};
\node[vt] (mop1) at (7, 4.5) {};
\draw (7.2,5.1) node[anchor=center] {{\small $p+1$}};
\draw (8,4.5) node[anchor=center] {{\small $\dots$}};
\node[vt] (moq) at (9, 4.5) {};
\draw (9,5.1) node[anchor=center] {{\small $q$}};
\node[w] (v1) at (5, 0.5) {};
\node[vt] (r) at (5, -1) {};
\draw [dashed] (r) edge (v1);
\draw (v1) edge (mc1);
\draw (v1) edge (mcr);
\draw [dashed] (v1) edge (mo1);
\draw [dashed] (v1) edge (mop);
\draw [dashed] (v1) edge (v2);
\draw [dashed] (v1) edge (moq1);
\draw [dashed] (v1) edge (mok);
\draw (v2) edge (mcr1);
\draw (v2) edge (mcn);
\draw  [dashed] (v2) edge (mop1);
\draw  [dashed] (v2) edge (moq);
\end{tikzpicture}
\caption{ The value of $\cD''$ on $\st^{\mo}_{n,k}$ for $n \ge 2$} \label{fig:cDprpr}
\end{figure}

Direct computations show that 
\begin{eqnarray}
\label{cDAs-square}
\big(\cD_{\As}\big)^2 = 0\,, \\[0.3cm]
\label{cDLie-square}
\big(\cD_{\Lie} + \cD' \big)^2 = 0\,, \\[0.3cm]
\label{cDAsLie-commute}
\cD_{\As} \circ (\cD_{\Lie} + \cD' )  + 
(\cD_{\Lie} + \cD') \circ \cD_{\As}  = 0\,, \\[0.3cm]
\label{cDLie-cDprpr}
(\cD_{\Lie} + \cD' )   \circ \cD''
+ \cD'' \circ
(\cD_{\Lie} + \cD')  = 0\,, \\[0.3cm]
\label{cDAs-cDprpr}
\cD_{\As}  \circ \cD''
+ \cD'' \circ \cD_{\As}   +  \cD''  \circ \cD'' = 0\,.
\end{eqnarray}

\begin{remark}
\label{rem:OC-cobar}
It is not hard to see that the differential $\cD$ on 
$\Op(\oc)$ defines on $\bs^{-1} \oc$ a structure of $2$-colored 
pseudo-cooperad.  Thus, if $\oc^{\vee}$ is the $2$-colored 
cooperad obtained from $\bs^{-1}\oc$ via formally adjoining
the counit, then\footnote{This fact was also observed
in \cite[Section 4.1]{SC}.} 
\begin{equation}
\label{OC-cobar}
\OC = \Cobar(\oc^{\vee})\,.
\end{equation}

We remark that
\begin{equation}
\label{oc-vee-mc}
\oc^{\vee}(n,0)^{\mc} = \La^2 \coCom(n)
\end{equation}
and
\begin{equation}
\label{oc-vee-mo}
\oc^{\vee}(0,k)^{\mo} = \La \coAs(k)\,.
\end{equation}
\end{remark}

\subsection{$\OC$-algebras}
As we stated above, an $\OC$-algebra is a
pair of  cochain complexes $(\cV, \cA)$ with 
the following data:
 
\begin{itemize}
\item A $\La\Lie_{\infty}$-structure on $\cV$, 

\item an $A_{\infty}$-structure on $\cA$, and 

\item a $\La\Lie_{\infty}$-morphism from $\cV$ to the 
Hochschild cochain complex $\Cbu(\cA)$ of $\cA$\,.

\end{itemize} 

Let us briefly recall how to get the above data from 
an operad morphism 
\begin{equation}
\label{OC-to-End}
\OC \to \End_{(\cV, \cA)}\,. 
\end{equation}

The desired $\La\Lie_{\infty}$ structure on $\cV$ 
$$
Q : \La^2\coCom_{\circ}(\cV) \to \cV
$$
comes from the action of corollas $\st^{\mc}_n$ 
in figure \ref{fig:bt-n-oc-mc} for $n \ge 2$. Namely, 
\begin{equation}
\label{Q-bt-mc-n}
Q (v_1, \dots, v_n) = \st^{\mc}_n (v_1, \dots, v_n)\,,
\end{equation}
where $v_1, \dots, v_n \in \cV$\,.

The desired $A_{\infty}$-structure 
$$
m : \La\coAs_{\circ}(\cA) \to \cA
$$
comes from the action of corollas $\st^{\mo}_k$ 
in figure \ref{fig:bt-k-oc-mo} for $k \ge 2$. Namely,  
\begin{equation}
\label{m-bt-mo-k}
m (a_1, \dots, a_k) = (-1)^{\ve (a_1, \dots, a_k)}
\st^{\mo}_k (a_1, \dots, a_k)\,,
\end{equation}
where $a_1, \dots, a_k \in \cA$ and 
$$
\ve (a_1, \dots, a_k) = |a_1| (k-1) + |a_2| (k-2) + \dots + |a_{k-1}|\,. 
$$

Finally the action of corollas $\st^{\mo}_{n,k}$ gives 
us the desired  $\La\Lie_{\infty}$-morphism from 
$\cV$ to $\Cbu(\cA)$
$$
U :   \La^2\coCom(\cV) \otimes  T(\bs^{-1}\cA) \to \cA\,.
$$
Namely, 
\begin{equation}
\label{U-bt-mo-nk}
U  (v_1, \dots, v_n ; a_1, \dots, a_k) = 
 (-1)^{\ve' (v_1, \dots, v_n ; a_1, \dots, a_k)} 
 \st^{\mo}_{n,k}  (v_1, \dots, v_n ; a_1, \dots, a_k)\,,
\end{equation}
where
\begin{equation}
\label{ve-pr}
\ve' (v_1, \dots, v_n ; a_1, \dots, a_k) = 
k (|v_1| + \dots + |v_n|) + 
|a_1| (k-1) + |a_2| (k-2) + \dots + |a_{k-1}|\,. 
\end{equation}

%
%
\section{Stable formality quasi-isomorphisms and their homotopies}
\label{sec:stable}

Several vectors of $\KGra$ will play a special role in 
the definition of a stable formality quasi-isomorphism (SFQ) and 
in further considerations. These are
\begin{equation}
\label{binary}
\G_{\ed} =   \begin{tikzpicture}[scale=0.5, >=stealth']
\tikzstyle{w}=[circle, draw, minimum size=4, inner sep=1]
\tikzstyle{b}=[circle, draw, fill, minimum size=4, inner sep=1]
\node [b] (b1) at (0,0) {};
\draw (0,0.6) node[anchor=center] {{\small $1_{\mc}$}};
\node [b] (b2) at (1.5,0) {};
\draw (1.6,0.6) node[anchor=center] {{\small $2_{\mc}$}};
\draw (b1) edge (b2);
\end{tikzpicture}
%
%
%
%
%
, \qquad \qquad
\G_{\ww} =   \begin{tikzpicture}[scale=0.5, >=stealth']
\tikzstyle{w}=[circle, draw, minimum size=4, inner sep=1]
\tikzstyle{b}=[circle, draw, fill, minimum size=4, inner sep=1]
\node [w] (w1) at (0,0) {};
\draw (0,0.6) node[anchor=center] {{\small $1_{\mo}$}};
\node [w] (w2) at (1.5,0) {};
\draw (1.6,0.6) node[anchor=center] {{\small $2_{\mo}$}};
\draw (w1) (w2);
\end{tikzpicture}
\end{equation}
and the series of ``brooms'' for $k \ge 0$ depicted in 
figure \ref{fig:brooms}.
\begin{figure}[htp]
\centering 
\begin{tikzpicture}[scale=0.5, >=stealth']
\tikzstyle{w}=[circle, draw, minimum size=4, inner sep=1]
\tikzstyle{b}=[circle, draw, fill, minimum size=4, inner sep=1]
\draw (-3,1.5) node[anchor=center] {{$\G^{\br}_{k}  \quad = \quad $}};
\node [b] (v) at (2.5,3) {};
\node [w] (v1) at (0,0) {};
\node [w] (v2) at (1,0) {};
\node [w] (vk) at (4,0) {};
\draw (2.5,0.7) node[anchor=center] {{\small $\dots$}};
\draw [->] (v) edge (1,1.2);
\draw  (1,1.2) edge (v1);
\draw [->] (v) edge (1.5,1);
\draw  (1.5,1) edge (v2);
\draw [->] (v) edge (3.5,1);
\draw  (3.5,1) edge (vk);
\draw (2.6,3.6) node[anchor=center] {{\small $1_{\mc}$}};
\draw (0.1,-0.6) node[anchor=center] {{\small $1_{\mo}$}};
\draw (1.1,-0.6) node[anchor=center] {{\small $2_{\mo}$}};
\draw (4.1,-0.6) node[anchor=center] {{\small $k_{\mo}$}};
\end{tikzpicture}
~\\[0.3cm]
\caption{Edges are ordered in this way 
$(1_{\mc}, 1_{\mo}) < (1_{\mc}, 2_{\mo}) < \dots < (1_{\mc}, k_{\mo})$
} \label{fig:brooms}
\end{figure} 

Note that the graph $\G^{\br}_{0} \in \KGra(1,0)^{\mo}$ 
consists of a single black vertex labeled by $1_{\mc}$ and it has no edges.  
 
According to Section \ref{sec:acts-on-V-A},
the 2-colored operad $\KGra$ acts on the pair $(V_A, A)$
where $A$  (resp. $V_A$) is the algebra of functions (resp. the algebra 
of polyvector fields) on a graded affine space.
Hence, every morphism of operads (in $\Ch_{\bbK}$) $F :  \OC \to \KGra$ 
gives us a $\La\Lie_{\infty}$-structure on 
$V_A$, an $A_{\infty}$-structure on $A$ and an 
$\La\Lie_{\infty}$-morphism from $V_A$ to the Hochschild cochain complex $\Cbu(A)$ of $A$.
Moreover, this construction works for a graded affine space of {\it any dimension.}
This observation motivates the following definition. 
%
\begin{defi} 
\label{dfn:stable}
A stable formality quasi-isomorphism (SFQ) is a morphism of 
2-colored operads in the category of cochain complexes
\begin{equation}
\label{eq:stable}
F :  \OC \to \KGra
\end{equation}
satisfying the following ``boundary conditions'': 
\begin{equation}
\label{Schouten}
F(\st^{\mc}_n) = \begin{cases}
   \G_{\ed} \qquad {\rm if} ~~ n = 2\,,  \\
   0 \qquad {\rm if} ~~  n \ge 3\,,
\end{cases}
\end{equation}
\begin{equation}
\label{mult}
F(\st^{\mo}_2) = \G_{\ww}\,, 
\end{equation}
and
\begin{equation}
\label{HKR}
F(\st^{\mo}_{1,k}) = \frac{1}{k!} \G^{\br}_{k}\,,
\end{equation}
where $\st^{\mc}_n$, $\st^{\mo}_k$, and $\st^{\mo}_{n,k}$ are corollas
depicted 
in figures \ref{fig:bt-n-oc-mc}, \ref{fig:bt-k-oc-mo}, \ref{fig:bt-nk-oc-mo}, 
respectively, and   $\G_{\ed}$, $\G_{\ww}$ and $\G^{\br}_k$ are the vectors 
of $\KGra$ specified in the beginning of this section.
\end{defi}
 
To interpret the ``boundary  conditions'' we consider the 
OCHA structure induced by the morphism $F$ \eqref{eq:stable}
on the pair  $(V_A, A)$\,.

The first condition (eq. \eqref{Schouten})  implies that 
the $\La\Lie_{\infty}$-structure on polyvector fields induced by the 
morphism $F$  coincides with the standard Schouten-Nijenhuis algebra 
structure.  

The second condition (eq. \eqref{mult}) implies that the binary 
operation of the induced $A_{\infty}$-structure on $A$ coincides
with the ordinary (commutative) multiplication.  For degree reasons,  
the image $F(\st^{\mo}_{k})$ of the corolla $\st^{\mo}_{k}$ in $\KGra(0,k)^{\mo}$
is zero for all $k \ge 3$\,. Thus the induced $A_{\infty}$-structure on $A$
coincides with the original associative (and commutative) algebra structure.

The third boundary condition  (eq. \eqref{HKR})  implies that the corresponding 
$\La\Lie_{\infty}$-morphism from $V_A$ to $\Cbu(A)$ starts with the 
Hochschild-Kostant-Rosenberg embedding. The latter condition guarantees that
the induced $\La\Lie_{\infty}$-morphism is a quasi-iso\-mor\-phism.   

\begin{remark}  
\label{rem:q-iso}
It should be mentioned that the map in \eqref{eq:stable} is never a quasi-isomorphism
of dg operads. Indeed, the restriction of any morphism of dg operads $F :  \OC \to \KGra$
satisfying the above ``boundary conditions''
to the spaces $\OC(n,0)^{\mc}$ (for $n \ge 0$) gives us the morphism of operads
$\La\Lie_{\infty} \to \dGra$ which coincides with the composition
of the canonical quasi-isomorphism $\La\Lie_{\infty} \stackrel{\sim}{\longrightarrow} \La\Lie$ and the standard 
embedding of operads $\La\Lie \hookrightarrow \dGra$ \cite[Section 7.1]{notes}. 
This composition is not a quasi-isomorphism because the embedding 
$\La\Lie \hookrightarrow \dGra$ is not onto. 
\end{remark}

%
%
\subsection{SFQs as MC elements. Homotopies of SFQs}
\label{sec:sfqi-homot}

Due to Proposition \ref{prop:cobar-conv} and Remark \ref{rem:OC-cobar}, SFQs are in bijection 
with MC elements $\al$ of the Lie algebra 
\begin{equation}
\label{Conv-oc-KGra}
\Conv(\oc^{\vee}_{\circ},  \KGra)
\end{equation}
subject to the three conditions
\begin{equation}
\label{al-Schouten}
\al(\bsi \st^{\mc}_n) = \begin{cases}
   \G_{\ed} \qquad {\rm if} ~~ n = 2\,,  \\
   0 \qquad {\rm if} ~~  n \ge 3\,,
\end{cases}
\end{equation}
\begin{equation}
\label{al-mult}
\al(\bsi \st^{\mo}_2) = \G_{\ww}\,, 
\end{equation}
and
\begin{equation}
\label{al-HKR}
\al(\bsi \st^{\mo}_{1,k}) = \frac{1}{k!} \G^{\br}_{k}\,,
\end{equation}
where $\st^{\mc}_n$, $\st^{\mo}_k$, and $\st^{\mo}_{n,k}$ are corollas
depicted in figures \ref{fig:bt-n-oc-mc}, \ref{fig:bt-k-oc-mo}, \ref{fig:bt-nk-oc-mo}, 
respectively, and   $\G_{\ed}$, $\G_{\ww}$ and $\G^{\br}_k$ are the vectors 
of $\KGra$ specified in the beginning of this section.

We would like to remark that, since all vectors in $\KGra(0,k)^{\mo}$
have degree zero, we have 
\begin{equation}
\label{al-mult-higher}
\al(\bs^{-1}\, \st^{\mo}_k) = 0\,, 
\end{equation}
for all $k\ge 3$ and for all degree $1$ elements 
$\al$ in \eqref{Conv-oc-KGra}.

In what follows, we denote by $\al_{F}$ the MC element in \eqref{Conv-oc-KGra}
corresponding to an SFQ $F$\,.

According to Section \ref{sec:Conv}, the Lie algebra 
$\Conv(\oc^{\vee}_{\circ}, \KGra)$ is equipped with the
``arity'' filtration $\cF_{\bul} \Conv(\oc^{\vee}_{\circ}, \KGra)$  
such that $\Conv(\oc^{\vee}_{\circ}, \KGra)$ is complete 
with respect to this filtration. 

Hence, following general theory from Appendix \ref{app:MC},
the set of MC elements of the Lie algebra  \eqref{Conv-oc-KGra} 
is equipped with the action of the pro-unipotent group 
\begin{equation}
\label{G-Conv}
\exp \Big(\cF_1 \Conv(\oc^{\vee}_{\circ}, \KGra)^0 \Big)\,.
\end{equation}

We claim that 
%
%
\begin{prop}
\label{prop:boundary-homot}
Degree zero vectors
\begin{equation}
\label{xi-here}
\xi \in \Conv(\oc^{\vee}_{\circ}, \KGra)
\end{equation}
satisfying the ``boundary'' condition
\begin{equation}
\label{xi-cond-mc}
\xi(\bsi \st^{\mc}_n) =  0 \qquad \forall ~~ n \ge 2
\end{equation}
form a Lie subalgebra of $\cF_1 \Conv(\oc^{\vee}_{\circ}, \KGra)^0$. 
Moreover, if $\al$ is a MC element of  the Lie algebra 
$\Conv(\oc^{\vee}_{\circ}, \KGra)$ satisfying the boundary 
conditions \eqref{al-Schouten}, \eqref{al-mult}, 
\eqref{al-HKR} and $\xi$ is a 
degree zero vector  \eqref{xi-here} 
satisfying \eqref{xi-cond-mc} then the MC element 
\begin{equation}
\label{al-pr}
\al' = \exp(\xi)\, (\al) 
\end{equation}
also satisfies conditions \eqref{al-Schouten}, \eqref{al-mult}, \eqref{al-HKR}.
\end{prop}
\begin{proof}
First, we observe that every degree zero vector \eqref{xi-here} satisfies 
\begin{equation}
\label{xi-cond-HKR}
\xi(\bsi \st^{\mo}_{1,k}) =  0 \qquad \forall ~~ k \ge 0.
\end{equation}
Indeed, since the vector $\bsi \st^{\mo}_{1,k}$ has degree $-k-1$, 
$\xi(\bsi \st^{\mo}_{1,k}) $ must be a linear combination of graphs with $1$ black 
vertex, $k$ white vertices, and exactly $k+1$ edges. 
Since an edge cannot originate at any white vertex, multiple edges with the same direction 
and loops are not allowed, the set of such graphs is empty. 
  
Similarly, since all vectors in $\KGra(0,k)^{\mo}$ have degree zero, 
we conclude that 
\begin{equation}
\label{xi-on-assoc}
\xi(\bsi \st^{\mo}_{k}) = 0 \qquad \forall ~~ k \ge 2
\end{equation}
for any degree zero vector \eqref{xi-here}.

The inclusion 
\begin{equation}
\label{xi-F1}
\xi \in \cF_1 \Conv(\oc^{\vee}_{\circ}, \KGra)\,.
\end{equation}
follows immediately from equation \eqref{xi-cond-HKR} for $k=0$\,. 

Moreover, the vector $[\xi_1, \xi_2]$ satisfies 
condition \eqref{xi-cond-mc} if so do both $\xi_1$ and 
$\xi_2$\,. Thus the first statement of the proposition is proved. 

Using \eqref{xi-cond-mc}, \eqref{xi-cond-HKR}, and 
\eqref{xi-on-assoc}, it is easy to see that  $\al'$ in 
\eqref{al-pr} satisfies conditions \eqref{al-Schouten}, \eqref{al-mult}, \eqref{al-HKR} if so does $\al$\,. 
\end{proof}

We can now give the definition of homotopy between two SFQs:
\begin{defi}
\label{dfn:homotopy}
We say that an SFQ $F$ \eqref{eq:stable}
is homotopy equivalent to $\wt{F}$ if the corresponding MC elements 
$$
\al_{F}, \al_{\wt{F}}  \in \Conv (\oc^{\vee}_{\circ}, \KGra)
$$
are isomorphic via $\exp(\xi)$, where 
$\xi$  is a degree zero element in  $\cF_1\Conv (\oc^{\vee}_{\circ}, \KGra)$
satisfying condition \eqref{xi-cond-mc}.
\end{defi}
\begin{remark}
\label{rem:indeed-equiv}
Proposition \ref{prop:boundary-homot}
implies that the resulting relation on the set of SFQs is indeed an 
equivalence relation. 
\end{remark}
\begin{remark}
\label{rem:al-in-cF0}
Since $\Conv(\oc^{\vee}_{\circ}, \KGra) = \cF_0 \Conv(\oc^{\vee}_{\circ}, \KGra)$, 
every MC element $\al$ of \eqref{Conv-oc-KGra} satisfies the condition 
\begin{equation}
\label{al-in-cF0}
\al \in \cF_0 \Conv(\oc^{\vee}_{\circ}, \KGra).
\end{equation}
On the other hand, it is not true that a MC element $\al$ corresponding to an SFQ
belongs to $\cF_1 \Conv(\oc^{\vee}_{\circ}, \KGra)$. 
Indeed, according to \eqref{al-HKR}, we have $\al(\bsi \st^{\mo}_{1,0}) = \G^{\br}_0 \neq 0$. 

Using the ``boundary conditions''  \eqref{al-Schouten}, \eqref{al-mult} and \eqref{al-HKR} 
for MC elements $\al$ corresponding to SFQs, it is easy to see that 
\begin{equation}
\label{al-mc-0}
\al \in \cF^{\mc}_{0} \Conv(\oc^{\vee}_{\circ}, \KGra),
\end{equation} 
where the filtration $\cF^{\mc}_{\bul}$ is defined in \eqref{Conv-filtr-chi} (in Section \ref{sec:Conv}).
\end{remark}

To explain our motivation behind Definition \ref{dfn:homotopy}
we consider the pair $(V_A, A)$, where $A$ is a finitely generated 
free commutative algebra in $\grVect_{\bbK}$ and $V_{A}$ be the algebra of 
polyvector fields on the corresponding (graded) affine space. 

Recall that an SFQ $F$ gives us a $\La\Lie_{\infty}$ quasi-isomorphism $U_F$ from 
$V_A$ to $\Cbu(A)$ which admits a graphical expansion. 

We claim that, if two SFQs $F$ and $\wt{F}$ are homotopy equivalent, then 
the corresponding $\La\Lie_{\infty}$-morphisms $U_F$ and $U_{\wt{F}}$ are also homotopy 
equivalent. Furthermore, the homotopy between $U_F$ and $U_{\wt{F}}$
admits a graphical expansion.

Indeed, according to \cite[Lemma 2.9]{HAform} or \cite[Section 1.3]{Dots-Poncin}, 
any $\La\Lie_{\infty}$-morphism
$$
U : V_A \leadsto \Cbu(A)
$$
is a MC element of the following auxiliary Lie algebra 
\begin{equation}
\label{aux-Lie}
\bs \Hom (\bs^2 S(\bs^{-2} V_A), \Cbu(A))\,.
\end{equation}
For the definition of the differential and the Lie bracket on \eqref{aux-Lie}, 
see Section 2.1 in \cite{HAform}.  

Furthermore, following Definition 4.7 in \cite{HAform},
two $\La\Lie_{\infty}$-morphisms 
$$
U : V_A \leadsto \Cbu(A)\qquad 
\textrm{and} \qquad  
\wt{U} : V_A \leadsto \Cbu(A)
$$ 
are homotopy equivalent if and only if the corresponding 
MC elements in the Lie algebra \eqref{aux-Lie} are isomorphic. 
  
By merely unfolding definitions it is not hard to see that, if 
MC elements 
$$
\al_{F}, \al_{\wt{F}}  \in \Conv (\oc^{\vee}_{\circ}, \KGra)
$$
are isomorphic via $\exp(\xi)$ for a degree zero vector \eqref{xi-here} satisfying 
\eqref{xi-cond-mc}, then the MC elements in \eqref{aux-Lie}
corresponding to $U_{F}$ and $U_{\wt{F}}$ are 
isomorphic via 
$$
\exp(\xi')
$$
where $\xi'$ is the degree zero vector in \eqref{aux-Lie} given by the formula: 
\begin{equation}
\label{xi-pr}
\xi' (v_1, v_2, \dots, v_n; a_1, a_2, \dots, a_n) 
= 
\end{equation}
$$
(-1)^{\ve' (v_1, \dots, v_n ; a_1, \dots, a_k)} 
 \xi (\bs^{-1}\st^{\mo}_{n,k})(v_1, v_2, \dots, v_n; a_1, a_2, \dots, a_n)\,, 
$$
where $\ve'(v_1, \dots, v_n ; a_1, \dots, a_k)$ is defined in \eqref{ve-pr}.

\begin{example}
\label{ex:Kontsevich}
In his famous paper \cite{K} M. Kontsevich proposed 
a construction of a $\La\Lie_{\infty}$ quasi-isomorphism 
from the Lie algebra of polyvector fields $V_A$ on 
$\bbR^d$ to polydifferential operators on $\bbR^d$\,. 
The structure maps of this $\La\Lie_{\infty}$ quasi-isomorphism  
are defined using graphical expansion and the 
$\La\Lie_{\infty}$ quasi-isomorphism starts with the standard 
Hochschild-Kostant-Rosenberg embedding.
Thus Kontsevich's construction from \cite{K} gives 
us an SFQ over any extension of the field $\bbR$\,. 
For more details, we refer the reader to \cite[Section 2.4]{overQ}.
\end{example} 

%
%
\section{The action of Kon\-tse\-vich's graph comp\-lex on stable for\-ma\-li\-ty qua\-si-iso\-mor\-phisms}
\label{sec:dfGC}
It is possible to produce new homotopy types of stable 
formality quasi-isomorphisms using the action of 
Kon\-tse\-vich's graph comp\-lex on the Lie algebra 
$$
\Conv(\oc^{\vee}_{\circ}, \KGra)\,.
$$   
We describe this action here. 

\subsection{Reminder of the full directed graph complex $\dfGC$}
Let us consider the Lie algebra (in $\grVect_{\bbK}$) 
\begin{equation}
\label{dfGC}
\dfGC  =  \Conv(\La^2\coCom, \dGra)\,.
\end{equation}
Since 
$$
 \Conv(\La^2\coCom, \dGra) = 
\prod_{n=1}^{\infty} \bs^{2n-2} \big( \dGra(n) \big)^{S_n}
$$ 
vectors in \eqref{dfGC} are (possibly infinite) linear combinations 
$$
\ga = \sum_{n=1}^{\infty} \ga_n
$$
where $\ga_n$ is an $S_n$-invariant vector in $\dGra(n)$\,.

If all graphs in the linear combination $\ga_n \in (\dGra(n))^{S_n}$
have the same number of edges $e$ then $\ga_n$ is a homogeneous 
vector in $\dfGC$ of degree
\begin{equation}
\label{deg-in-dfGC}
|\ga_n| = 2n - 2 -e\,.
\end{equation}

For example, the vector $\G_{\ed} \in \dGra(2)$ defined in 
\eqref{binary} is $S_2$-invariant and hence is a vector in 
$\dfGC$\,. According to \eqref{deg-in-dfGC}, the vector 
$\G_{\ed}$ carries degree $1$\,. A direct computation 
shows that 
\begin{equation}
\label{G-ed-MC}
[\G_{\ed}, \G_{\ed}] = 0\,.
\end{equation}
Hence, $\G_{\ed}$ is a MC element and it can be used 
to equip the graded vector space \eqref{dfGC} with 
the non-zero differential 
\begin{equation}
\label{diff-dfGC}
\pa = \ad_{\G_{\ed}}\,.
\end{equation}

\begin{defi}
\label{dfn:dfGC}
The graded vector space $\dfGC$ \eqref{dfGC} with 
the differential \eqref{diff-dfGC} is called the full directed graph complex. 
\end{defi}

For example the graph $\G_{\bul} \in \dgra_1$ which consists of 
a single vertex without edges gives us a degree zero vector 
in $\dfGC$\,. According to the definition of the Lie bracket 
on $\dfGC$, we have 
\begin{equation}
\label{G-bul-NOT}
[\G_{\ed}, \G_{\bul}] = \G_{\ed} \circ_1 \G_{\bul} +
\si_{12}  (\G_{\ed} \circ_1 \G_{\bul}) - \G_{\bul} \circ_1 \G_{\ed} = 
\G_{\ed} + \G_{\ed} - \G_{\ed} = \G_{\ed}\,,
\end{equation}
where $\si_{12}$ is the transposition in $S_2$\,.

Thus $\G_{\bul}$ is not a cocycle in $\dfGC$\,. 

According to Section \ref{sec:Conv}, the Lie algebra 
$\dfGC$ is equipped with the descending filtration \eqref{Conv-filtr}
such that $\dfGC$ is complete with respect to this filtration. 
Unfolding  \eqref{Conv-filtr}, it is easy to see that $\cF_m \dfGC$ consists of sums
$$
\ga = \sum_{n=m+1}^{\infty} \ga_n\,, \qquad \ga_n \in  \big( \dGra(n) \big)^{S_n}\,.
$$
I.e. $\ga \in \cF_m \dfGC$ if and only if each graph in $\ga$ has $\ge m+1$ vertices. 
For example, $\G_{\ed} \in \cF_1 \dfGC$. Therefore the differential  \eqref{diff-dfGC}
is compatible with the filtration on $\dfGC$. 
 
Since loops are not allowed, $\G_{\bul}$ is the only element 
of $\dgra_1$. Therefore, since $\pa \G_{\bul} \neq 0$ and the differential 
$\pa$ raises the number of vertices up by $1$, 
every cocycle $\ga \in \dfGC$ has the property\footnote{Inclusion 
\eqref{ga-in-F1} no longer holds if we allow loops. Indeed, a simple computation 
shows that the single loop $\lp \in \dgra_1$ would be a cocycle in $\dfGC$. 
See \cite{dgraphs} for more details.}
\begin{equation}
\label{ga-in-F1}
\ga \in \cF_1 \dfGC\,.
\end{equation}
Thus, the Lie algebra $H^0(\dfGC)$ is pro-nilpotent. 

To give an example of a degree zero cocycle in $\dfGC$ we consider the tetrahedron in $\dGra(4)$ depicted 
in figure \ref{fig:tetra}. 
\begin{figure}[htp]
\centering 
\begin{tikzpicture}[scale=0.5, >=stealth']
\tikzstyle{ahz}=[circle, draw, fill=gray, minimum size=26, inner sep=1]
\tikzstyle{ext}=[circle, draw, minimum size=5, inner sep=1]
\tikzstyle{int}=[circle, draw, fill, minimum size=5, inner sep=1]
\node [int] (v1) at (2,0) {};
\draw (2,-0.6) node[anchor=center] {{\small $1$}};
\node [int] (v2) at (4,3) {};
\draw (4,3.6) node[anchor=center] {{\small $2$}};
\node [int] (v3) at (0,3) {};
\draw (0,3.6) node[anchor=center] {{\small $3$}};
\node [int] (v4) at (2,1.8) {};
\draw (2, 2.4) node[anchor=center] {{\small $4$}};
\draw (v1) edge (v2);
\draw (v1) edge (v3);
\draw (v1) edge (v4);
\draw (v2) edge (v3);
\draw (v2) edge (v4);
\draw (v3) edge (v4);
\end{tikzpicture}
~\\[0.3cm]
\caption{We may choose this order on the 
set of edges: $ (1,2) < (1,3) < (1,4) < (2,3)< (2,4) < (3,4)$ } \label{fig:tetra}
\end{figure} 
This graph is invariant with 
respect to the action of $S_4$ and hence it can be 
viewed as a vector in $\dfGC$\,. According to 
\eqref{deg-in-dfGC}, this vector has degree zero. 
A direct computation shows that it is a cocycle and it 
is easy to see that this cocycle is non-trivial.

In fact, it was proved in \cite[Proposition 9.1]{Thomas} that, for every 
odd number $n \ge 3$, there exists a non-trivial cocycle which has a non-zero 
coefficient in front of the  
wheel with $n$ spokes (see figure \ref{fig:wheel}).  
Note that, labels on vertices do not play an important role 
because vectors in $\dfGC$ are invariant under the action 
of the symmetric group. 
\begin{figure}[htp]
\centering 
\begin{tikzpicture}[scale=0.5, >=stealth']
\tikzstyle{ext}=[circle, draw, minimum size=5, inner sep=1]
\tikzstyle{int}=[circle, draw, fill, minimum size=5, inner sep=1]
\node [int] (v0) at (3,3) {};
\draw (3.7, 2.5) node[anchor=center] {{\small $n+1$}};
\node [int] (vn) at (0,3) {};
\draw (-0.5, 3) node[anchor=center] {{\small $n$}};
\node [int] (v1) at (0.5,5) {};
\draw (0.5, 5.6) node[anchor=center] {{\small $1$}};
\node [int] (v2) at (2,6) {};
\draw (2, 6.6) node[anchor=center] {{\small $2$}};
\node [int] (v3) at (4,6) {};
\draw (4, 6.6) node[anchor=center] {{\small $3$}};
\draw (5.4, 5.15) node[anchor=center, rotate=140] {{\small $\dots$}};
\draw (0.5, 1.5) node[anchor=center, rotate=120] {{\small $\dots$}};
\draw (vn) edge (0.25,2);
\draw (vn) edge (v1);
\draw (v1) edge (v2);
\draw (v2) edge (v3);
\draw (v3) edge (5,5.5);
\draw (vn) edge (v0);
\draw (v1) edge (v0);
\draw (v2) edge (v0);
\draw (v3) edge (v0);
\end{tikzpicture}
~\\[0.3cm]
\caption{Here $n$ is an odd integer $\ge 3$} \label{fig:wheel}
\end{figure} 

Using \cite[Theorem 1.1]{Thomas} and the ideas sketched in \cite[Appendix K]{Thomas}, 
one can prove the following statement:
\begin{thm}[\cite{dgraphs}, Corollary 3.6]
\label{thm:grt-dfGC}
For the full directed graph complex $\dfGC$, we have an 
isomorphism of Lie algebras 
\begin{equation}
\label{H-0-dfGC}
H^{0}(\dfGC) \cong  \grt_1, 
\end{equation}
where $\grt_1$ is the Grothendieck-Teichmueller Lie algebra \cite[Section 4.2]{AT}, \cite[Section 6]{Drinfeld}.
\end{thm}
\begin{remark}
\label{rem:no-pikes}
Let $\G$ be an element in $\dgra_n$\,. We say 
that a vertex $v$ of $\G$  is a {\it pike} if $v$ has valency 
$1$ and the edge adjacent to $v$ terminates at $v$\,.
We observe that, due to \cite[Proposition 3.5]{dgraphs} any 
cocycle $\ga$ in $\dfGC$ is cohomologous to a cocycle in which 
all graphs do not have pikes\footnote{Note that, in this paper, we work 
exclusively with the loopless version of the full directed graph complex. 
So what we denote by $\dfGC$ in this paper is denoted by $\dfGC^{\nl}$ in \cite{dgraphs}.}. 
\end{remark}

\subsection{The action of $\dfGC$ on SFQs}
\label{sec:the-action}

For our purposes it is convenient to extend the Lie 
algebra  $\Conv(\oc^{\vee}_{\circ}, \KGra)$ \eqref{Conv-oc-KGra}
to the Lie algebra 
\begin{equation}
\label{Conv-oc-KGra-more}
\Conv(\oc^{\vee}, \KGra)
\end{equation}
and view MC elements of $\Conv(\oc^{\vee}_{\circ}, \KGra)$ as MC elements of 
its extension $\Conv(\oc^{\vee}, \KGra)$.

Using equation \eqref{oc-vee-mc}, we define a natural 
embedding of $\dfGC$ \eqref{dfGC} into 
the Lie algebra  $\Conv(\oc^{\vee}, \KGra)$ 
\begin{equation}
\label{dfGC-Conv}
J : \dfGC \hookrightarrow  \Conv(\oc^{\vee}, \KGra)\,.
\end{equation}
This embedding is given by the formulas 
\begin{equation}
\label{J}
J (\ga)  \Big |_{\oc^{\vee}(n,0)^{\mc}} = \ga\,,
\qquad 
J (\ga) \Big |_{\oc^{\vee}(n,k)^{\mo}} = 0\,.  
\end{equation}

The embedding $J$ is obviously compatible with 
the Lie brackets and with the filtrations by arity 
on $\dfGC$ and $\Conv(\oc^{\vee}, \KGra)$ (see \eqref{Conv-filtr}).
However, we should point out that 
$J$ is not compatible with the differentials.  
Indeed, the Lie 
algebra  $\Conv(\oc^{\vee}, \KGra)$ carries the zero differential 
while $\dfGC$ carries the non-zero differential \eqref{diff-dfGC}.

Let $\al_{F}$ be a MC element in $\Conv(\oc^{\vee}_{\circ}, \KGra)$
corresponding to an SFQ. We claim that
\begin{prop}
\label{prop:dfGC-action}
For every degree zero cocycle $\ga \in \dfGC$
the equation 
\begin{equation}
\label{dfGC-action}
\al' = \exp(\ad_{J(\ga)}) \al_{F} 
\end{equation}
defines a MC element $\al'$ in $\Conv(\oc^{\vee}_{\circ}, \KGra)$
satisfying conditions \eqref{al-Schouten}, \eqref{al-mult}, 
and \eqref{al-HKR}\,. 
\end{prop}
\begin{proof}
It is obvious that $\al'$ satisfies the MC equation in 
$\Conv(\oc^{\vee}, \KGra)$.

Furthermore, since each cocycle in $\dfGC$ belongs to 
$\cF_1 \dfGC$, the MC element $\al'$ belongs to  the 
Lie subalgebra 
$$
\Conv(\oc^{\vee}_{\circ}, \KGra) \subset  \Conv(\oc^{\vee}, \KGra),.
$$

Next,  using the cocycle condition for $\ga$ 
$$
[\G_{\ed}, \ga] = 0
$$ 
it is not hard to show that $\al'$ satisfies condition
\eqref{al-Schouten}.  

Finally, it is straightforward to verify that $\al'$
also satisfies  \eqref{al-mult} and \eqref{al-HKR}.
\end{proof}

Due to Proposition \ref{prop:dfGC-action}, the group 
$\exp\big(\cZ^0(\dfGC)\big)$ acts on SFQs. In the following proposition 
we list important properties of this action.
%
%
\begin{prop}
\label{prop:to-homotopy}
Let $\ga$ be a degree zero cocycle in $\dfGC$\,. 
If $\al$ and $\wt{\al}$ are MC elements of \eqref{Conv-oc-KGra} corresponding to homotopy equivalent 
SFQs $F$ and $\wt{F}$, then the MC elements 
$$
 \exp(\ad_{J(\ga)}) \al\,, \quad \textrm{and} \quad 
 \exp(\ad_{J(\ga)}) \wt{\al}
$$
also correspond to homotopy equivalent SFQs.

Furthermore,  if 
\begin{equation}
\label{ga-cF-n}
\ga  = [\G_{\ed}, \psi] 
\end{equation}
then there exists a degree zero vector 
$$
\xi \in \cF_1 \Conv(\oc^{\vee}_{\circ}, \KGra)
$$
for which \eqref{xi-cond-mc} holds and 
\begin{equation}
\label{al-xi-ga}
 \exp(\ad_{J(\ga)}) \al = \exp(\ad_{\xi}) \al\,.
\end{equation}

If, in addition, 
\begin{equation}
\label{psi-in-cF-1n}
\psi \in \cF_{n-1} \dfGC
\end{equation}
for some $n \ge 2$ then the vector $\xi$ in \eqref{al-xi-ga} can 
be chosen in such a way that 
\begin{equation}
\label{xi-mc-n}
\xi \in  \cF^{\mc}_{n} \Conv(\oc^{\vee}_{\circ}, \KGra)
\end{equation}
and 
\begin{equation}
\label{xi-more}
\xi (\bsi \st^{\mo}_{n,0}) = \psi (1_n)\,,  
\end{equation}
where $1_n$ is the generator $\bs^{2-2n} 1 \in \bs^{2-2n} \bbK \cong \La^2 \coCom(n).$
\end{prop}
\begin{proof}
Since $\al$ and $\wt{\al}$ represent homotopy equivalent SFQs, there exists a degree zero vector
$$
\xi \in \cF_1 \Conv(\oc^{\vee}_{\circ}, \KGra)
$$
for which \eqref{xi-cond-mc} holds and 
\begin{equation}
\label{xi}
\wt{\al} = \exp(\ad_{\xi}) \al\,.
\end{equation}

Applying $\exp(\ad_{J(\ga)}) $ to both sides of equation \eqref{xi} we get 
\begin{equation}
\label{Jga-xi}
 \exp(\ad_{J(\ga)}) \wt{\al} =  \exp(\ad_{J(\ga)}) \exp(\ad_{\xi}) \al =
\end{equation}
$$
 \exp(\ad_{\wt{\xi}})  \left(   \exp(\ad_{J(\ga)}) \al  \right)\,,
$$
where 
\begin{equation}
\label{wt-xi}
\wt{\xi} =  \exp(\ad_{J(\ga)})\, \xi .
\end{equation}

The vector $\wt{\xi}$ obviously belongs to 
$\cF_1 \Conv(\oc^{\vee}_{\circ}, \KGra)$\,. Furthermore, 
$\wt{\xi}$ satisfies the condition
$$
\wt{\xi}(\bsi \st^{\mc}_n) = 0\,, \qquad \forall~~ n \ge 2
$$
since so does $\xi$.

Thus the MC elements
$$
 \exp(\ad_{J(\ga)}) \al\,, \quad \textrm{and} \quad 
 \exp(\ad_{J(\ga)}) \wt{\al}
$$
indeed correspond to homotopy equivalent stable 
formality quasi-isomorphisms. 

To prove the second statement, we introduce the Lie algebra (in $\grVect_{\bbK}$) 
\begin{equation}
\label{Conv-add-t}
\Conv(\oc^{\vee}, \KGra)  \hotimes \bbK[t],
\end{equation}
where $\Conv(\oc^{\vee}, \KGra)$ is considered with the topology coming
from the filtration $\cF_{\bul}$  ``by arity'' \eqref{Conv-filtr} and $\bbK[t]$ 
is considered with the discrete topology. 

Let us denote by $\al(t)$ the following 
vector in \eqref{Conv-add-t}
\begin{equation}
\label{al-t}
\al(t) =  \exp(t\, \ad_{J(\ga)}) \al\,.   
\end{equation}
It is easy to see that $\al(t)$ enjoys the MC equation 
\begin{equation}
\label{MC-al-t}
[\al(t), \al(t)] =0
\end{equation}
and the conditions 
\begin{equation}
\label{al-t-Schouten}
\al(t)\,(\bs^{-1}\, \st^{\mc}_n) = \begin{cases}
   \G_{\ed} \qquad {\rm if} ~~ n = 2\,,  \\
   0 \qquad {\rm if} ~~  n \ge 3\,,
\end{cases}
\end{equation}
\begin{equation}
\label{al-t-mult}
\al(\bs^{-1}\, \st^{\mo}_2) = \G_{\ww}\,, 
\end{equation}
\begin{equation}
\label{al-t-HKR}
\al(\bs^{-1} \, \st^{\mo}_{1,k}) = \frac{1}{k!} \G^{\br}_{k}
\end{equation}
since so does $\al$ and since $\ga$ is cocycle in $\dfGC$\,.
Furthermore, $\al(t)$ satisfies the following (formal) differential 
equation 
\begin{equation}
\label{diff-eq-al-t}
\frac{d}{d t} \, \al(t) = [ J(\ga)  , \al(t) ]
\end{equation}
with the initial condition 
\begin{equation}
\label{initial}
\al(t) \Big|_{t=0} =  \al\,.
\end{equation}

Let us now assume that 
\begin{equation}
\label{ga-psi}
\ga  = [\G_{\ed}, \psi]
\end{equation}
for a degree $-1$ vector in $\dfGC$\,.

Since loops are not allowed and $\psi$ has degree $-1$, we have 
\begin{equation}
\label{psi-in-cF1}
\psi \in  \cF_1 \dfGC\,.
\end{equation}

Moreover, since the map $J$ \eqref{dfGC-Conv} is compatible with 
Lie brackets, we have 
$$
J(\ga) =  [J(\G_{\ed}), J(\psi)]
$$
and hence the vector $\al(t)$ satisfies 
the equation 
\begin{equation}
\label{diff-eq-al-psi}
\frac{d}{d t} \, \al(t)  = 
\big[  [J(\G_{\ed}), J(\psi)], \al(t) \big]\,.
\end{equation}

Let us denote by $\D\al(t)$ the difference
\begin{equation}
\label{D-al}
\D\al(t)  = \al(t) - J(\G_{\ed}) \in  \Conv(\oc^{\vee}, \KGra) \hotimes \bbK[t]
\end{equation}
The MC equation \eqref{MC-al-t}
for $\al(t)$ implies that 
\begin{equation}
\label{MC-D-al}
[J(\G_{\ed}), \D\al(t) ]  + \frac{1}{2} [\D\al(t), \D\al(t)] = 0\,.
\end{equation}  
Furthermore, due to equation \eqref{al-t-Schouten}, we have 
\begin{equation}
\label{D-al-mc}
\D\al(t)\, \big(\bsi \st^{\mc}_m \big) = 0 \qquad \forall ~ ~ m \ge 2\,. 
\end{equation}  

Using the Jacobi idenity,  identity $[J(\G_{\ed}), J(\G_{\ed})] = 0$, 
and equation \eqref{MC-D-al} we rewrite 
equation \eqref{diff-eq-al-psi} as follows
$$
\frac{d}{d t} \, \al(t)  = 
\big[  [J(\G_{\ed}), J(\psi)], J(\G_{\ed})\big] +
\big[  [J(\G_{\ed}), J(\psi)], \D\al(t) \big] = \big[  [J(\G_{\ed}), J(\psi)], \D\al(t) \big] =
$$
$$
- \big[ [J(\psi),  \D\al(t) ], J(\G_{\ed})\big] - \big[ [J(\G_{\ed}),  \D\al(t)], J(\psi) \big] =
$$
$$
- \big[ [J(\psi),  \D\al(t) ], J(\G_{\ed})\big]  + \frac{1}{2}  \big[ [\D\al(t),  \D\al(t)], J(\psi) \big] =
$$
$$
- \big[ [J(\psi),  \D\al(t) ], J(\G_{\ed})\big] - \frac{1}{2}  \big[ [ \D\al(t), J(\psi)], \D\al(t) \big]
- \frac{1}{2}  \big[ [J(\psi),  \D\al(t)], \D\al(t) \big]= 
$$
$$
- \big[ [J(\psi),  \D\al(t) ], \al(t) \big]\,.
$$

Thus the vector $\al(t)$ \eqref{al-t} satisfies the (formal) 
differential  equation 
\begin{equation}
\label{diff-eq-xi}
\frac{d}{d t} \, \al(t)  = [\eta(t), \al(t)]\,,
\end{equation}
where $\eta(t)$ is the degree zero vector
\begin{equation}
\label{eta-t}
\eta(t) = -[J(\psi),  \D\al(t) ] \in \Conv(\oc^{\vee}, \KGra)  \hotimes \bbK[t].
\end{equation}
It is clear that $\eta(t)$ satisfies the conditions
$$
\eta(t) \big(\bsi \st^{\mc}_n \big) = 0 \qquad \forall~~ n \ge 2 
$$ 
and 
$$
\eta(t) \big( \bsi \st^{\mo}_{1,k} \big) =  0 \qquad \forall ~~ k \ge 0\,.
$$

Let us apply Theorem \ref{thm:isom} from Appendix \ref{app:diffura} to the case when $\cL = \Conv(\oc^{\vee}, \KGra)$ and 
$\mg$ consists of vectors in $\Conv(\oc^{\vee}, \KGra)^0$ which satisfy \eqref{xi-cond-mc}. 
Due to this theorem, there exists a vector
$$
\xi \in \cF_1 \Conv(\oc^{\vee}_{\circ}, \KGra)
$$
which satisfies \eqref{xi-cond-mc} and \eqref{al-xi-ga}. 

Thus the second statement of Proposition \ref{prop:to-homotopy} is proved. 

To prove the last statement, we observe that \eqref{psi-in-cF-1n}, 
\eqref{D-al-mc}, and the identities 
$$
J(\psi) (\bsi \st^{\mo}_{m, k}) = J(\psi) (\bsi \st^{\mo}_{k_1}) = 0 \qquad \forall~~ m \ge 1, ~k \ge 0, ~k_1 \ge 2 
$$
imply that $\eta(t) \in \cF^{\mc}_{n}\Conv(\oc^{\vee}, \KGra)  \hotimes \bbK[t]$ and 
hence the inclusion in \eqref{xi-mc-n} holds. 

Inclusion \eqref{psi-in-cF-1n} implies that $\ga \in \cF_n \dfGC$. 
Combining this fact with \eqref{al-in-cF0}, we conclude that
$$
\D\al(t) - (\al - \G_{\ed}) \in \cF_{n}  \Conv(\oc^{\vee}, \KGra).
$$
Therefore, 
$$
\eta(t) =  - [J(\psi), \al - \G_{\ed}] \quad  \textrm{mod} \quad  \cF_{n}  \Conv(\oc^{\vee}, \KGra)[[t]]. 
$$

Hence, due to the second part of Theorem \ref{thm:isom}, there exists 
$$
\xi \in \cF_1 \Conv(\oc^{\vee}_{\circ}, \KGra)
$$
such that \eqref{xi-cond-mc} and \eqref{al-xi-ga} hold, and 
we have 
$$
\xi(\bsi \st^{\mo}_{n,0}) =  -[J(\psi),  \al - \G_{\ed}] (\st^{\mo}_{n,0}) = 
$$
$$
=   \al (\bsi \st^{\mo}_{1,0} ) \circ_{1, \mc} J(\psi) (\bsi \st^{\mc}_{n}) = 
\G^{\br}_{0} \circ_{1,\mc} \psi(1_n)  = \psi(1_n)\,.
$$
Thus equation \eqref{xi-more} also holds. 
\end{proof}

Proposition \ref{prop:to-homotopy} implies that 
%
%
\begin{cor}
\label{cor:descends}
The action of $\exp\big(\cZ^0(\dfGC)\big)$ on SFQs descends to 
an action of $\exp\big(H^0(\dfGC)\big)$ on homotopy classes of SFQs. 
\end{cor}
\begin{proof}
Let $\ga$ be a degree zero cocycle of $\dfGC$\,.
Due to the first statement of Proposition \ref{prop:to-homotopy}, 
$\exp(\ga)$ transforms homotopy equivalent SFQs to homotopy equivalent SFQs.

Thus it remains to prove that, if $\ga'$ is cohomologous to $\ga$ then 
MC elements 
\begin{equation}
\label{MC-MC}
\exp(\ad_{J(\ga')}) \al  \qquad \textrm{and} \qquad
\exp(\ad_{J(\ga)}) \al
\end{equation}
are connected by the action of $\exp(\ad_{\xi})$ for a vector
$$
\xi \in \cF_{1} \Conv(\oc^{\vee}_{\circ}, \KGra)
$$
satisfying condition \eqref{xi-cond-mc}.

Using the fact that the difference $\ga' - \ga$ is exact, it is 
easy to see that 
$$
\CH(-\ga, \ga') 
$$
is also exact. 

Therefore, due to the second statement of  Proposition \ref{prop:to-homotopy}, 
the MC elements 
$$
\exp(\ad_{J(\ga)})\, \exp(\ad_{J(\ga')})\, \al  \qquad \textrm{and} \qquad \al
$$
are connected by the action of $\exp(\ad_{\xi})$ for a vector
$\xi \in \cF_{1} \Conv(\oc^{\vee}_{\circ}, \KGra)$
satisfying condition \eqref{xi-cond-mc}.

Hence, the MC elements \eqref{MC-MC} represent homotopy equivalent SFQs.
\end{proof}
\begin{remark}
Let $A$ be a finitely generated free commutative algebra 
in $\grVect_{\bbK}$ and $V_{A}$ be the algebra of 
polyvector fields on the corresponding (graded) affine space. 
It is not hard to see that $\La\Lie_{\infty}$-morphisms 
from $V_A$ to $\Cbu(A)$ which correspond to 
$\al_F$ and \eqref{dfGC-action} are connected by 
the action described in \cite[Section 5]{K-conj} by 
M. Kontsevich.
\end{remark}

Let us now state the main result of this paper 
%
%
\begin{thm}
\label{thm:main}
The pro-unipotent group $\exp (H^0(\dfGC))$ acts simply 
transitively on the set of homotopy classes of 
stable formality quasi-isomorphisms (SFQs).
\end{thm}
The proof of this theorem occupies the next two sections 
of the paper and it depends on a few technical 
statements which are proved in Appendices \ref{app:Hoch} 
and \ref{app:Hg}.  

Combining Theorem \ref{thm:main} with Theorem \ref{thm:grt-dfGC} 
stated above, we deduce that
\begin{cor}
\label{cor:main}
The set of homotopy classes of SFQs form a torsor for the Gro\-then\-dieck-Teich\-muel\-ler
group $\GRT_1$. $\Box$
\end{cor}
\begin{remark}  
\label{rem:about-Dima}
We should mention that this result agrees very well with Tamarkin's 
approach \cite{Hinich}, \cite{Dima-Proof} to Kontsevich's formality theorem 
\cite{K}, \cite{K-conj}. Tamarkin's construction \cite[Section 2]{GRTEquiv}, \cite{Dima-Proof}
may be viewed as a map from the set of Drinfeld associators\footnote{Here, we only consider 
Drinfeld associators whose ``braiding'' constant is $1$. In \cite[Section 5]{Drinfeld}, this set 
is denoted by $M_1$. In \cite[Section 4.1]{GRTEquiv}, this set is denoted 
by $\mathrm{DrAssoc}_1$.} to the set of 
homotopy classes of formality quasi-isomorphisms for Hochschild cochains. 
Due to \cite[Proposition 5.5]{Drinfeld} and \cite[Theorem 1.1]{Thomas}, both the source and 
the target of Tamarkin's map are equipped with the actions of the group $\GRT_1$. 
Moreover, according\footnote{For more details, we refer the 
reader to \cite{GRTEquiv}.} to \cite[Section 10.1]{Thomas}, Tamarkin's construction 
is equivariant with respect to the action of $\GRT_1$.

Section 6 of paper \cite{Thomas-more} contains a sketch on a version of 
Tamarkin's construction in ``stable'' setting. According to this sketch, every SFQ 
is homotopic to an SFQ which can be extended to a stable $\Ger_{\infty}$-morphism 
from polyvector fields to Hochschild cochains (for some choice of Tamarkin's $\Ger_{\infty}$-structure 
on Hochschild cochains). 
\end{remark}
  
%
%

\section{The action of $\exp\big( H^0(\dfGC)\big)$ is transitive}
\label{sec:transit}

%
%

Let $\al$ and $\wt{\al}$ be MC elements of the graded Lie algebra 
$\Conv(\oc^{\vee}_{\circ}, \KGra)$ corresponding to SFQs.

Since $\al$ and $\wt{\al}$ satisfy conditions \eqref{al-Schouten}, \eqref{al-mult},
\eqref{al-HKR}, and \eqref{al-mult-higher}, the difference 
\begin{equation}
\label{del-al}
\de \al : = \wt{\al} - \al
\end{equation}
satisfies the identities
\begin{equation}
\label{del-al-mc-mult}
\de\al(\bs^{-1}\,\st^{\mc}_{m}) = 0\,, \qquad 
\de\al(\bs^{-1}\,\st^{\mo}_{k}) = 0\,, \qquad \forall~~ m,k \ge 2,
\end{equation}
and 
\begin{equation}
\label{del-al-Hoch}
\de\al(\bs^{-1}\,\st^{\mo}_{1,k}) = 0\,, \qquad \forall \quad k \ge 0.
\end{equation}
Therefore\footnote{See conditions \eqref{color-is-chi} and \eqref{color-is-not-chi} in Section \ref{sec:Conv}.}, 
\begin{equation}
\label{del-al-cF-mc-2}
\de\al \in \cF^{\mc}_2\, \Conv(\oc^{\vee}_{\circ}, \KGra).
\end{equation}

We will deduce the transitivity of the action of $\exp\big( H^0(\dfGC)\big)$ on 
homotopy classes of SFQs from the following statement
\begin{prop}
\label{prop:transit-refined}
If $\al$ and $\wt{\al}$ are MC elements of the graded Lie algebra \eqref{Conv-oc-KGra} 
corresponding to SFQs and 
\begin{equation}
\label{diff-in-cF-mc-n}
\wt{\al} - \al \in \cF^{\mc}_n\, \Conv(\oc^{\vee}_{\circ}, \KGra)
\end{equation}
for some $n \ge 2$ then there exists a degree zero cocycle 
$\ga \in \cF_{n-1} \dfGC$ and a degree zero vector 
$$
\xi \in \cF^{\mc}_{n-1}\, \Conv(\oc^{\vee}_{\circ}, \KGra)
$$
for which \eqref{xi-cond-mc} holds and 
\begin{equation}
\label{new-difference}
\exp(\ad_{\xi})\, \wt{\al} ~ - ~  \exp(\ad_{J(\ga)}) \al \in \cF^{\mc}_{n + 1}\, \Conv(\oc^{\vee}_{\circ}, \KGra).
\end{equation}
\end{prop}
The proof of this proposition consists of two parts. The first part is given in Section \ref{sec:de-al-n-0} 
and the second part is given in Section \ref{sec:de-al-n-k}.

\subsection{Proposition \ref{prop:transit-refined} implies that the action is transitive}
\label{sec:transit-proved}

Using \eqref{del-al-cF-mc-2} and applying Proposition \ref{prop:transit-refined} recursively, 
we see that there exist infinite sequences of degree zero vectors 
\begin{equation}
\label{xi-seq}
\xi_1, \xi_2, \xi_3, \dots~ \in~ \Conv(\oc^{\vee}_{\circ}, \KGra), \qquad 
\xi_m \in  \cF^{\mc}_m \Conv(\oc^{\vee}_{\circ}, \KGra)
\end{equation}
\begin{equation}
\label{ga-seq}
\ga_1, \ga_2, \ga_3, \dots ~\in~ \dfGC, \qquad 
\ga_m \in \cF_m \dfGC   
\end{equation}
such that each $\xi_m$ satisfies \eqref{xi-cond-mc} and
\begin{equation}
\label{diff-desired}
\exp(\ad_{\xi_{m}}) \dots \exp(\ad_{\xi_{1}})\, \wt{\al} ~ - ~
\exp(\ad_{J(\ga_{m})})\dots \exp(\ad_{J(\ga_{1})}) \al \in  \cF^{\mc}_{m+2}\, \Conv(\oc^{\vee}_{\circ}, \KGra)
\end{equation}
for every $m \ge 1$. 

Since $\Conv(\oc^{\vee}_{\circ}, \KGra)$ (resp. $\dfGC$) is complete with respect to the 
filtration $\cF^{\mc}_{\bul}$ (resp. $\cF_{\bul}$), the existence of the above sequence 
implies that the homotopy class of the SFQ corresponding to $\wt{\al}$ has a 
representative which lies on the $\exp(\cZ^0(\dfGC))$-orbit of the SFQ corresponding to $\al$.

\subsection{Taking care of $\de\al(\bsi \st^{\mo}_{n,0})$}
\label{sec:de-al-n-0}

Due to \eqref{del-al-mc-mult} and \eqref{del-al-Hoch},
the element $\de\al$ is uniquely determined by the vectors 
\begin{equation}
\label{all-del-al-mk}
\de\al_{m,k} : = \de\al(\bsi \st^{\mo}_{m,k}) ~ \in ~ \KGra(m,k). 
\end{equation}

In addition, since the restriction 
$$
\de\al \Big|_{\oc^{\vee}_{\circ} (m,k)^{\mo}} 
$$ 
is $S_m \times S_k$ equivariant and $\de\al$ has 
degree $1$,  $\de\al_{m,k}$ may be viewed as a \emph{degree zero} vector in 
\begin{equation}
\label{KGra-inv-mk}
\bs^{2m-2+k} \big( \KGra(m,k)^{\mo} \big)^{S_m}\,.
\end{equation}

Condition \eqref{diff-in-cF-mc-n} is equivalent to $\de\al  \in  \cF^{\mc}_n\, \Conv(\oc^{\vee}_{\circ}, \KGra)$.
In other words, we know that
\begin{equation}
\label{del-al-mk}
\de\al_{m,k} = 0 \qquad \forall \quad m \le n-1, ~ k \ge 0\,.
\end{equation}
So vectors $\de\al_{m,k}$ may be non-zero only for $m \ge n$. 

Since both $\al$ and $\wt{\al}$ satisfy the MC equations
$$
[\al, \al] = 0, \qquad [\,\wt{\al}\,, \,\wt{\al}\,] = 0,
$$
the difference $\de\al$ satisfies the equation 
\begin{equation}
\label{del-al-MC}
[\al, \de\al] +  \frac{1}{2} [\de \al, \de \al] = 0.
\end{equation}

Since $\de\al  \in  \cF^{\mc}_n\, \Conv(\oc^{\vee}_{\circ}, \KGra)$,  
\begin{equation}
\label{del-al-del-al-n-k}
[\de\al, \de\al] (\bsi \st^{\mo}_{n,k}) = 0 \qquad \forall ~~ k \ge 0.
\end{equation}
Moreover, it is easy to see that only $\cD_{\As}(\st^{\mo}_{n,k})$ may contribute to 
the expression $[\al, \de\al] (\bsi \st^{\mo}_{n,k})$. 

Unfolding $[\al, \de\al](\bsi \st^{\mo}_{n,k})$, we get 
$$
[\al, \de\al](\bsi \st^{\mo}_{n,k}) = 
$$
$$
- \big(  \al(\bsi \st^{\mo}_2) \circ_{2, \mo} \de\al_{n,k-1} +
\sum_{i=1}^{k-1}  (-1)^i \de\al_{n,k-1}  \circ_{i, \mo} \al(\bsi \st^{\mo}_2)
+(-1)^k \al(\bsi \st^{\mo}_2) \circ_{1, \mo} \de\al_{n,k-1}  \big) =
$$
$$
- \big(  \G_{\ww} \circ_{2, \mo} \de\al_{n,k-1} +
\sum_{i=1}^{k-1}  (-1)^i \de\al_{n,k-1}  \circ_{i, \mo} \G_{\ww}
+(-1)^k \G_{\ww} \circ_{1, \mo} \de\al_{n,k-1}  \big).
$$

Thus \eqref{del-al-MC}, \eqref{del-al-del-al-n-k}, and the above calculation imply the following statement:  
\begin{claim}
\label{cl:del-al-Hoch}
For each $k \ge 0$, the vector $\de\al_{n,k}$
is a degree zero cocycle in the cochain complex 
\begin{equation}
\label{Hoch}
\KGra^{\Hoch}_{\inv}  =
\bs^{2n-2} \, \bigoplus_{k \ge 0} \bs^{k}\big( \KGra(n,k)^{\mo} \big)^{S_n}
\end{equation}
with the differential $\pa^{\Hoch}$ given by the formula
\begin{equation}
\label{pa-Hoch}
\pa^{\Hoch} (\ga) = 
\G_{\ww} \circ_{2, \mo} \ga - \ga \circ_{1, \mo} \G_{\ww} + 
\ga \circ_{2, \mo} \G_{\ww} - \dots 
\end{equation}
$$
+ (-1)^{k} \ga \circ_{k, \mo} \G_{\ww}
+ (-1)^{k+1}\G_{\ww} \circ_{1, \mo} \ga\,,
$$
$$
\ga \in  \bs^{2n-2 + k}\big( \KGra(n,k)^{\mo}  \big)^{S_n}\,.
$$
\qed
\end{claim}

The cochain complex \eqref{Hoch} is examined in detail in Appendix \ref{app:Hoch}. 
For now, we use Corollary \ref{cor:k-one} and 
Claim \ref{cl:del-al-Hoch} to deduce that 
\begin{claim}
\label{cl:del-al-n1}
The white vertex of each graph in the linear 
combination 
$$
\de\al_{n,1} \in  \bs^{2n-1}\big( \KGra(n,1)^{\mo}  \big)^{S_n}
$$
has valency $1$\,.  $\Box$
\end{claim}

\subsubsection{Pikes in $\de\al_{n,0}$ can be ``killed''}
In general linear combinations $\de\al_{m,k}$ 
\eqref{all-del-al-mk} may 
contain graphs with a black vertex
of valency $1$ whose adjacent edge terminates 
at this vertex. We call such vertices {\it pikes}. 

The following statement says that, if equation \eqref{del-al-mk} holds, 
then pikes in $\de\al_{n,0}$ can be ``killed''. More precisely, 
\begin{claim}
\label{cl:kill-pikes}
If equation \eqref{del-al-mk} holds, then there exists a degree zero vector 
$$
\xi \in  \cF^{\mc}_{n-1} \Conv(\oc^{\vee}_{\circ}, \KGra)
$$
such that \eqref{xi-cond-mc} holds, each graph in the linear combination   
\begin{equation}
\label{new-diff}
\big(\exp([\xi,~]) \wt{\al} \,\big) (\bsi \st^{\mo}_{n,0}) ~ - ~ \al(\bsi \st^{\mo}_{n,0})
\end{equation}
does not have pikes, and 
\begin{equation}
\label{new-diff-Fmc-n}
\big(\exp([\xi,~]) \wt{\al} \,\big) ~ - ~ \al ~\in~   \cF^{\mc}_{n} \Conv(\oc^{\vee}_{\circ}, \KGra).
\end{equation}
\end{claim}
\begin{proof}
Let us denote by $\de\al^{r}_{n,0}$ the linear combination 
in $\KGra(n,0)^{\mo}$ which is obtained from $\de\al_{n,0}$ by 
retaining only graphs with exactly $r$ pikes. 

Since   $\de\al^r_{n,0}$ is a linear combination of 
graphs without white vertices,  it is a cocycle
in the complex \eqref{Hg} with the differential $\md$ \eqref{md} 
examined in detail in Appendix \ref{app:Hg}. 
According to  Lemma \ref{lem:Koszul} from this appendix, we have 
\begin{equation}
\label{mdmdr}
\md \md^* (\de\al^{r}_{n,0}) = r \de\al^{r}_{n,0}\,.
\end{equation}
Thus, for the vector 
\begin{equation}
\label{chi}
\chi_{n-1,1} = - \sum_{r \ge 1} \frac{1}{r} \md^* (\de\al^{r}_{n,0}) 
\in \bs^{2(n-1)-1}\big( \KGra(n-1,1)^{\mo}  \big)^{S_{n-1}},
\end{equation}
the linear combination 
$$
\de\al_{n,0} + \md (\chi_{n-1,1}) 
$$
does not have pikes. 

Next, we define the degree $0$ vector $\xi \in \Conv(\oc^{\vee}_{\circ}, \KGra)$ by setting
\begin{equation}
\label{xi-pikes}
\xi (\bs^{-1}\,\st^{\mo}_{n-1,1}) =  \chi_{n-1, 1}\,,
\qquad 
\xi (\bs^{-1}\,\st^{\mc}_{m_1}) = \xi (\bs^{-1}\, \st^{\mo}_{k_1}) =
\xi(\bs^{-1}\, \st^{\mo}_{m,k}) = 0  
\end{equation}
for all $m_1$, $k_1$ and for all pairs $(m,k) \neq (n-1,1)$\,.
By construction, $\xi \in \cF^{\mc}_{n-1} \Conv(\oc^{\vee}_{\circ}, \KGra)$ and satisfies 
\eqref{xi-cond-mc}. 

Let us denote by $\wt{\al}'$ and $\de\al'$ the new MC element 
\begin{equation}
\label{wt-al-pr}
\wt{\al}' : = \exp(\ad_{\xi}) \wt{\al}
\end{equation}
and the new difference $\de\al' = \wt{\al}' - \al$, respectively. 
Clearly, 
\begin{equation}
\label{del-al-pr}
\de\al' = \exp(\ad_{\xi}) \de\al + 
 \exp(\ad_{\xi}) \al - \al.
\end{equation}

To prove that each graph in the linear combination $\de\al' (\bsi \st^{\mo}_{n,0})$
does not have pikes, we observe that, for every $f \in \Conv(\oc^{\vee}_{\circ}, \KGra)$,
$$
[\xi, f] (\bsi \st^{\mo}_{n,0})  = 
\sum_{i=1}^n  (\tau_{n,i}, \id) \big(
\xi (\bsi \st^{\mo}_{n-1,1}) \circ_{1,\mo} f (\bsi \st^{\mo}_{1,0}) \big)\,,
$$
where $\tau_{n,i}$ is the following family of cycles in $S_n$
\begin{equation}
\label{tau-n-i}
\tau_{n,i} :=  (i, i+1, \dots, n-1, n).
\end{equation}

Therefore 
$$
\de\al'(\bsi \st^{\mo}_{n,0}) =  \de\al (\bsi \st^{\mo}_{n,0}) + [\xi, \al] (\bsi \st^{\mo}_{n,0}) =
$$
$$
\de\al_{n,0} + \sum_{i=1}^n  (\tau_{n,i}, \id) \big( \xi (\bsi \st^{\mo}_{n-1,1}) \circ_{1,\mo} \al (\bsi \st^{\mo}_{1,0}) \big) =
$$
$$
 \de\al_{n,0}  + 
 \sum_{i=1}^n  (\tau_{n,i}, \id) \big(
 \xi (\bsi \st^{\mo}_{n-1,1}) \circ_{1,\mo} \G^{\br}_0 \big) ~=~
 \de\al_{n,0}  +   \md(\chi_{n-1,1}).
$$
As we showed above, all graphs in this linear combination 
do not have pikes. 

It remains to show that $\de\al' \in \cF^{\mc}_{n}\Conv(\oc^{\vee}_{\circ}, \KGra)$. 

Since  $\xi \in \cF^{\mc}_{n-1} \Conv(\oc^{\vee}_{\circ}, \KGra)$ and 
$\wt{\al} \in  \cF^{\mc}_{0} \Conv(\oc^{\vee}_{\circ}, \KGra)$, 
$$
\ad^k_{\xi} (\wt{\al}) \in \cF^{\mc}_{k(n-1)} \Conv(\oc^{\vee}_{\circ}, \KGra).
$$

For $k \ge 2$, this observation (together with the inequality $n \ge 2$) implies that 
$\ad^k_{\xi} (\wt{\al}) \in  \cF^{\mc}_{n}\Conv(\oc^{\vee}_{\circ}, \KGra)$. 

Since $[\xi, \wt{\al}] \in  \cF^{\mc}_{n-1}\Conv(\oc^{\vee}_{\circ}, \KGra)$ and 
$[\xi, \wt{\al}] (\bsi \st^{\mc}_{m}) = 0$ for all $m$, it is sufficient to show that 
\begin{equation}
\label{xi-wt-al}
[\xi, \wt{\al}] (\bsi \st^{\mo}_{n-1,k}) = 0, \qquad \forall~ k \ge 0.
\end{equation}

Using the defining equations of $\xi$ (see \eqref{xi-pikes}), it is easy to see that 
$[\xi, \wt{\al}] (\bsi \st^{\mo}_{n-1,k}) = 0$ for all $k \neq 2$. 
As for $k =2$, we have 
$$
[\xi, \al] \big(\bs^{-1}\,\st^{\mo}_{n-1, 2} \big) =
\chi_{n-1,1}\circ_{1, \mo} \G_{\ww} - 
\G_{\ww} \circ_{1, \mo} \chi_{n-1,1}  - \G_{\ww} \circ_{2, \mo} \chi_{n-1,1} = 0
$$ 
because the white vertex in each graph of the linear 
combination $\chi_{n-1,1}$ has valency $1$\,.

Claim \ref{cl:kill-pikes} is proved. 
\end{proof}

\subsubsection{$\de\al_{n,0}$ is a degree zero cocycle in $\dfGC$}
\label{sec:cocycle}

Since $\de\al_{n,0} = \de\al(\bsi \st^{\mo}_{n,0})$ is a linear combination 
of graphs with only black vertices and it is invariant with respect to the 
action of $S_n$, we may view $\de\al_{n,0}$ as a vector in the full directed 
graph complex $\dfGC$ (see \eqref{dfGC}). 

We will now prove the following statement:
\begin{claim}
\label{cl:del-al-cocycle}
Let $\al$ and $\wt{\al}$ be MC elements in \eqref{Conv-oc-KGra} 
corresponding to SFQs.
If $\de\al$ satisfies \eqref{del-al-mk} and 
all graphs in the linear combination
$\de\al_{n,0}$ do not have pikes, then $\de\al_{n,0}$ is a degree zero cocycle
in $\dfGC$ \eqref{dfGC}.
\end{claim}
\begin{proof} The claim that $\de\al_{n,0}$ is a degree zero vector in $\dfGC$ follows 
immediately from the fact that $\st^{\mo}_{n,0}$ has degree $2-2n$. So we will proceed to 
the proof of the cocycle condition for $\de\al_{n,0}$\,. 

Since
$$
\cD(\st^{\mo}_{n+1,0}) ~=~ 
\sum_{p=2}^{n+1} \sum_{\tau\in \Sh_{p, n+1-p}}
(\tau, \id) \big(  \st^{\mo}_{n+2-p,0} \circ_{1, \mc} \st^{\mc}_{p}  \big) 
$$
$$
- \sum_{r=1}^{n} 
\sum_{\si \in \Sh_{r, n+1-r}}
(\si, \id) \big( \st^{\mo}_{r, 1} \circ_{1, \mo} \st^{\mo}_{n+1-r, 0}  \big),
$$
$\de\al (\bsi\st^{\mc}_m) = 0$ for all $m$ and \eqref{del-al-mk} holds, we have 
$$
\de\al\bsi\otimes\de\al\bsi(\cD(\st^{\mo}_{n+1,0})) = 0\,,
$$
$$
\de\al\bsi \otimes \al\bsi(\cD(\st^{\mo}_{n+1,0}))~ = ~  
\sum_{\tau\in \Sh_{2, n-1}}
(\tau, \id) \big(  \de\al_{n,0} \circ_{1, \mc} \G_{\ed}  \big) 
- \sum_{i=1}^{n+1} 
(\tau_{n+1, i}, \id) \big(  \de\al_{n,1} \circ_{1, \mo} \G^{\br}_0  \big)\,,
$$
and
\begin{equation}
\label{al-del-al-n1}
\al\bsi \otimes \de\al\bsi(\cD(\st^{\mo}_{n+1,0}))= 
- \sum_{i=1}^{n+1} 
(\si_{n+1, i}, \id) \big(  \G^{\br}_1 \circ_{1, \mo} \de\al_{n,0}  \big)\,,
\end{equation}
where $\si_{n+1,i}$ and  $\tau_{n+1,i}$ are the cycles 
$(i, i-1,\dots, 1)$ and $(i, i+1, \dots, n+1)$ in $S_{n+1}$, respectively. 

Thus, applying the right hand side of \eqref{del-al-MC} to $\bsi \st^{\mo}_{n+1, 0}$, we get 
the identity for vectors $\de\al_{n,0}$ and $\de\al_{n,1}$\,:
\begin{equation}
\label{iden-delaln0}
\sum_{\tau\in \Sh_{2, n-1}}
(\tau, \id) \big(  \de\al_{n,0} \circ_{1, \mc} \G_{\ed}  \big) 
- \sum_{i=1}^{n+1} 
(\si_{n+1,i}, \id) \big(  \G^{\br}_1 \circ_{1, \mo} \de\al_{n,0}  \big) 
\end{equation}
$$
- \sum_{i=1}^{n+1} 
(\tau_{n+1, i}, \id) \big(  \de\al_{n,1} \circ_{1, \mo} \G^{\br}_0  \big) =0\,.
$$

Notice that $\G^{\br}_0$ consists of a single black vertex
and the insertion  $\de\al_{n, 1} \circ_{1, \mo}  \G^{\br}_{0}$
is nothing but replacing the single white vertex in each 
graph of the linear combination $\de\al_{n, 1} $ by black vertex
with label $n+1$\,. 

On the other hand, Claim \ref{cl:del-al-n1} says that 
all white vertices in $\de\al_{n,1}$ have valency $1$\,.
Thus, for each graph  
in  $\de\al_{n, 1} \circ_{1, \mo}  \G^{\br}_{0}$ the black 
vertex with label $n+1$ is necessarily a pike. 

Since $\de\al_{n,0}$ does not have pikes, the sum 
$$
- \sum_{i=1}^{n+1} 
(\tau_{n+1, i}, \id) \big(
  \de\al_{n, 1} \circ_{1, \mo}  \G^{\br}_{0} \big)\,.
$$
should necessarily cancel the linear combination $L_{pikes}$ which is obtained from  
\begin{equation}
\label{sum-with-pikes}
\sum_{\tau \in \Sh_{2, n-1} } (\tau, \id) \big( \de\al_{n,0} \circ_{1, \mc} \G_{\ed} \big)
\end{equation}
by retaining only the graphs with pikes. 

It is not hard to see that the linear combination $L_{pikes}$ coincides with 
the sum\footnote{Here we use the fact that $\de\al_{n,0}$ carries an even degree.}
\begin{equation}
\label{pikes}
\sum_{i=1}^{n+1} 
(\tau_{n+1,i}, \id) \big(  \G^{>}_{\ed} \circ_{1, \mc} \de\al_{n,0}  \big)\,,
\end{equation}
where $\tau_{n+1,i}$ is the cycle $(i, i+1, \dots, n+1)$ in $S_{n+1}$ and 
\begin{equation}
\label{edge>}
\G^{>}_{\ed} : = \begin{tikzpicture}[scale=0.5, >=stealth']
\tikzstyle{w}=[circle, draw, minimum size=4, inner sep=1]
\tikzstyle{b}=[circle, draw, fill, minimum size=4, inner sep=1]
\node [b] (b1) at (0,0) {};
\draw (0,0.6) node[anchor=center] {{\small $1$}};
\node [b] (b2) at (1.5,0) {};
\draw (1.5,0.6) node[anchor=center] {{\small $2$}};
\draw [->] (b1) edge (b2);
\end{tikzpicture}\,.
\end{equation}

Thus we conclude that 
$$
\sum_{i=1}^{n+1} 
(\tau_{n+1,i} , \id) \big(
  \de\al_{n, 1} \circ_{1, \mo}  \G^{\br}_{0} \big) = 
\sum_{i=1}^{n+1} 
(\tau_{n+1,i}, \id) \big(  \G^{>}_{\ed} \circ_{1, \mc} \de\al_{n,0}  \big).
$$

On the other hand, we have 
$$
\sum_{i=1}^{n+1} 
(\si_{n+1,i}, \id) \big(  \G^{\br}_1 \circ_{1, \mo} \de\al_{n,0}  \big) 
= 
\sum_{i=1}^{n+1} 
(\tau_{n+1,i}, \id) \big(  \G^{<}_{\ed} \circ_{1, \mc} \de\al_{n,0}  \big)\,,
$$
where 
\begin{equation}
\label{edge<}
\G^{<}_{\ed} = \begin{tikzpicture}[scale=0.5, >=stealth']
\tikzstyle{w}=[circle, draw, minimum size=4, inner sep=1]
\tikzstyle{b}=[circle, draw, fill, minimum size=4, inner sep=1]
\node [b] (b1) at (0,0) {};
\draw (0,0.6) node[anchor=center] {{\small $1$}};
\node [b] (b2) at (1.5,0) {};
\draw (1.5,0.6) node[anchor=center] {{\small $2$}};
\draw [<-] (b1) edge (b2);
\end{tikzpicture}\,.
\end{equation}

Therefore identity \eqref{iden-delaln0} implies that
$$
\sum_{\tau\in \Sh_{2, n-1}}
(\tau, \id) \big(  \de\al_{n,0} \circ_{1, \mc} \G_{\ed}  \big) 
- \sum_{i=1}^{n+1} 
(\tau_{n+1,i}, \id) \big(  \G_{\ed} \circ_{1, \mc} \de\al_{n,0}  \big) = 0.
$$

In other words, $\de\al_{n,0}$ is indeed a cocycle in $\dfGC$.
\end{proof}

Claim \ref{cl:del-al-cocycle} has the following corollary. 
\begin{cor}
\label{cor:killing-del-al-n0}
Let $\al$ and $\wt{\al}$ be MC elements in \eqref{Conv-oc-KGra} 
corresponding to SFQs.
If $\de\al$ satisfies \eqref{del-al-mk},  and 
all graphs in the linear combination
$\de\al_{n,0}$ do not have pikes, then there
exists degree zero cocycle $\ga\in \cF_{n-1}\dfGC$ such that 
\begin{equation}
\label{ga-kills-del-aln0}
\Big(\wt{\al} \quad - \quad   \big(\exp(\ad_{J(\ga)}) \al \big) \Big)\, (\bsi\st^{\mo}_{n,0}) ~ = ~ 0.
\end{equation}
and 
\begin{equation}
\label{still-in-Fmc-n}
\wt{\al} \quad - \quad   \big(\exp(\ad_{J(\ga)}) \al \big)  ~\in ~ \cF^{\mc}_{n}\Conv(\oc^{\vee}_{\circ}, \KGra)
\end{equation}
\end{cor}
\begin{proof}
Due to Claim \ref{cl:del-al-cocycle} the linear combination 
$\de\al_{n,0}$ is a degree zero cocycle in $\dfGC$\,.
So we set $\ga := -\de\al_{n,0}\,.$ In other words, 
$$
\ga (1_m) := \begin{cases}
 - \de\al_{n,0}  \qquad {\rm if} \quad m = n \quad {\rm and}   \\
   0 \qquad \qquad {\rm otherwise}\,,
\end{cases}
$$
where $1_m$ denotes the generator $\bs^{2-2m} 1 \in \bs^{2-2m} \bbK \cong \La^2 \coCom(m)$.
By construction, $\ga$ belongs to $\cF_{n-1}\dfGC$.

For any degree $1$ element $f \in \Conv(\oc^{\vee}, \KGra)$ we have 
$$
[J(\ga), f] (\bsi \st^{\mo}_{n,0}) = - f(\bsi  \st^{\mo}_{1,0} )  \circ_{1, \mc} \ga
\qquad \textrm{and} \qquad
[J(\ga), f] (\bsi \st^{\mo}_{1,0}) = 0.
$$

Therefore,
\begin{equation}
\label{exp-J-ga}
\wt{\al} (\bsi \st^{\mo}_{n,0}) ~ - ~
\big( \exp(\ad_{J(\ga)}) \al \big)\,  (\bsi \st^{\mo}_{n,0}) = \de \al (\bsi \st^{\mo}_{n,0}) 
+ \al(\bsi  \st^{\mo}_{1,0} )  \circ_{1, \mc} \ga =
\end{equation}
$$
\de\al_{n,0}  + \G^{\br}_0 \circ_{1, \mc} \ga = \de\al_{n,0} + \ga =  0
$$
and equation \eqref{ga-kills-del-aln0} follows. 

To prove \eqref{still-in-Fmc-n} we observe that
$$
J(\ga) \in \cF^{\mc}_{n-1} \Conv(\oc^{\vee}, \KGra). 
$$
Combining this observation with the fact that $\al \in \cF^{\mc}_{0} \Conv(\oc^{\vee}, \KGra)$ and 
the inequality $n \ge 2$, it is easy to see that 
$$
\ad^q_{J(\ga)} (\al) \in  \cF^{\mc}_{n} \Conv(\oc^{\vee}, \KGra) \qquad \forall ~ q \ge 2
$$
and 
\begin{equation}
\label{ga-al-Fmc-1n}
[J(\ga), \al] \in \cF^{\mc}_{n-1} \Conv(\oc^{\vee}, \KGra).
\end{equation}

Since $\ga$ is a cocycle in $\dfGC$, $[J(\ga), \al] (\bsi \st^{\mc}_m) = 0$ for all $m$. 
Furthermore, 
\begin{equation}
\label{ga-al-mc-1n}
[J(\ga), \al] (\bsi \st^{\mo}_{n-1,k}) = 0 \qquad \forall~~ k \ge 0
\end{equation}
since the vector $\st^{\mc}_{n}$ does not show up in $\cD(\st^{\mo}_{n-1,k})$. 
 
Thus  \eqref{still-in-Fmc-n} follows from \eqref{ga-al-Fmc-1n} and \eqref{ga-al-mc-1n}.  
\end{proof}

\subsection{Taking care of vectors $\de\al(\bsi \st^{\mo}_{n,k})$ for $k \ge 1$}
\label{sec:de-al-n-k}
Let us prove that following auxiliary statement: 
\begin{claim}
\label{cl:del-al-n-q}
Let $q$ be an integer $\ge 1$ and  $\al$, $\wt{\al}$ be MC elements of \eqref{Conv-oc-KGra} 
corresponding to SFQs. If $\de\al$ satisfies \eqref{del-al-mk} and 
\begin{equation}
\label{del-al-n-le-q}
\de\al_{n,k} = 0 \qquad \forall~~ k \le q-1,
\end{equation}
then there exists a degree zero vector 
\begin{equation}
\label{xi-n-q}
\xi \in \cF^{\mc}_{n-1} \Conv(\oc^{\vee}, \KGra) \cap \cF_{n+q-2} \Conv(\oc^{\vee}, \KGra)
\end{equation}
such that \eqref{xi-cond-mc} holds, 
\begin{equation}
\label{still-fine}
\big(\exp([\xi,~]) \wt{\al} \,\big) ~ - ~ \al ~\in~   \cF^{\mc}_{n} \Conv(\oc^{\vee}_{\circ}, \KGra)
\end{equation}
and 
\begin{equation}
\label{even-better}
\big(\exp([\xi,~]) \wt{\al} \,\big) (\bsi\st^{\mo}_{n,k}) ~ = ~ \al (\bsi\st^{\mo}_{n,k})  
\qquad \forall~~ k \le q. 
\end{equation}
\end{claim}
\begin{proof} The proof of this claim consists of two steps. 
First, we show that there exists a degree zero vector 
$$
\xi^{(1)} \in  \Conv(\oc^{\vee}, \KGra)
$$
such that $\xi^{(1)}(\bsi\st^{\mc}_m)  = 0$ for all $m$, 
$\xi^{(1)}(\bsi \st^{\mo}_{m_1,k}) = 0$ for all pairs $(m_1, k) \neq (n, q-1)$, 
\begin{equation}
\label{de-al-1-Fmc}
\big(\exp([\xi^{(1)},~]) \wt{\al} \,\big) ~ - ~ \al ~\in~   \cF^{\mc}_{n} \Conv(\oc^{\vee}_{\circ}, \KGra),
\end{equation}
\begin{equation}
\label{de-al-1-n-k}
\big(\exp([\xi^{(1)},~]) \wt{\al} \,\big) (\bsi\st^{\mo}_{n,k}) ~ = ~ \al (\bsi\st^{\mo}_{n,k})  \qquad \forall~~ k \le q-1  
\end{equation}
and the difference  
\begin{equation}
\label{n-q-after-1}
\big(\exp([\xi^{(1)},~]) \wt{\al} \,\big) (\bsi\st^{\mo}_{n,q}) ~ - ~ \al (\bsi\st^{\mo}_{n,q})
\end{equation}
satisfies Properties \ref{P:one}, \ref{P:anti-symm}, i.e. for each 
graph in \eqref{n-q-after-1} white vertices have valency $1$ and 
the linear combination \eqref{n-q-after-1} is anti-symmetric with 
respect to permutations of labels on white vertices. 

Let us denote by $\wt{\al}^{(1)}$ and $\de\al^{(1)}$
the new MC element 
\begin{equation}
\label{wt-al-1}
\wt{\al}^{(1)}  : = \big(\exp([\xi^{(1)},~]) \wt{\al} \,\big)
\end{equation}
and the new difference\footnote{We also set $\de\al^{(1)}_{m,k} := \de\al^{(1)} (\bsi \st^{\mo}_{m,k})$.} 
\begin{equation}
\label{del-al-1}
\de\al^{(1)} : = \big(\exp([\xi^{(1)},~]) \wt{\al} \,\big) ~ - ~ \al. 
\end{equation}

In the second step, we show that there exists a degree zero vector  
$$
\xi^{(2)} \in  \Conv(\oc^{\vee}, \KGra)
$$
such that  $\xi^{(2)}(\bsi\st^{\mc}_m)  = 0$ for all $m$, 
$\xi^{(2)}(\bsi \st^{\mo}_{m_1,k}) = 0$ for all pairs $(m_1, k) \neq (n-1, q+1)$, 
$$
\big(\exp([\xi^{(2)},~]) \wt{\al}^{(1)} \,\big) ~ - ~ \al ~\in~   \cF^{\mc}_{n} \Conv(\oc^{\vee}_{\circ}, \KGra)
$$
and 
$$
\big(\exp([\xi^{(2)},~]) \wt{\al}^{(1)} \,\big) (\bsi\st^{\mo}_{n,k}) ~ = ~ \al (\bsi\st^{\mo}_{n,k})  \qquad \forall~~ k \le q.  
$$

Since both vectors $\xi^{(1)}$ and $\xi^{(2)}$ satisfy \eqref{xi-cond-mc}, 
$\xi^{(1)}(\bsi \st^{\mo}_{m_1,k}) = 0$ for all pairs $(m_1, k) \neq (n, q-1)$  and 
$\xi^{(2)}(\bsi \st^{\mo}_{m_1,k}) = 0$ for all pairs $(m_1, k) \neq (n-1, q+1)$, 
$$
\xi^{(1)}, ~\xi^{(2)} ~\in~  
\cF^{\mc}_{n-1} \Conv(\oc^{\vee}, \KGra) \cap \cF_{n+q-2} \Conv(\oc^{\vee}, \KGra).
$$
Thus the desired vector \eqref{xi-n-q} is obtained by setting 
$$
\xi : = \CH(\xi^{(2)}, \xi^{(1)}). 
$$

~\\
{\bf Step 1.} If $q=1$ then the linear combination $\de\al_{n,q}$
already satisfies Properties  \ref{P:one}, \ref{P:anti-symm} due 
to Claim \ref{cl:del-al-n1}. So, in this case, we proceed to Step 2. 
In Step 1, it remains to consider the case $q \ge 2$\,. 
     
Due to Claim \ref{cl:del-al-Hoch}, the vector 
\begin{equation}
\label{del-al-nk-pr}
\de\al_{n,q} \in  \bs^{2n-2+q}  \Big( \KGra(n,q)^{\mo}  \Big)^{S_n}
\end{equation}
is a cocycle in the complex \eqref{Hoch-app} with the differential \eqref{pa-Hoch-app}. 
Thus, Corollary \ref{cor:Hoch} implies that there exists a vector 
$$
\psi_{n, q-1} \in  \bs^{2n+q-3}  \Big( \KGra(n,q-1)^{\mo}  \Big)^{S_n}
$$
such that the difference $\de\al_{n,q} - \pa^{\Hoch}\psi_{n, q-1}$
satisfies Properties \ref{P:one}, \ref{P:anti-symm}.

So we define $\xi^{(1)}$ by setting 
$$
\xi^{(1)} (\bsi\st^{\mo}_{n,q-1}) = \psi_{n, q-1}\,, 
\qquad  
\xi^{(1)} (\bsi\st^{\mc}_{m_1}) = \xi^{(1)} (\bsi \st^{\mo}_{k_1}) = 
\xi^{(1)} (\bsi \st^{\mo}_{m,k}) =0
$$  
for all $m_1 \ge 2$, $k_1 \ge 2$, and all pairs $(m,k) \neq (n, q-1)$\,.

It is easy to see that 
\begin{equation}
\label{xi-1-Fmc-n}
\xi^{(1)} ~\in~  \cF^{\mc}_{n} \Conv(\oc^{\vee}, \KGra) 
\end{equation}
and 
\begin{equation}
\label{xi-1-F-nq2}
\xi^{(1)} ~\in~  \cF_{n+q-2} \Conv(\oc^{\vee}, \KGra). 
\end{equation}

Since the vector \eqref{del-al-1} can be rewritten as
\begin{equation}
\label{del-al-1-new}
\de\al^{(1)} =  \exp(\ad_{\xi^{(1)}}) (\de\al) +  \exp(\ad_{\xi^{(1)} }) (\al) - \al,
\end{equation}
\eqref{de-al-1-Fmc} follows from \eqref{xi-1-Fmc-n} and the inclusions
$$
\de\al \in \cF^{\mc}_{n} \Conv(\oc^{\vee}, \KGra),
\qquad
\al \in \cF^{\mc}_{0} \Conv(\oc^{\vee}, \KGra).
$$

Since $\al \in \cF_{0} \Conv(\oc^{\vee}, \KGra)$ (see \eqref{al-in-cF0}), $n \ge 2$ and $q \ge 2$, 
inclusion \eqref{xi-1-F-nq2} implies that 
\begin{equation}
\label{r-ge-2}
\ad^r_{\xi^{(1)}} (\al)  \in    \cF_{n+q} \Conv(\oc^{\vee}, \KGra) \qquad \forall~~ r \ge 2 
\end{equation}
and 
\begin{equation}
\label{xi-1-al}
[\xi^{(1)}, \al] \in  \cF_{n+q-2} \Conv(\oc^{\vee}, \KGra).
\end{equation}

Moreover, since\footnote{Note that $\de\al (\bsi\st^{\mo}_{n+1,0})$ can be non-zero.} 
$\de\al \in \cF_{n} \Conv(\oc^{\vee}, \KGra)$ and $n \ge 2$, 
\begin{equation}
\label{r-ge-1}
\ad^r_{\xi^{(1)}} (\de \al)  \in  \cF_{n+q} \Conv(\oc^{\vee}, \KGra) \qquad \forall~~ r \ge 1. 
\end{equation}

Combining \eqref{r-ge-2}, \eqref{xi-1-al}, \eqref{r-ge-1} with 
$$
\de\al (\bsi \st^{\mo}_{n,k}) = 0 \qquad \forall ~~ k \le q-1,
$$
we deduce that 
\begin{equation}
\label{de-al-k-le-q2}
\de\al^{(1)} (\bsi \st^{\mo}_{n,k}) = 0 \qquad \forall~~ k \le q-2,
\end{equation}
\begin{equation}
\label{at-n-1q}
\de\al^{(1)} (\bsi \st^{\mo}_{n,q-1})  = [\xi^{(1)}, \al] (\bsi \st^{\mo}_{n,q-1}),
\end{equation}
and 
\begin{equation}
\label{at-n-q}
\de\al^{(1)} (\bsi \st^{\mo}_{n,q})  = \de\al_{n,q} + [\xi^{(1)}, \al] (\bsi \st^{\mo}_{n,q}).
\end{equation}

Thus, to prove \eqref{de-al-1-n-k}, it remains to show that 
$\de\al^{(1)} (\bsi \st^{\mo}_{n,q-1})  = 0$. The latter follows from 
\eqref{at-n-1q} and the fact that $\st^{\mo}_{n,q-1}$ does not show up 
in $\cD(\st^{\mo}_{n,q-1})$. 

Finally computing the right hand side of \eqref{at-n-q}, we deduce that 
$$
\de\al^{(1)} (\bsi \st^{\mo}_{n,q})  =  \de\al_{n,q} - \pa^{\Hoch} \psi_{n,q-1}
$$
which means that $\de\al^{(1)} (\bsi \st^{\mo}_{n,q})$ satisfies Properties \ref{P:one}, \ref{P:anti-symm}.

~\\
{\bf Step 2.} Since $\wt{\al}^{(1)}$ and $\al$ satisfy MC equations in $\Conv(\oc^{\vee}, \KGra)$, 
the difference $\de\al^{(1)}$ satisfies the equation 
\begin{equation}
\label{de-al-1-MC}
[\al,  \de\al^{(1)}] +  \frac{1}{2}[\de\al^{(1)} ,  \de\al^{(1)} ] = 0. 
\end{equation}

Due to \eqref{de-al-1-Fmc}, $\de\al^{(1)} \in \cF^{\mc}_n \Conv(\oc^{\vee}, \KGra)$ and hence 
$[\de\al^{(1)} ,  \de\al^{(1)} ] (\bsi \st^{\mo}_{n+1,q-1}) = 0$. So applying the left hand side of 
\eqref{de-al-1-MC} to $\bsi \st^{\mo}_{n+1,q-1}$ and using (see \eqref{de-al-1-n-k}) 
$$
\de\al^{(1)} (\bsi \st^{\mo}_{n,k}) = 0 \qquad \forall k \le q-1,
$$
we get the identity 
\begin{equation}
\label{n1-q-1}
 \sum_{i=1}^{n+1}  \sum_{p=0}^{q-1}
(-1)^p \, (\tau_{n+1,i}, \id)
\big(
\de\al^{(1)}_{n, q} \circ_{p+1, \mo} \G^{\br}_{0}
\big) ~=~  \G_{\ww} \circ_{2, \mo} \de\al^{(1)}_{n+1, q-2}  + 
\end{equation}
$$
 \sum_{p=1}^{q-2} (-1)^{p}
\de\al^{(1)}_{n+1, q-2} \circ_{p, \mo} \G_{\ww}
+ (-1)^{q-1} \G_{\ww} \circ_{1, \mo} \de\al^{(1)}_{n+1, q-2}\,,
$$
where $\tau_{n+1,i}$ is the cycle $(i, i+1, \dots, n+1)$ in $S_{n+1}$. 

In other words, the vector
\begin{equation}
\label{n-q}
\rho_{n+1, q-1} =
 \sum_{i=1}^{n+1}  \sum_{p=1}^{q}
(-1)^{p+1} \, (\tau_{n+1,i}, \id)
\big(
\de\al^{(1)}_{n, q} \circ_{p, \mo} \G^{\br}_{0}
\big) 
\end{equation}
is $\pa^{\Hoch}$-exact in 
\begin{equation}
\label{Hoch-n-k-pr-1}
\bs^{2(n+1)-2+ (q-1)} \big( \KGra(n+1, q-1) \big)^{S_{n+1}}\,.
\end{equation}

Since the difference \eqref{n-q-after-1} coincides with $\de\al^{(1)}_{n, q}$, 
the vector  $\de\al^{(1)}_{n, q}$ satisfies Properties \ref{P:one}, \ref{P:anti-symm}.
So, using the antisymmetry of  $\de\al^{(1)}_{n, q}$ with respect to the 
action of $S_{q}$ on the labels of white vertices, we see that 
$$
\sum_{p=1}^{q}
(-1)^{p+1}\, (\tau_{n+1, i}, \id)
\big(
\de\al^{(1)}_{n, q} \circ_{p, \mo} \G^{\br}_{0}
\big) = q\, (\tau_{n+1,i}, \id) \big( \de\al^{(1)}_{n, q} \circ_{1, \mo} \G^{\br}_{0} \big).
$$
In other words, 
$$
\rho_{n+1, q-1} = \md (\de\al^{(1)}_{n, q})\,,
$$
where $\md$ is the operation defined in \eqref{md} in Appendix \ref{app:Hg}.

Hence $\rho_{n+1, q-1}$ is a vector in \eqref{Hoch-n-k-pr-1}
satisfying Properties \ref{P:one}, \ref{P:anti-symm}. 
Combining this observation with the fact that 
$\rho_{n+1, q-1}$ is $\pa^{\Hoch}$-exact and 
using the second claim in Corollary \ref{cor:Hoch} we conclude 
that 
$$
\rho_{n+1, q-1} = 0.
$$
In other words, $\de\al^{(1)}_{n, q}$ is a cocycle in the cochain 
complex \eqref{Hg} with the differential $\md$ \eqref{md}.

Since $q \ge 1$, Corollary \ref{cor:Hg} from Appendix 
\ref{app:Hg}  implies that there exists a vector (of degree $-1$) 
$$
\psi_{n-1, q+1} \in 
 \bs^{2(n-1) -2  + q+1}\, \big( \KGra(n-1,q+1)^{\mo} \big)^{S_{n-1}} 
$$ 
which satisfies Properties  \ref{P:one}, \ref{P:anti-symm} and such that 
\begin{equation}
\label{md-exact}
\de\al^{(1)}_{n, q}  =  \md(\psi_{n-1, q+1}).
\end{equation}

Using $\psi_{n-1, q+1}$, we define the following degree zero vector 
$$
\xi^{(2)} \in \Conv(\oc^{\vee}_{\circ}, \KGra)
$$
by setting
\begin{equation}
\label{xi-2}
\xi^{(2)} (\bs^{-1}\,\st^{\mo}_{n-1, q+1}) = - \psi_{n-1, q+1}\,,
\qquad 
\xi^{(2)} (\bs^{-1}\,\st^{\mc}_{m_1}) = \xi^{(2)}(\bs^{-1}\, \st^{\mo}_{k_1}) =
\xi^{(2)}(\bs^{-1}\, \st^{\mo}_{m,k}) = 0  
\end{equation}
for all $m_1$, $k_1$ and for all pairs $(m,k) \neq (n-1, q+1)$\,.

It is obvious that 
\begin{equation}
\label{2xi-cF}
\xi^{(2)} \in \cF_{n+q-1} \Conv(\oc^{\vee}, \KGra)
\end{equation}
and 
\begin{equation}
\label{2xi-cF-mc}
\xi^{(2)} \in \cF^{\mc}_{n-1} \Conv(\oc^{\vee}, \KGra).
\end{equation}

Next we consider the MC element
$$
\wt{\al}^{(2)} = \exp(\ad_{\xi^{(2)} }) (\wt{\al}^{(1)})
$$
and rewrite the difference $\de\al^{(2)} : = \wt{\al}^{(2)} - \al$ as follows
\begin{equation}
\label{del-al-2-new}
\de\al^{(2)} =  \exp(\ad_{\xi^{(2)} }) (\de \al^{(1)}) 
+ \exp(\ad_{\xi^{(2)} }) (\al)  -  \al\,.
\end{equation}

Since  $\de\al^{(1)} \in \cF^{\mc}_n \Conv(\oc^{\vee}, \KGra)$ (see \eqref{de-al-1-Fmc})
and $\xi^{(2)} \in \cF^{\mc}_{n-1} \Conv(\oc^{\vee}, \KGra)$, 
\begin{equation}
\label{de-al2-Fmc-1n}
\de\al^{(2)}  \in  \cF^{\mc}_{n-1} \Conv(\oc^{\vee}, \KGra)
\end{equation}
or equivalently $\de\al^{(2)} (\bsi \st^{\mo}_{m,k}) =0$ for all $m < n-1$ and $k \ge 0$. 

Moreover, 
\begin{equation}
\label{de-al2-1n-k}
\de\al^{(2)} (\bsi \st^{\mo}_{n-1,k}) = [\xi^{(2)}, \al](\bsi \st^{\mo}_{n-1,k}).
\qquad \forall ~k \ge 0.
\end{equation}

Since $\xi^{(2)}(\bsi \st^{\mo}_{m,k}) = 0$ for all $(m,k) \neq (n-1, q+1)$, 
\begin{equation}
\label{de-al2-not-q2}
[\xi^{(2)}, \al](\bsi \st^{\mo}_{n-1,k}) = 0 \qquad \forall ~~ k \neq q+2  
\end{equation}
and 
\begin{equation}
\label{at-q-2}
[\xi^{(2)}, \al](\bsi \st^{\mo}_{n-1,q+2}) = - \pa^{\Hoch}  \xi^{(2)} (\bsi \st^{\mo}_{n-1,q+2}).
\end{equation}

On the other hand $\pa^{\Hoch}  \xi^{(2)} (\bsi \st^{\mo}_{n-1,q+2}) = 0$
since the vector $\xi^{(2)} (\bsi \st^{\mo}_{n-1,q+2}) = \psi_{n-1, q+2}$ satisfies Property \ref{P:one}. 
Thus, combining \eqref{de-al2-Fmc-1n} with \eqref{de-al2-1n-k}, \eqref{de-al2-not-q2}
and \eqref{at-q-2},  we deduce that 
$$
\de\al^{(2)}  \in  \cF^{\mc}_{n} \Conv(\oc^{\vee}, \KGra).
$$

To complete Step 2, it remains to show that 
\begin{equation}
\label{level-n}
\de\al^{(2)} (\bsi\st^{\mo}_{n,k}) = 0 \qquad \forall ~~ k \le q\,.
\end{equation}

Since $\al \in  \cF_{0} \Conv(\oc^{\vee}, \KGra)$ (see \eqref{al-in-cF0}), 
$\de\al^{(1)} \in \cF_{n} \Conv(\oc^{\vee}, \KGra)$ and $n+q-1 \ge 2$, 
inclusion \eqref{2xi-cF} implies that 
\begin{equation}
\label{level-n-k-less-q}
\de\al^{(2)} (\bsi\st^{\mo}_{n,k}) = 0 \qquad \forall ~~ k < q
\end{equation}
and 
\begin{equation}
\label{de-al2-at-n-q}
\de\al^{(2)} (\bsi \st^{\mo}_{n,q}) = \de\al^{(1)}_{n,q} 
+  [\xi^{(2)} , \al] \big( \bsi\st^{\mo}_{n,q} \big).
\end{equation}

Unfolding the right hand side of \eqref{de-al2-at-n-q} and 
using the fact that $\psi_{n-1, q+1}$ is anti-symmetric
with respect to permutations of labels on white vertices, we get 
$$
\de\al^{(2)} (\bsi \st^{\mo}_{n,q}) =
\de\al^{(1)}_{n,q}  +
\sum_{p=0}^k 
\sum_{i=1}^n (-1)^p
(\tau_{n,i}, \id) 
\Big( \xi^{(2)} (\bsi\st^{\mo}_{n-1, q+1} ) \circ_{p+1, \mo}
 \al(\bsi\st^{\mo}_{1, 0}) \Big) =
$$
$$
\de\al^{(1)}_{n,q}  -
\sum_{p=0}^{q}
\sum_{i=1}^n (-1)^p
(\tau_{n,i}, \id) 
\big( \psi_{n-1, q+1}  \circ_{p+1, \mo}
 \G^{\br}_0 \big) =
(q+1)  \sum_{i=1}^n 
(\tau_{n,i}, \id) 
\big( \psi_{n-1, q+1}  \circ_{1, \mo}
 \G^{\br}_0 \big),
$$
where $\tau_{n,i}$ is the cycle $(i, i+1, \dots, n-1,n)$ in $S_n$. 

Hence $\de\al^{(2)} (\bsi \st^{\mo}_{n, q})  = 0$ follows from 
equation \eqref{md-exact}.  

Since \eqref{level-n} is proved, Step 2 is complete and so is the proof of Claim \ref{cl:del-al-n-q}.  
\end{proof}

We can now prove the following statement: 
\begin{claim}
\label{cl:del-al-n-k}
Let $\al$ and $\wt{\al}$ be MC elements corresponding to SFQs.
If $\de\al : = \wt{\al} - \al$ belongs to $ \cF^{\mc}_{n} \Conv(\oc^{\vee}, \KGra)$ and 
$$
\de\al(\bsi \st^{\mo}_{n,0}) = 0,
$$
then there exists a degree zero vector $\xi \in \cF^{\mc}_{n-1} \Conv(\oc^{\vee}, \KGra)$
satisfying condition \eqref{xi-cond-mc} and such that
\begin{equation}
\label{now-in-F-mc-n1}
\exp(\ad_{\xi}) (\wt{\al}) - \al ~\in~ \cF^{\mc}_{n+1} \Conv(\oc^{\vee}, \KGra).
\end{equation}
\end{claim}
\begin{proof}
Claim \ref{cl:del-al-n-q} implies that there exists an infinite sequence of degree zero vectors 
$$
\xi_1, \xi_2, \dots, \in  \cF^{\mc}_{n-1} \Conv(\oc^{\vee}, \KGra)
$$
such that each $\xi_r$ satisfies \eqref{xi-cond-mc},
$$
\xi_q \in  \cF_{n+q-2} \Conv(\oc^{\vee}, \KGra) \qquad \forall~~ q \ge 1,
$$
$$
\big( \exp(\ad_{\xi_q}) \dots \exp(\ad_{\xi_1}) (\wt{\al}) - \al  \big) \in \cF^{\mc}_{n} \Conv(\oc^{\vee}, \KGra).
$$
and 
$$
\big( \exp(\ad_{\xi_q}) \dots \exp(\ad_{\xi_1}) (\wt{\al}) \big) (\bsi \st^{\mo}_{n,k}) =  \al (\bsi \st^{\mo}_{n,k})
$$
for all $k \le q$. 

Since the graded Lie algebra  $\cF^{\mc}_{n-1} \Conv(\oc^{\vee}, \KGra)$ is complete with respect to 
the filtration 
$$
\cF_{\bul} \Conv(\oc^{\vee}, \KGra) \,\cap\, \cF^{\mc}_{n-1} \Conv(\oc^{\vee}, \KGra),
$$
the limit $\xi$ of the sequence
$$
\{\,\CH(\xi_q, \CH(\xi_{q-1}, \dots, \CH(\xi_2, \xi_1)..) \,\}_{q \ge 1}
$$
exists in $\cF^{\mc}_{n-1} \Conv(\oc^{\vee}, \KGra)$ and it satisfies \eqref{now-in-F-mc-n1}. 
\end{proof}

\subsection{The end of the proof of Proposition \ref{prop:transit-refined}}
\label{sec:proof-transit}

Let us now put together the results of Sections \ref{sec:de-al-n-0} and \ref{sec:de-al-n-k} to finish 
the proof of Proposition \ref{prop:transit-refined}.

Due to Claim \ref{cl:kill-pikes}, there exists a degree zero vector 
$$
\xi^{\pike} \in  \cF^{\mc}_{n-1} \Conv(\oc^{\vee}_{\circ}, \KGra)
$$
such that $\xi^{\pike}(\bsi\st^{\mc}_m) = 0$ for all $m$, 
$$
\big(\exp(\ad_{\xi^{\pike}}) \wt{\al} \,\big) ~ - ~ \al ~\in~   \cF^{\mc}_{n} \Conv(\oc^{\vee}_{\circ}, \KGra)
$$
and each graph in the difference 
$$
\big(\exp(\ad_{\xi^{\pike}}) \wt{\al} \,\big) (\bsi \st^{\mo}_{n,0}) ~ - ~ \al(\bsi \st^{\mo}_{n,0})
$$
does not have pikes. 

Applying Corollary \ref{cor:killing-del-al-n0} to the MC elements $\al$ and setting
\begin{equation}
\label{wt-al-dia}
\wt{\al}^{\dia} : = \exp(\ad_{\xi^{\pike}})\, \wt{\al},
\end{equation}
we deduce that there exists a degree zero cocycle $\ga \in \cF_{n-1}\dfGC$ such that 
$$
\wt{\al}^{\dia} \quad - \quad   \big( \exp(\ad_{J(\ga)}) \al \big)  
~\in ~ \cF^{\mc}_{n}\Conv(\oc^{\vee}_{\circ}, \KGra)
$$
and 
$$
\wt{\al}^{\dia} (\bsi\st^{\mo}_{n,0}) ~ = ~ 
\big(\exp(\ad_{J(\ga)}) \al \big) (\bsi\st^{\mo}_{n,0}).
$$

Finally, applying Claim \ref{cl:del-al-n-k} to the MC elements 
$$
\al^{\dia} : =  \exp(\ad_{J(\ga)}) \al
$$
and $\wt{\al}^{\dia}$, we deduce that there exists a degree zero 
vector 
$$
\xi^{\sh} \in \cF^{\mc}_{n-1} \Conv(\oc^{\vee}, \KGra)
$$
such that $\xi^{\sh}(\bsi \st^{\mc}_m) = 0$ for all $m$ and 
$$
\exp(\ad_{\xi^{\sh}}) (\wt{\al}^{\dia}) - \al^{\dia} ~\in~ \cF^{\mc}_{n+1} \Conv(\oc^{\vee}, \KGra).
$$

Thus, setting $\xi : = \CH(\xi^{\sh}, \xi^{\pike})$, we get that 
$$
\big( \exp(\ad_{\xi}) (\wt{\al}) \big) ~ - ~  \big( \exp(\ad_{J(\ga)}) \al \big) 
~\in~ \cF^{\mc}_{n+1} \Conv(\oc^{\vee}, \KGra).
$$
In other words, \eqref{new-difference} is satisfied. 

The proof of transitivity of the action of  $\exp\big(H^0(\dfGC)\big)$ on homotopy classes of SFQs 
is now complete. 

%
%

\section{The action of $\exp\big(H^0(\dfGC)\big)$ is free}
\label{sec:free}
 
Let $\al$ be a MC element of  $\Conv(\oc^{\vee}, \KGra)$ representing an SFQ
and $\ga$ be a degree zero cocycle in $\dfGC$. Let us assume that there exists a degree zero vector
\begin{equation}
\label{the-xi}
\xi \in \cF_1\Conv(\oc^{\vee}, \KGra)
\end{equation}
which satisfies condition \eqref{xi-cond-mc} and such that 
\begin{equation}
\label{vo-kak}
\exp\big(\ad_{J(\ga)}\big) \al = \exp\big(\ad_{\xi}\big) \al\,.
\end{equation}
Our goal is to show that $\ga$ is exact. 

Due to Remark \ref{rem:no-pikes}, we assume, without loss of generality, 
that we deal exclusively with cocycles $\ga$ of $\dfGC$ which do not involve 
graphs with pikes. 

To prove that $\ga$ is exact, we will need the following technical claims
which are proved below in Sections \ref{sec:killing-pikes-xi-proof}
and \ref{sec:exact-refined-proof}, respectively.  
\begin{claim}
\label{cl:killing-pikes-xi}
Let $n$ be an integer $\ge 2$, $\al$ be a MC element of  $\Conv(\oc^{\vee}, \KGra)$ 
corresponding to an SFQ  and
\begin{equation}
\label{xi-in-cFmc-n}
\xi \in \cF^{\mc}_{n}\Conv(\oc^{\vee}, \KGra) 
\end{equation}
be a degree zero vector satisfying condition \eqref{xi-cond-mc}. 
There exists a degree zero vector 
$$
\ti{\xi} \in \cF^{\mc}_{n}\Conv(\oc^{\vee}, \KGra) 
$$
for which \eqref{xi-cond-mc} holds, all graphs in 
$$
\ti{\xi} (\bsi \st^{\mo}_{n,0})
$$
do not have pikes and 
\begin{equation}
\label{ti-xi-works}
\exp\big(\ad_{\ti{\xi}}\big) \al =  \exp\big(\ad_{\xi}\big) \al.
\end{equation}
\end{claim}
\begin{claim}
\label{cl:exact-refined}
Let $\ga$ be a degree zero cocycle in $\dfGC$, $n$ be an integer $\ge 2$ and   
\begin{equation}
\label{xi-cF-n}
\xi \in \cF^{\mc}_{n}\Conv(\oc^{\vee}, \KGra)
\end{equation}
be a degree zero vector satisfying condition \eqref{xi-cond-mc}. 
If equation \eqref{vo-kak} holds and all graphs in $\xi(\bsi \st^{\mo}_{n,0})$ do not have pikes,
then\footnote{i.e. all graphs in $\ga$ have $\ge (n+1)$ vertices.} 
\begin{equation}
\label{ga-in-cF-n}
\ga \in \cF_{n} \dfGC.
\end{equation}
Moreover, there exist a degree $-1$ vector $\ka \in  \cF_{n-1} \dfGC$
and a degree $0$ vector 
$$
\eta \in \cF^{\mc}_{n} \Conv(\oc^{\vee}, \KGra)
$$
satisfying \eqref{xi-cond-mc}, 
\begin{equation}
\label{vo-kak-better}
\exp\big(\ad_{J( \CH(\pa \ka, \ga ) )}\big) \al =  \exp\big(\ad_{ \CH(\eta, \xi) }\big) \al,
\end{equation}
\begin{equation}
\label{better-inclusions}
\CH(\pa \ka, \ga ) \in \cF_{n+1} \dfGC  \qquad \textrm{ and } \qquad
\CH(\eta, \xi) \in  \cF^{\mc}_{n+1}\Conv(\oc^{\vee}, \KGra).
\end{equation}
\end{claim}

%
%

\subsection{Claims \ref{cl:killing-pikes-xi} and \ref{cl:exact-refined} imply that the action is free}
\label{sec:proof-freeness}

Claims \ref{cl:killing-pikes-xi} and \ref{cl:exact-refined} imply that, 
for every degree $0$ cocycle $\ga$ satisfying \eqref{vo-kak}, 
there exists a sequence 
\begin{equation}
\label{seq-kappa}
\{\ka_{m} \}_{m \ge 1}, \qquad \ka_{m}  \in  \cF_{m} \dfGC
\end{equation}
such that for every $n \ge 1$
$$
\CH(\pa \ka_{n}, \dots, \CH(\pa \ka_2, \CH(\pa \ka_1,  \ga)\dots ) \in  \cF_{n+2} \dfGC\,.
$$

Since $\dfGC$ is complete with respect to the filtration $\cF_{\bul}\dfGC$, the 
existence of this sequence implies that $\ga$ is indeed a coboundary. 

\subsection{Proof of Claim \ref{cl:killing-pikes-xi}} 
\label{sec:killing-pikes-xi-proof}

If $\xi (\bsi \st^{\mo}_{n,0})$ does not involve graphs with pikes then we set $\ti{\xi} : = \xi$
and equation \eqref{ti-xi-works} obviously holds. 

Otherwise, we observe that, since $[\al, \al] =0$, $\ad_{[\psi, \al]}$ acts trivially on $\al$
for every degree $-1$ vector $\psi \in  \cF_1\Conv(\oc^{\vee}_{\circ}, \KGra)$. 
Hence, we have 
\begin{equation}
\label{same-fig}
\exp\left( \ad_{\CH(\xi, [\psi, \al])} \right) (\al) = \exp(\ad_{\xi}) (\al).
\end{equation}

We will prove that there exists a degree $-1$ vector 
\begin{equation}
\label{psi-needed}
\psi \in \cF^{\mc}_{n-1}\Conv(\oc^{\vee}_{\circ}, \KGra) 
\end{equation}
such that the element 
\begin{equation}
\label{new-xi-Fmc-m}
\ti{\xi} : = \CH(\xi, [\psi, \al]) 
\end{equation}
\begin{itemize}

\item belongs to $\cF^{\mc}_{n} \Conv(\oc^{\vee}, \KGra)$,  

\item satisfies condition \eqref{xi-cond-mc}, and

\item all graphs in $\ti{\xi} (\bsi \st^{\mo}_{n,0})$ do not have pikes.

\end{itemize}

Let us set $\xi_{n,0} : = \xi(\bsi \st^{\mo}_{n,0}). $

Since the graphs in $\xi_{n,0}$ do not have white vertices, 
the vector $\xi_{n,0}$ is a cocycle in the complex \eqref{Hg} with 
the differential $\md$ \eqref{md} (see Appendix \ref{app:Hg}).

Let us denote by $\xi^r_{n,0}$ the linear combination in $\KGra(n,0)^{\mo}$ 
which is obtained from $\xi_{n,0}$ by retaining only the graphs with 
exactly $r$ pikes. According to Lemma \ref{lem:Koszul} from Appendix 
\ref{app:Hg}, we have 
$$
\md \md^* (\xi^r_{n,0}) = r \xi^r_{n,0}\,.
$$
Thus, if 
\begin{equation}
\label{psi-1n-1}
\psi_{n-1,1} : = - \sum_{r \ge  1} \frac{1}{r} \md^* (\xi^r_{n,0}), 
\end{equation}
then each graph in the linear combination
$$
\xi_{n,0} + \md (\psi_{n-1,1})
$$
does not have pikes. 

Next, we define a degree $-1$ vector \eqref{psi-needed} by setting
\begin{equation}
\label{psi-mixed}
\psi(\bsi\st^{\mo}_{n-1,1}) = \psi_{n-1,1}\,,
\qquad 
\psi(\bsi\st^{\mo}_{n_2, k_2}) = 0 \quad \forall ~~ (n_2, k_2) \neq (n-1,1),
\end{equation}
and 
\begin{equation}
\label{psi-homog}
\psi(\bsi\st^{\mc}_{n_1}) = \psi(\bsi\st^{\mo}_{k_1}) = 0 \qquad 
\forall ~~ n_1, k_1\ge 2.
\end{equation}

Then we consider the vector 
\begin{equation}
\label{xi-new}
\ti{\xi} : =   \CH(\xi,  [\psi, \al]).
\end{equation}

By construction \eqref{psi-1n-1}, all white vertices in graphs in $\psi_{n-1, 1}$
have valency one. Hence  $\psi_{n-1, 1}$ belongs to the kernel of 
the differential $\pa^{\Hoch}$ \eqref{pa-Hoch}. Using this fact, \eqref{psi-mixed} and \eqref{psi-homog}, 
it is not hard to show that 
\begin{equation}
\label{psi-al-Fmc-n}
[\psi, \al] \in \cF^{\mc}_{n} \Conv(\oc^{\vee}, \KGra)\,.
\end{equation}
Therefore, 
$$
\ti{\xi}  \in  \cF^{\mc}_{n} \Conv(\oc^{\vee}, \KGra).
$$

Equation \eqref{psi-homog} implies that $[\psi, \al](\st^{\mc}_{n_1}) = 0 $ for all $ n_1 \ge 2.$
Combining this observation with the fact that $\xi$ satisfies \eqref{xi-cond-mc}, 
we conclude that $\ti{\xi}$ also satisfies \eqref{xi-cond-mc}. 

Using \eqref{psi-mixed}, \eqref{psi-homog}, \eqref{psi-al-Fmc-n}, and the inequality $n \ge 2$, we get 
$$
\CH(\xi,  [\psi, \al])\, (\bsi \st^{\mo}_{n,0}) =
\xi(\bsi \st^{\mo}_{n, 0})  +   [\psi, \al](\bsi \st^{\mo}_{n,0}) =
$$
$$
\xi_{n,0} + \sum_{i=1}^{n} \tau_{n,i}
\big(\psi(\bsi \st^{\mo}_{n-1,1}) \circ_{\mo,1} \al(\bsi \st^{\mo}_{1,0}) \big) =
$$
$$
= \xi_{n,0} + \sum_{i=1}^{n} \tau_{n,i}
\big(\psi_{n-1,1} \circ_{\mo,1} \G^{\br}_0 \big)
$$
where $\tau_{n,i}$ is the cycle $(i, i+1, \dots, n-1, n)$ in $S_n$. 

Thus, by definition of the operator $\md$ \eqref{md}, we deduce that 
$$
\ti{\xi}(\bsi \st^{\mo}_{n,0}) = \xi_{n,0} + \md \psi_{n-1,1}\,.
$$
Since each graph in the linear combination $\xi_{n,0} + \md \psi_{n-1,1}$
does not have pikes, Claim \ref{cl:killing-pikes-xi} is proved.

\subsection{Proof of Claim \ref{cl:exact-refined}}
\label{sec:exact-refined-proof}

Let $m$ be an integer $\le n$ such that 
\begin{equation}
\label{ga-less-m}
\ga (1_k) = 0 \qquad \forall ~~k < m,
\end{equation}
i.e. $\ga \in \cF_{m-1} \dfGC$. 

Due to \eqref{ga-less-m}, 
$$
\big(\exp\big(\ad_{J(\ga)}\big) \al \big) (\bsi \st^{\mo}_{m,0}) =  
\al (\bsi \st^{\mo}_{m,0}) - \al (\bsi \st^{\mo}_{1,0}) \circ_{1,\mc} \ga (1_m) =
$$
$$
\al (\bsi \st^{\mo}_{m,0}) - \G^{\br}_{1,0}   \circ_{1,\mc} \ga (1_m),
$$
i.e.
\begin{equation}
\label{exp-ga--al-st-m}
\big(\exp\big(\ad_{J(\ga)}\big) \al \big) (\bsi \st^{\mo}_{m,0}) = \al (\bsi \st^{\mo}_{m,0}) - \ga (1_m).
\end{equation}

Since $m \le n$,  $\xi \in   \cF^{\mc}_{n}\Conv(\oc^{\vee}, \KGra)$ 
and $\xi$ satisfies \eqref{xi-cond-mc}, we have 
$$
[\xi,  \ad^k_{\xi} (\al) ] (\bsi \st^{\mo}_{m,0})  = 0 \qquad \forall ~~ k \ge 0 
$$
and hence 
\begin{equation}
\label{exp-xi--al-st-m}
\big( \exp\big(\ad_{\xi}\big) \al \big)  (\bsi \st^{\mo}_{m,0})  = \al  (\bsi \st^{\mo}_{m,0}). 
\end{equation}

Combining \eqref{vo-kak}, \eqref{exp-ga--al-st-m}, and \eqref{exp-xi--al-st-m}, 
we conclude that 
$$
\ga (1_k) = 0 
$$ 
for all $k \le m$. 

Thus inclusion \eqref{ga-in-cF-n} indeed holds, i.e. 
\begin{equation}
\label{ga-in-cF-n-again}
\ga(1_k) = 0 \qquad \forall ~~ k \le n.  
\end{equation}

Let us deduce from \eqref{vo-kak} that 
\begin{claim}
\label{cl:xi-n-1-white}
The white vertex in every graph in 
$$
\xi (\bsi \st^{\mo}_{n,1}) 
$$
has valency $1$. 
\end{claim}
\begin{proof}[ of Claim \ref{cl:xi-n-1-white}]
Evaluating both sides of \eqref{vo-kak} on $\bsi \st^{\mo}_{n, 2}$, and using  
\eqref{xi-cond-mc}, \eqref{xi-cF-n} and \eqref{ga-in-cF-n}, we deduce that
$$
\al (\bsi \st^{\mo}_{n,2})  =  \al (\bsi \st^{\mo}_{n,2}) + [\xi, \al]  (\bsi \st^{\mo}_{n,2}). 
$$
Hence, 
$$
 [\xi, \al]  (\bsi \st^{\mo}_{n,2}) = 0
$$
or equivalently 
\begin{equation}
\label{pa-Hoch-xi-n-2}
\pa^{\Hoch} \big( \xi (\bsi \st^{\mo}_{n,1}) \big) = 0, 
\end{equation}
where $\pa^{\Hoch}$ is defined in \eqref{pa-Hoch}. 

Combining \eqref{pa-Hoch-xi-n-2} with Corollary \ref{cor:k-one}
from Appendix \ref{app:Hoch}, we conclude that the white 
vertex in each graph in $ \xi (\bsi \st^{\mo}_{n,1})$ must have 
valency $1$.

Thus Claim \ref{cl:xi-n-1-white} is proved.
\end{proof}

We will now deduce Claim \ref{cl:exact-refined} by evaluating 
both sides of \eqref{vo-kak} on $\bsi \st^{\mo}_{n+1, 0}$. 

Using \eqref{ga-in-cF-n-again}, it is easy to show that 
\begin{equation}
\label{ga-1n1-shows-up}
\big(\exp (\ad_{J(\ga)}) \al \big) (\bsi \st^{\mo}_{n+1,0}) = \al (\bsi \st^{\mo}_{n+1,0}) - \ga (1_{n+1}).
\end{equation}

On the other hand, using \eqref{xi-cF-n} and \eqref{xi-cond-mc}, we get that
$$
\big( \exp(\ad_{\xi}) \al \big)  (\bsi \st^{\mo}_{n+1,0}) = \al (\bsi \st^{\mo}_{n+1,0})
- \sum_{\tau \in \Sh_{2, n-1}} \tau \big(  \xi ( \bsi \st^{\mo}_{n,0} ) \circ_{\mc, 1} \al (\bsi \st^{\mc}_2)  \big)
$$
$$
- \sum_{i = 1}^{n+1}  \si_{n+1, i } \big(  \al ( \bsi \st^{\mo}_{1, 1} )   \circ_{\mo, 1}   \xi ( \bsi \st^{\mo}_{n, 0}) \big)
+ \sum_{i = 1}^{n+1} \tau_{n+1, i}  \big(   \xi ( \bsi \st^{\mo}_{n, 1} ) \circ_{\mo, 1}   \al (\bsi \st^{\mo}_{1,0})   \big).
$$
Hence,
$$
\big( \exp(\ad_{\xi}) \al \big)  (\bsi \st^{\mo}_{n+1,0}) = \al (\bsi \st^{\mo}_{n+1,0}) 
- \sum_{\tau \in \Sh_{2, n-1}} \tau  \big(  \xi_{n,0}  \circ_{\mc, 1} \G_{\ed} \big)
$$
\begin{equation}
\label{exp-ad-xi-al}
- \sum_{i = 1}^{n+1}  \si_{n+1, i } \big(  \G^{\br}_1 \circ_{\mo, 1}   \xi_{n,0}  \big)
+ \sum_{i = 1}^{n+1} \tau_{n+1, i}  \big(   \xi_{n,1} \circ_{\mo, 1}  \G^{\br}_0   \big),
\end{equation}
where $\xi_{n,0} : =   \xi ( \bsi \st^{\mo}_{n, 0} ) $ and $ \xi_{n,1} : =  \xi ( \bsi \st^{\mo}_{n, 1} ) $. 

Combining \eqref{ga-1n1-shows-up} with \eqref{exp-ad-xi-al}, we conclude that 
\begin{equation}
\label{master-ga-n1}
\ga (1_{n+1}) = \sum_{\tau \in \Sh_{2, n-1}} \tau  \big(  \xi_{n,0}  \circ_{\mc, 1} \G_{\ed} \big)
+ \sum_{i = 1}^{n+1}  \si_{n+1, i } \big(  \G^{\br}_1 \circ_{\mo, 1}   \xi_{n,0}  \big)
- \sum_{i = 1}^{n+1} \tau_{n+1, i}  \big(   \xi_{n,1} \circ_{\mo, 1}  \G^{\br}_0   \big).
\end{equation}

By Claim \ref{cl:xi-n-1-white}, the white vertex of every graph in $\xi_{n,1}$
has valency $1$. Hence every graph in the last sum in the right hand side of 
\eqref{master-ga-n1} has a pike. Therefore, since neither $\ga (1_{n+1}) $ nor 
$ \xi_{n,0}$ involve graphs with pikes, the linear combination
$$
\sum_{i = 1}^{n+1} \tau_{n+1, i}  \big(   \xi_{n,1} \circ_{\mo, 1}  \G^{\br}_0   \big)
$$
is obtained from 
$$
\sum_{\tau \in \Sh_{2, n-1}} \tau  \big(  \xi_{n,0}  \circ_{\mc, 1} \G_{\ed} \big)
$$
by keeping only graphs with pikes. 

Thus the right hand side of  \eqref{master-ga-n1} equals
$$
[\G_{\ed}, \xi_{n,0}],
$$
where $\xi_{n,0}$ is viewed as a vector in $\dfGC$. 

We set 
$$
\ka : = - \xi_{n,0}
$$ 
and recall that, due to the second part of Proposition \ref{prop:to-homotopy}, 
there exists a degree $0$ vector 
\begin{equation}
\label{eta}
\eta \in  \cF^{\mc}_{n}\Conv(\oc^{\vee}, \KGra)
\end{equation}
which satisfies \eqref{xi-cond-mc},
\begin{equation}
\label{eta-n-0}
\eta(\bsi \st^{\mo}_{n,0}) =  - \xi_{n,0}\,,
\end{equation}
and such that \eqref{vo-kak-better} holds. 

The desired inclusions in \eqref{better-inclusions} follow easily from 
$\ga + \pa \xi_{n,0} = 0$, \eqref{xi-cF-n}, \eqref{ga-in-cF-n}, \eqref{eta}, 
and \eqref{eta-n-0}. 

Claim \ref{cl:exact-refined} is proved. 

We showed that the action of $\exp\big( H^0(\dfGC) \big)$ on homotopy classes of SFQs is free. 
Thus Theorem \ref{thm:main} is proved.

\appendix

\section{A cochain complex that is closely connected with the Hochschild complex of a cofree cocommutative coalgebra}
\label{app:Hoch}
In this appendix we compute the cohomology of the cochain 
complex 
\begin{equation}
\label{Hoch-app}
\KGra^{\Hoch}_{\inv} =
\bs^{2n-2} \, \bigoplus_{k \ge 0} \bs^{k}\big( \KGra(n,k)^{\mo} \big)^{S_n}
\end{equation}
with the differential $\pa^{\Hoch}$ given by the formula
\begin{equation}
\label{pa-Hoch-app}
\pa^{\Hoch} (\ga) = 
\G_{\ww} \circ_{2, \mo} \ga - \ga \circ_{1, \mo} \G_{\ww} + 
\ga \circ_{2, \mo} \G_{\ww} - \dots 
\end{equation}
$$
+ (-1)^{k} \ga \circ_{k, \mo} \G_{\ww}
+ (-1)^{k+1}\G_{\ww} \circ_{1, \mo} \ga\,,
$$
$$
\ga \in  \bs^{2n-2 + k}\big( \KGra(n,k)^{\mo}  \big)^{S_n}\,.
$$

For this purpose we consider a slightly simpler cochain complex 
\begin{equation}
\label{Hoch-simpler}
\KGra^{\Hoch} =
\bs^{2n-2} \, \bigoplus_{k \ge 0} \bs^{k} \KGra(n,k)^{\mo} 
\end{equation}
with the differential $\pa^{\Hoch}$ defined by the same formula
\eqref{pa-Hoch-app}. 

The cochain complex \eqref{Hoch-simpler} is equipped with 
the obvious action of the group $S_n$ and  \eqref{Hoch-app}
is nothing but the complex of $S_n$-invariants.

\begin{example}
\label{ex:pa-Hoch}
An example of computation of $\pa^{\Hoch} (\G)$ for a graph 
$\G \in \dgra_{3,1}$ is shown in figure \ref{fig:pa-Hoch}.
\begin{figure}[htp]
\begin{minipage}[t]{0.99\linewidth}
\centering 
\begin{tikzpicture}[scale=0.5, >=stealth']
\tikzstyle{w}=[circle, draw, minimum size=4, inner sep=1]
\tikzstyle{b}=[circle, draw, fill, minimum size=4, inner sep=1]
\draw (-1.5,1) node[anchor=center] {{$\pa^{\Hoch}$}};
\node [b] (b1) at (0,1) {};
\draw (0,1.6) node[anchor=center] {{\small $1$}};
\node [b] (b3) at (1,2) {};
\draw (1,2.6) node[anchor=center] {{\small $3$}};
\node [b] (b2) at (2,1) {};
\draw (2,1.6) node[anchor=center] {{\small $2$}};
\node [w] (w1) at (1,0) {};
\draw (1,-0.6) node[anchor=center] {{\small $1$}};
\draw [->] (b1) edge (b3);
\draw [->] (b1) edge (w1);
\draw [->] (b3) edge (w1);
\draw [->] (b2) edge (w1);
\draw (3.5,1) node[anchor=center] {{$= ~~ -$}};
\node [b] (bb1) at (5,1) {};
\draw (5,1.6) node[anchor=center] {{\small $1$}};
\node [b] (bb3) at (6,2) {};
\draw (6,2.6) node[anchor=center] {{\small $3$}};
\node [b] (bb2) at (7,1) {};
\draw (7,1.6) node[anchor=center] {{\small $2$}};
\node [w] (ww1) at (5.5,0) {};
\draw (5.5,-0.6) node[anchor=center] {{\small $1$}};
\node [w] (ww2) at (7,0) {};
\draw (7,-0.6) node[anchor=center] {{\small $2$}};
\draw [->] (bb1) edge (bb3);
\draw [->] (bb1) edge (ww1);
\draw [->] (bb3) edge (ww1);
\draw [->] (bb2) edge (ww2);
\draw (8.5,1) node[anchor=center] {{$-$}};
\node [b] (bbb1) at (10,1) {};
\draw (10,1.6) node[anchor=center] {{\small $1$}};
\node [b] (bbb3) at (11,2) {};
\draw (11,2.6) node[anchor=center] {{\small $3$}};
\node [b] (bbb2) at (12,1) {};
\draw (12,1.6) node[anchor=center] {{\small $2$}};
\node [w] (www1) at (10.5,0) {};
\draw (10.5,-0.6) node[anchor=center] {{\small $1$}};
\node [w] (www2) at (12,0) {};
\draw (12,-0.6) node[anchor=center] {{\small $2$}};
\draw [->] (bbb1) edge (bbb3);
\draw [->] (bbb1) edge (www1);
\draw [->] (bbb3) edge (www2);
\draw [->] (bbb2) edge (www1);
\draw (13.5,1) node[anchor=center] {{$-$}};
\node [b] (1b) at (15,1) {};
\draw (15,1.6) node[anchor=center] {{\small $1$}};
\node [b] (3b) at (16,2) {};
\draw (16,2.6) node[anchor=center] {{\small $3$}};
\node [b] (2b) at (17,1) {};
\draw (17,1.6) node[anchor=center] {{\small $2$}};
\node [w] (1w) at (15.5,0) {};
\draw (15.5,-0.6) node[anchor=center] {{\small $1$}};
\node [w] (2w) at (17,0) {};
\draw (17,-0.6) node[anchor=center] {{\small $2$}};
\draw [->] (1b) edge (3b);
\draw [->] (1b) edge (2w);
\draw [->] (3b) edge (1w);
\draw [->] (2b) edge (1w);
\end{tikzpicture}
\end{minipage}
\begin{minipage}[t]{0.99\linewidth}
\centering 
\begin{tikzpicture}[scale=0.5, >=stealth']
\tikzstyle{w}=[circle, draw, minimum size=4, inner sep=1]
\tikzstyle{b}=[circle, draw, fill, minimum size=4, inner sep=1]
\draw (3.5,1) node[anchor=center] {{$-$}};
\node [b] (bb1) at (5,1) {};
\draw (5,1.6) node[anchor=center] {{\small $1$}};
\node [b] (bb3) at (6,2) {};
\draw (6,2.6) node[anchor=center] {{\small $3$}};
\node [b] (bb2) at (7,1) {};
\draw (7,1.6) node[anchor=center] {{\small $2$}};
\node [w] (ww1) at (5.5,0) {};
\draw (5.5,-0.6) node[anchor=center] {{\small $2$}};
\node [w] (ww2) at (7,0) {};
\draw (7,-0.6) node[anchor=center] {{\small $1$}};
\draw [->] (bb1) edge (bb3);
\draw [->] (bb1) edge (ww1);
\draw [->] (bb3) edge (ww1);
\draw [->] (bb2) edge (ww2);
\draw (8.5,1) node[anchor=center] {{$-$}};
\node [b] (bbb1) at (10,1) {};
\draw (10,1.6) node[anchor=center] {{\small $1$}};
\node [b] (bbb3) at (11,2) {};
\draw (11,2.6) node[anchor=center] {{\small $3$}};
\node [b] (bbb2) at (12,1) {};
\draw (12,1.6) node[anchor=center] {{\small $2$}};
\node [w] (www1) at (10.5,0) {};
\draw (10.5,-0.6) node[anchor=center] {{\small $2$}};
\node [w] (www2) at (12,0) {};
\draw (12,-0.6) node[anchor=center] {{\small $1$}};
\draw [->] (bbb1) edge (bbb3);
\draw [->] (bbb1) edge (www1);
\draw [->] (bbb3) edge (www2);
\draw [->] (bbb2) edge (www1);
\draw (13.5,1) node[anchor=center] {{$-$}};
\node [b] (1b) at (15,1) {};
\draw (15,1.6) node[anchor=center] {{\small $1$}};
\node [b] (3b) at (16,2) {};
\draw (16,2.6) node[anchor=center] {{\small $3$}};
\node [b] (2b) at (17,1) {};
\draw (17,1.6) node[anchor=center] {{\small $2$}};
\node [w] (1w) at (15.5,0) {};
\draw (15.5,-0.6) node[anchor=center] {{\small $2$}};
\node [w] (2w) at (17,0) {};
\draw (17,-0.6) node[anchor=center] {{\small $1$}};
\draw [->] (1b) edge (3b);
\draw [->] (1b) edge (2w);
\draw [->] (3b) edge (1w);
\draw [->] (2b) edge (1w);
\end{tikzpicture}
\end{minipage}
\caption{Computing $\pa^{\Hoch}$} \label{fig:pa-Hoch}
\end{figure}
Let us say that we chose this order
$(1_{\mc}, 3_{\mc}) <  (1_{\mc}, 1_{\mo}) <  (2_{\mc}, 1_{\mo})<
(3_{\mc}, 1_{\mo})$  on the set of 
edges of $\G$\,. 
The orders on the sets of edges of graphs in the right hand side 
are inherited from the total order on the edges of $\G$
in the obvious way. 
For example, the first graph in the sum on the right hand side 
has its edges ordered this way: $
(1_{\mc}, 3_{\mc}) <  (1_{\mc}, 1_{\mo}) <  (2_{\mc}, 2_{\mo})<
(3_{\mc}, 1_{\mo})\,.$
\end{example}

Before computing the cohomology of \eqref{Hoch-simpler}
let us make a couple of remarks about vectors 
\begin{equation}
\label{c}
c \in  \bs^{2n-2+k} \KGra(n,k)^{\mo}  
\quad \textrm{or} \quad
 c \in \bs^{2n-2 + k}\big( \KGra(n,k)^{\mo}  \big)^{S_n}
\end{equation}
satisfying these two properties:
\begin{pty}
\label{P:one}
All white vertices in each graph of the 
linear combination $c$ have valency one.
\end{pty}
\begin{pty}
\label{P:anti-symm}
For every $\si \in S_k$ we have  
\begin{equation}
\label{app-anti-symm}
(\id , \si)\, (c) = (-1)^{|\si|} c\,.
\end{equation}
\end{pty} 
For example, the ``brooms'' $\G^{\br}_k$ depicted 
in figure \ref{fig:brooms} obviously satisfy these properties.

\begin{remark}
\label{rem:Hoch-closed}
It is easy to see that every vector  \eqref{c} 
satisfying Properties \ref{P:one} and \ref{P:anti-symm}
is closed with respect to $\pa^{\Hoch}$\,. 
Furthermore, it is not hard to see that a cocycle $c$ 
satisfying Properties \ref{P:one} and \ref{P:anti-symm}
is trivial if and only if $c=0$\,.
\end{remark}

\subsection{The Hochschild complex of a cofree cocommutative coalgebra}

To compute the cohomology of \eqref{Hoch-simpler}
we consider the cofree cocommutative $\bbK$-coalgebra $\C_r$ with counit 
co-generated by degree $0$ elements $h_1, h_2, \dots, h_r$\,.

To the coalgebra $\C_r$ we assign the following 
cochain complex
\begin{equation}
\label{Hoch-C}
\Hoch(\C_r) =  \bigoplus_{k\ge 0} \bs^k \big( \C_r \big)^{\otimes \, k}
\end{equation}
with the differential 
$$
\pa^{\C} :  \big( \C_r \big)^{\otimes \, k} \to  \big( \C_r \big)^{\otimes \, (k+1)}
$$
given by the formula
\begin{equation}
\label{pa-Hoch-C}
\pa^{\C} (X) = 1 \otimes X +
\sum_{i =1}^k (-1)^i (\id, \dots, \id, \underbrace{\D}_{i\textrm{-th spot}},
 \id, \dots, \id) (X)
 + (-1)^{k+1} X \otimes 1\,,    
\end{equation}
where $\D$ denotes the comultiplication on $\C_r$\,. 

The complex $\Hoch(\C_r)$ obviously splits into 
the direct sum of sub-complexes 
\begin{equation}
\label{Hoch-C-splits}
\Hoch(\C_r) =  \bigoplus_{m\ge 0} \Hoch(\C_r)_m\,,
\end{equation}
where $ \Hoch(\C_r)_m$ is spanned by tensor monomials
with the total degree in co-generators being $m$\,. 

In \cite[Section 4.6.1.1]{K} it was proved that 
\begin{claim}[Section 4.6.1.1, \cite{K}]
\label{cl:Hoch-C}
If $X$ is a cocycle in 
$$
 \bs^k \big( \C_r \big)^{\otimes \, k} ~\cap~ \Hoch(\C_r)_m
$$
and $m \neq k$ then
$X$ is $\pa^{\C}$-exact. 
Furthermore, if  $X$ is a cocycle in 
$$
 \bs^k \big( \C_r \big)^{\otimes \, k} ~\cap~ \Hoch(\C_r)_m
$$
and $m = k$ then
there exists 
$$
\wt{X} \in \bs^{k-1} \big( \C_r \big)^{\otimes \, (k-1)} ~\cap~
 \Hoch(\C_r)_m
$$
such that 
$$
X - \pa^{\C} (\wt{X})  = 
\sum_{i_1 i_2 \dots i_k} 
\la^{i_1 i_2 \dots i_k}  (h_{i_1}, h_{i_2}, \dots, h_{i_k}) \,,
$$
where $\la^{i_1 i_2 \dots i_k} \in \bbK$ and
$$
\la^{\dots  i_{p} i_{p+1} \dots } = - 
\la^{\dots  i_{p+1} i_{p} \dots }\,. 
$$
Finally a cocycle of the form 
$$
\sum_{i_1 i_2 \dots i_k} 
\la^{i_1 i_2 \dots i_k}  (h_{i_1}, h_{i_2}, \dots, h_{i_k})\,, 
\qquad 
\la^{\dots  i_{p} i_{p+1} \dots } = - 
\la^{\dots  i_{p+1} i_{p} \dots } \in \bbK
$$
is exact if and only if all coefficients $\la^{i_1 i_2 \dots i_k} = 0$\,. $\Box$
\end{claim}

For our purposes we will need the following 
subcomplex of $\Hoch(\C_r)$: 
\begin{equation}
\label{Hoch-pr-C}
\Hoch'(\C_r) =  \Big\{ X\in \Hoch(\C_r)_r  ~\big|~ 
\textrm{each co-generator $h_i$ appears} 
\end{equation}
$$
\textrm{in the tensor 
monomial } X \textrm{ exactly once}\Big\}\,.
$$

Using Claim \ref{cl:Hoch-C} about cocycles
in $\Hoch(\C_r)$ it is easy to deduce an analogous 
statement for the cochain complex  $\Hoch'(\C_r)$: 
\begin{claim}
\label{cl:Hoch-pr}
If $X$ is a cocycle in 
$$
 \bs^k \big( \C_r \big)^{\otimes \, k} ~\cap~ \Hoch'(\C_r)
$$
and $k \neq r$ then $X$ is $\pa^{\C}$-exact. 
Furthermore, if  $X$ is a cocycle in 
$$
 \bs^k \big( \C_r \big)^{\otimes \, k} ~\cap~ \Hoch'(\C_r)
$$
and $k = r$ then there exists 
$$
\wt{X} \in \bs^{k-1} \big( \C_r \big)^{\otimes \, (k-1)} ~\cap~
 \Hoch'(\C_r)
$$
such that 
$$
X - \pa^{\C} (\wt{X})  = 
\sum_{\si \in S_r} (-1)^{|\si|} \la (h_{\si(1)}, h_{\si(2)}, \dots, h_{\si(r)}) \,,
$$
for $\la \in \bbK$\,.
Finally, the cocycle 
$$
\sum_{\si \in S_r} (-1)^{|\si|} (h_{\si(1)}, h_{\si(2)}, \dots, h_{\si(r)}) 
$$
is non-trivial. $\Box$ 
\end{claim}

\subsection{Computing cohomology of  $\KGra^{\Hoch}$ and 
$\KGra^{\Hoch}_{\inv}$}
Let us now return to the cochain complex $\KGra^{\Hoch}$ \eqref{Hoch-simpler}\,.
  
It is clear that  $\KGra^{\Hoch}$ splits into the direct sum
of sub-complexes  
\begin{equation}
\label{KGra-Hoch-splits}
\KGra^{\Hoch} = \bigoplus_r  \KGra^{\Hoch}_r 
\end{equation}
where $ \KGra^{\Hoch}_r $ is spanned by graphs with 
exactly $r$ edges terminating at white vertices.

To compute the cohomology of  $ \KGra^{\Hoch}_r $ we 
introduce an auxiliary subspace: 
\begin{equation}
\label{KGra-pr}
\KGra'(n,r) \subset \KGra(n,r) 
\end{equation}
which consists of linear combinations of graphs in 
$\dgra_{n,r}$ with all white vertices (if any) having valency $1$\,.

Let us now suppose that we are given 
a tensor monomial with $k$ factors
\begin{equation}
\label{monom}
X =  h_{i_{11}} h_{i_{12}} \dots  h_{i_{1r_1}} \otimes 
h_{i_{21}} h_{i_{22}} \dots h_{i_{2r_2}} \otimes 
\dots  \otimes 
h_{i_{k1}} h_{i_{k2}} \dots  h_{i_{kr_k}}  \in \Hoch'(\C_r) 
\end{equation}
and a graph $\G' \in \dgra_{n,r}$ with all white vertices 
having valency $1$\,. To the pair $(X, \G')$ we 
assign a graph $\G\in \dgra_{n,k}$ following these steps:

\begin{itemize}

\item First, for each $i \in \{1,2, \dots, r\}$ we find 
the number of the tensor factor in \eqref{monom}
which contains the 
co-generator\footnote{Recall that each co-generator 
$h_i$ enters the monomial \eqref{monom} exactly once.} $h_i$\,. 
We denote this number by $d_i$\,.

\item Second, we erase white vertices of $\G'$ and attach
the resulting free edges to new $k$ white vertices 
with labels   $1,2, \dots, k$ following this rule: 
the edged which previously terminated at the white 
vertex with label $i$ should now terminate at 
the white vertex with label $d_i$\,.

\item Finally, in the resulting graph $\G$, 
we keep the same total order on 
the set of edges as for $\G'$\,. 

\end{itemize}

\begin{example}
\label{ex:map-C}
To a graph $\G'$ depicted in figure \ref{fig:map-C}
and the monomial 
$$
(h_1h_2, 1, h_3, 1) \in \Hoch'(\C_3)
$$
we should assign the graph $\G$ shown on 
figure \ref{fig:map-C1}.  The total order on the set 
of edges of $\G$ is inherited from the total order 
on the set of edges of $\G'$\,. 
\begin{figure}[htp]
\begin{minipage}[t]{0.45\linewidth}
\centering 
\begin{tikzpicture}[scale=0.5, >=stealth']
\tikzstyle{w}=[circle, draw, minimum size=4, inner sep=1]
\tikzstyle{b}=[circle, draw, fill, minimum size=4, inner sep=1] 
\node [b] (b1) at (0,1) {};
\draw (0,1.6) node[anchor=center] {{\small $1$}};
\node [b] (b3) at (1,2) {};
\draw (1,2.6) node[anchor=center] {{\small $3$}};
\node [b] (b2) at (3,1.5) {};
\draw (3,2.1) node[anchor=center] {{\small $2$}};
\node [w] (w1) at (0.5,0) {};
\draw (0.5,-0.6) node[anchor=center] {{\small $1$}};
\node [w] (w2) at (2,0) {};
\draw (2,-0.6) node[anchor=center] {{\small $2$}};
\node [w] (w3) at (3,0) {};
\draw (3,-0.6) node[anchor=center] {{\small $3$}};
\draw [->] (b1) edge (b3);
\draw [->] (b1) edge (w1);
\draw [->] (b3) edge (w2);
\draw [->] (b2) edge (w3);
\end{tikzpicture}
\caption{A graph $\G' \in \dgra_{3,3}$} \label{fig:map-C}
\end{minipage}
\begin{minipage}[t]{0.45\linewidth}
\centering 
\begin{tikzpicture}[scale=0.5, >=stealth']
\tikzstyle{w}=[circle, draw, minimum size=4, inner sep=1]
\tikzstyle{b}=[circle, draw, fill, minimum size=4, inner sep=1] 
\node [b] (b1) at (0,1) {};
\draw (0,1.6) node[anchor=center] {{\small $1$}};
\node [b] (b3) at (1,2) {};
\draw (1,2.6) node[anchor=center] {{\small $3$}};
\node [b] (b2) at (3,1.5) {};
\draw (3,2.1) node[anchor=center] {{\small $2$}};
\node [w] (w1) at (0.5,0) {};
\draw (0.5,-0.6) node[anchor=center] {{\small $1$}};
\node [w] (w2) at (2,0) {};
\draw (2,-0.6) node[anchor=center] {{\small $2$}};
\node [w] (w3) at (3,0) {};
\draw (3,-0.6) node[anchor=center] {{\small $3$}};
\node [w] (w4) at (4,0) {};
\draw (4,-0.6) node[anchor=center] {{\small $4$}};
\draw [->] (b1) edge (b3);
\draw [->] (b1) edge (w1);
\draw [->] (b3) edge (w1);
\draw [->] (b2) edge (w3);
\end{tikzpicture}
\caption{The graph $\G \in \dgra_{3,4}$} \label{fig:map-C1}
\end{minipage}
\end{figure}
\end{example}

The described procedure gives us an obvious map 
\begin{equation}
\label{map-C-pr}
\Ups' : \bs^{2n-2}\KGra'(n,r) \otimes \Hoch'(\C_r) \to
 \KGra^{\Hoch}_r\,.
\end{equation}

The group $S_r$ acts in the obvious way on the 
source of the map \eqref{map-C-pr} by simultaneously 
rearranging the labels on white vertices and co-generators 
of $\C_r$\,. It is easy to see that $\Ups'$  \eqref{map-C-pr} descends to 
an isomorphism
\begin{equation}
\label{map-C}
\Ups : \left( \bs^{2n-2}\KGra'(n,r) \otimes \Hoch'(\C_r)\right)_{S_r} \to
 \KGra^{\Hoch}_r\,.
\end{equation}
from the space 
$$
\left( \bs^{2n-2}\KGra'(n,r) \otimes \Hoch'(\C_r)\right)_{S_r}
$$
of $S_r$-coinvariants to the complex in question $\KGra^{\Hoch}_r$\,.
It is not hard to see that the map \eqref{map-C} is compatible with 
the differential $\pa^{\Hoch}$ on  $\KGra^{\Hoch}_r$ and 
the differential on the source coming from $\pa^{\C}$ on  
$\Hoch'(\C_r)$\,.

Thus, using Claim \ref{cl:Hoch-pr}, it is not hard to 
prove the following statement about cohomology 
of $\KGra^{\Hoch}$ \eqref{Hoch-simpler}.
\begin{prop}
\label{prop:Hoch-simpler}
For every cocycle 
$$
\ga \in \bs^{2n-2+k} \KGra(n,k)^{\mo} 
$$
there exists a vector 
$$
\ga_1 \in  \bs^{2n-2+k-1} \KGra(n,k-1)^{\mo} 
$$
such that the difference 
$$
c = \ga - \pa^{\Hoch} (\ga_1) 
$$
satisfies Properties \ref{P:one} and \ref{P:anti-symm}. 
A cocycle $c$ in  \eqref{Hoch-simpler} 
satisfying Properties \ref{P:one} and 
\ref{P:anti-symm} is trivial if and only if $c=0$\,. $\Box$
\end{prop}

To deduce an analogous statement for the cochain
complex $\KGra^{\Hoch}_{\inv}$ \eqref{Hoch-app}
we need to use the averaging operator 
$$
\frac{1}{n!} \, \sum_{\si\in S_n} \si\,.
$$
More precisely,  Proposition  \ref{prop:Hoch-simpler} implies 
that 
\begin{cor}
\label{cor:Hoch}
For every cocycle 
$$
\ga \in \bs^{2n-2+k}  \Big( \KGra(n,k)^{\mo}  \Big)^{S_n}
$$
there exists a vector 
$$
\ga_1 \in  \bs^{2n-2+k-1} \Big( \KGra(n,k-1)^{\mo} \Big)^{S_n}
$$
such that the difference 
$$
c = \ga - \pa^{\Hoch} (\ga_1) 
$$
satisfies Properties \ref{P:one} and \ref{P:anti-symm}. 
A cocycle $c$ in the complex   \eqref{Hoch-app}
satisfying Properties \ref{P:one} and 
\ref{P:anti-symm} is trivial if and only if $c=0$\,. $\Box$
\end{cor}

It is clear that for every vector 
$$
\ga \in \bs^{2n-2} \Big( \KGra(n,0)^{\mo}  \Big)^{S_n}
$$
\begin{equation}
\label{k-equals-0}
\pa^{\Hoch} (\ga) = 0\,. 
\end{equation}

Due to this observation Corollary \ref{cor:Hoch}
implies the following statement.
\begin{cor}
\label{cor:k-one}
A vector 
\begin{equation}
\label{k-one}
\ga \in \bs^{2n-1} \Big(\KGra(n,1)^{\mo} \Big)^{S_n}
\end{equation}
is a cocycle in \eqref{Hoch-app} if and only if 
the white vertex in each  
graph in the linear combination $\ga$ 
has valency $1$\,. Furthermore, a cocycle $\ga$ in 
$ \bs^{2n-1} \Big( \KGra(n,1)^{\mo} \Big)^{S_n} $ 
is trivial if and only if  $\ga = 0$\,.     $\Box$
\end{cor}

\section{The complex of  ``hedgehogs''} 
\label{app:Hg}
This appendix is devoted to an auxiliary cochain complex
which is assembled from graphs $\G \in \dgra_{m,k}$ 
satisfying the additional property: {\it each white vertex of $\G$
has valency $1$}\,. Since such graphs look like hedgehogs 
we call this cochain complex the complex of ``hedgehogs''.  

This cochain complex and especially Corollary 
\ref{cor:Hg} (proved below) are used in Sections \ref{sec:transit} and \ref{sec:free}.
 
We start by introducing the following graded vector space
\begin{equation}
\label{Hg}
\Hg = \Big\{
\ga \in 
\bigoplus_{m, k} \bs^{2m- 2 + k} \big( \KGra(m,k)^{\mo} \big)^{S_m} \quad \big| \quad 
\ga ~ \textrm{obeys Properties \ref{P:one}, \ref{P:anti-symm} }
 \Big\}
\end{equation}
and the families of cycles
$\tau_{m,i} \in S_m$, and  $\si_{k,i}, \vs_{k,i} \in S_k$
\begin{equation}
\label{tau-m-i}
\tau_{m,i} : = (i, i+1, \dots, m-1, m),
\end{equation}
\begin{equation}
\label{si-k-i}
\si_{k,i} : = (i, i-1, \dots 2, 1),
\end{equation}
and
\begin{equation}
\label{vs-k-i}
\vs_{k,i} : = (1,2, \dots, i-1, i).
\end{equation}

Next, we denote by $\md$ the following degree $1$ operation
on $\Hg$
\begin{equation}
\label{md}
\md(\ga) =  k \, \sum_{i=1}^{m+1} \, (\tau_{m+1,i}, \id) 
\big(\, \ga  \circ_{1, \mo} \G^{\br}_{0}
\, \big)\,, \qquad  \ga \in \bs^{2m- 2 + k} \big( \KGra(m,k)^{\mo} \big)^{S_m}\,.
\end{equation}
Notice that, since the graph $\G^{\br}_0$ consists of a single 
black vertex and has no edges, the insertion $\circ_{1, \mo}$
of $\G^{\br}_{0}$ replaces the white vertex with label $1$
by a black vertex with label $m+1$ and shifts 
the labels on the remaining white vertices 
down by $1$\,. 

Using the fact that each linear combination 
$\ga\in \Hg$ is anti-symmetric with respect 
to permutations of labels on white vertices, it is not hard to deduce 
that
\begin{equation}
\label{md-sq}
\md^2  = 0\,.
\end{equation}
Thus $(\Hg, \md)$ is a cochain complex. 
We call this cochain complex the complex of ``hedgehogs''.

For our purposes, we need a degree $-1$ operation 
\begin{equation}
\label{md-star}
\md^*  : \Hg \to \Hg
\end{equation}
which we will now define.
Let $\ga $ be a vector in  $\bs^{2m- 2 + k} \big( \KGra(m,k)^{\mo} \big)^{S_m} $
satisfying Properties \ref{P:one}, \ref{P:anti-symm}.
To compute $\md^*(\ga)$ we follow these steps:

\begin{itemize}

\item First, we omit in $\ga$ all graphs for which the black 
vertex with label $1$ is not a pike. We denote 
the resulting linear combination in 
$\bs^{2m- 2 + k}  \KGra(m,k)^{\mo}$ by $\ga'$\,.

\item Second, we replace the black vertex with label $1$
in each graph of $\ga'$ by a white vertex and shift all labels 
on black vertices down by $1$\,. 
We assign label $1$ 
to this additional white vertex and shift the labels of the remaining 
white vertices up by $1$\,. We denote the resulting 
linear combination in 
$\bs^{2(m-1)- 2 + k+1}  \KGra(m-1,k+1)^{\mo}$ by $\ga''$\,.

\item Finally, we set 
\begin{equation}
\label{md-star-ga}
\md^*(\ga) = \sum_{i = 1}^{k+1} \frac{(-1)^{i-1}}{k+1} 
(\id , \si_{k+1, i}) (\ga'')\,.
\end{equation}

\end{itemize}
Note that the linear combination $\ga'$ is invariant with respect to the 
action of the group $S_{\{2, 3, \dots, m\}}$\,. Hence, the linear combination 
$\md^*(\ga)$ is $S_{m-1}$-invariant. Furthermore, $\md^*(\ga)$ obviously 
satisfies Properties \ref{P:one} and \ref{P:anti-symm}.

\begin{remark}
\label{rem:md-star}
Notice that 
\begin{equation}
\label{md-star-nopikes}
\md^* (\ga) = 0\,.
\end{equation}
if each graph in the linear combination $\ga$ does not have pikes.
\end{remark}

\begin{example}
\label{ex:md-md-star}
Let us denote by $\G_k$ the graph depicted in figure \ref{fig:G-k} 
and let 
\begin{equation}
\label{ga-exam}
\ga = \G_k + (\si_{12}, \id)\big(\G_k \big)\,,
\end{equation}
where $\si_{12}$ is the transposition in $S_2$\,.
\begin{figure}[htp]
\centering 
\begin{tikzpicture}[scale=0.5, >=stealth']
\tikzstyle{w}=[circle, draw, minimum size=4, inner sep=1]
\tikzstyle{b}=[circle, draw, fill, minimum size=4, inner sep=1]
\node [b] (b1) at (0,3) {};
\draw (0,3.6) node[anchor=center] {{\small $1$}};
\node [b] (b2) at (2.5,3) {};
\draw (2.5,3.6) node[anchor=center] {{\small $2$}};
\node [w] (v1) at (0,0) {};
\draw (0,-0.6) node[anchor=center] {{\small $1$}};
\node [w] (v2) at (1,0) {};
\draw (1,-0.6) node[anchor=center] {{\small $2$}};
\draw (2.5,0.7) node[anchor=center] {{\small $\dots$}};
\node [w] (vk) at (4,0) {};
\draw (4,-0.6) node[anchor=center] {{\small $k$}};
\draw [->] (b2) edge (b1);
\draw [->] (b2) edge (1,1.2);
\draw  (1,1.2) edge (v1);
\draw [->] (b2) edge (1.5,1);
\draw  (1.5,1) edge (v2);
\draw [->] (b2) edge (3.5,1);
\draw  (3.5,1) edge (vk);
\end{tikzpicture}
~\\[0.3cm]
\caption{Edges are ordered in this way 
$(2_{\mc}, 1_{\mc}) <(2_{\mc}, 1_{\mo}) < (2_{\mc}, 2_{\mo}) < \dots < (2_{\mc}, k_{\mo})$
} \label{fig:G-k}
\end{figure} 

It is obvious that $\ga$ is a vector in 
$\bs^{k+2}  \big( \KGra(2,k)^{\mo}  \big)^{S_2}$ satisfying 
Properties \ref{P:one}, \ref{P:anti-symm}.

Following the steps outlined above, we get 
$$
\ga' = \G_k \qquad \textrm{and} \qquad \ga'' = \G^{\br}_{k+1}\,,  
$$
where $\G^{\br}_{k}$ is the family of ``brooms'' shown on 
figure \ref{fig:brooms}\,. Since $ \G^{\br}_{k+1} $ is already antisymmetric 
with respect to permutations of labels on white vertices, 
$$
\md^* (\ga) = 
\sum_{i=1}^{k+1} \frac{(-1)^{i-1}}{k+1}
(\id, \si_{k+1, i}) (\G^{\br}_{k+1}) =  \G^{\br}_{k+1}\,.  
$$  
\end{example}

We need the following lemma.
\begin{lem}
\label{lem:Koszul}
For every vector
$$
\ga \in \bs^{2m- 2 + k} \big( \KGra(m,k)^{\mo} \big)^{S_m} 
$$
satisfying Properties \ref{P:one}, \ref{P:anti-symm} 
we have 
\begin{equation}
\label{Hodge-decomp}
\md \md^* (\ga) + \md^* \md(\ga) = k \ga 
+ \sum_{r \ge 1} r \ga_r\,,   
\end{equation}
where $\ga_r$ is the linear combination in $\Hg$
which is obtained from $\ga$ by retaining 
the graphs with exactly $r$ pikes. 
\end{lem}

\begin{proof}
Let us observe that the space
$$
\bs^{2m- 2 + k} \big( \KGra(m,k)^{\mo} \big)^{S_m} 
$$
is spanned by vectors of the form 
\begin{equation}
\label{sp-set}
\sum_{\tau\in S_m, \, \si \in S_k} (-1)^{|\si|}(\tau, \si) \big( \G \big)   
\end{equation}
where $\G$ is a graph in $\dgra_{m,k}$ with all
white vertices having valency $1$\,.
 
Thus we may assume, without loss of generality, that 
\begin{equation}
\label{this-ga}
\ga =  \sum_{\tau\in S_m,\, \si \in S_k} (-1)^{|\si|}(\tau, \si) \big(\G \big)  
\end{equation}
for a graph $\G \in \dgra_{m,k}$  with all
white vertices having valency $1$\,.

Using the cycles $\vs_{k,i}$ \eqref{vs-k-i} we 
rewrite \eqref{this-ga} as follows: 
\begin{equation}
\label{this-ga-new}
\ga = 
\sum_{\tau\in S_m} 
\sum_{\si' \in S_{\{2,3, \dots, k\}}} (-1)^{|\si'|}(\tau, \si') 
\left(
\sum_{i=1}^k (-1)^{i-1} (\id , \vs_{k,i}) (\G)
\right)\,,  
\end{equation}
where $S_{\{2,3, \dots, k\}}$ denotes the permutation group 
of the set $\{2,3, \dots, k\}$\,.

Next, using \eqref{this-ga-new} together with 
the obvious identity
$$
\big( (\id , \vs_{k,i}) (\G) \big) \circ_{1, \mo} \G^{\br}_0 = 
\G  \circ_{i, \mo} \G^{\br}_0
$$
we deduce that 
$$
\md(\ga) = k \,   
\sum_{j=1}^{m+1} (\tau_{m+1, j}, \id) \, 
\left(  
\sum_{\tau\in S_m,\, \si' \in S_{k-1}} (-1)^{|\si'|}(\tau, \si')
\left(\sum_{i=1}^{k} (-1)^{i-1} \G \circ_{i, \mo} \G^{\br}_{0} \right)
\right) =
$$
\begin{equation}
\label{md-ga}
k \,
\sum_{\tau\in S_{m+1},\, \si' \in S_{k-1}} (-1)^{|\si'|}(\tau, \si')
\left(\sum_{i=1}^{k} (-1)^{i-1} \G \circ_{i, \mo} \G^{\br}_{0} \right)\,.
\end{equation}

Let us, first, consider the case when the graph $\G$ 
does not have pikes. In this case, due to Remark \ref{rem:md-star}, 
we have 
$$
\md^* (\ga) = 0\,. 
$$
Furthermore, using \eqref{md-ga}, we get
\begin{equation}
\label{md-star-md-ga}
\md^* \md(\ga) =
\end{equation}
$$
\sum_{j =1}^{k} (-1)^{j-1} (\id, \si_{k,j}) \quad \left(
\sum_{\tau\in S_{m}} \sum_{\si' \in S_{\{2,3, \dots, k\}} } (-1)^{|\si'|}(\tau, \si')
\left(\sum_{i=1}^{k} (-1)^{i-1} (\id, \vs_{k,i}) (\G)  \right) \quad \right)\,,
$$
where $S_{\{2,3, \dots, k\}}$ denotes the permutation group 
of the set $\{2,3, \dots, k\}$\,, and $\vs_{k,i}$ is the family of 
cycles defined in \eqref{vs-k-i}.

According to \eqref{this-ga-new},
\begin{equation}
\label{simpler} 
\sum_{\tau\in S_{m}} \sum_{\si' \in S_{\{2,3, \dots, k\}} } (-1)^{|\si'|}(\tau, \si')
\left(\sum_{i=1}^{k} (-1)^{i-1} (\id, \vs_{k,i}) (\G)  \right) = \ga.
\end{equation}
Moreover, since $\ga$ is antisymmetric with respect to 
permutations of labels on white vertices,
$$ 
(\id, \si_{k,j}) (\ga) = (-1)^{j-1} \ga.
$$
Hence
\begin{equation}
\label{k-times}
 \sum_{j =1}^{k} (-1)^{j-1} (\id, \si_{k,j}) (\ga) = k \ga\,.
\end{equation}

Therefore, combining \eqref{md-star-md-ga} with \eqref{simpler}
and \eqref{k-times}, we get
\begin{equation}
\label{md-star-md-ga1}
\md^* \md(\ga) = k \ga\,.
\end{equation}
Thus, if each graph in a linear combination $\ga$ 
does not have pikes, then equation
\eqref{Hodge-decomp} holds. 

Let us now turn to the case when $\G$ has exactly $r \ge 1$ pikes. 

Without loss of generality, we may assume that the pikes 
of $\G$ are labeled by $1, 2, \dots, r$\,.

Let us recall that the vector $\ga'$ is obtained from $\ga$
by discarding all graphs for which the black 
vertex with label $1$ is not a pike.
In our case, the vector $\ga'$ 
can be written as follows:
\begin{equation}
\label{ga-pr}
\ga' = \sum_{\si \in S_k} \sum_{\tau' \in S_{\{2,3, \dots, m\}}} 
(-1)^{|\si|}(\tau', \si)
\left(\sum_{p=1}^r  (\vs_{m,p}, \id) (\G) \right) \,,
\end{equation}
where $S_{\{2,3, \dots, m\}}$ denotes the permutation group 
of the set $\{2,3, \dots, m\}$\,, and $\vs_{m,p}$ is the family of 
cycles in $S_m$ defined in \eqref{vs-k-i}.

Using \eqref{ga-pr} we get 
\begin{equation}
\label{md-star-ga1}
\md^*(\ga) = \sum_{i=1}^{k+1} \frac{(-1)^{i-1}}{k+1}  (\id, \si_{k+1,i}) (\ga'')
\end{equation}
with
\begin{equation}
\label{ga-prpr}
\ga'' = \sum_{\si \in S_{\{2,3, \dots, k+1\}}  } \sum_{\tau' \in S_{m-1}} 
(-1)^{|\si|}(\tau', \si)
\left(\sum_{p=1}^r  R_{\circ} \big( (\vs_{m,p}, \id) (\G) \big) \right)\,, 
\end{equation}
where $S_{\{2,3, \dots, k+1\}}$ is the group of permutations of 
the set $\{2,3, \dots, k+1\}$ and
$R_{\circ}$ is the operation which replaces the pike with label $1$ 
by a white vertex with label $1$, shifts labels on the remaining 
white vertices up by $1$ and shifts labels on black vertices down 
by $1$\,.

For the vector $\md \md^* (\ga) $ we get 
\begin{equation}
\label{md-md-star-ga}
\md \md^* (\ga) = \sum_{j=1}^{m} (\tau_{m, j} \,,\, \id)  \left(\,
\sum_{\si \in S_k} \sum_{\tau' \in S_{m-1}  } 
(-1)^{|\si|}(\tau', \si)
\left(\sum_{p=1}^r   (\vs_{m,p}, \id) (\G)   \right)\, \right) + 
\end{equation}
$$
+  \sum_{j=1}^{m} (\tau_{m,j}, \id)
\left(
\sum_{i=2}^{k+1}  (-1)^{i-1}
\big( (\id, \si_{k+1,i}) (\ga'') \big) \circ_{1, \mo} \G^{\br}_0 
\right)\,,
$$
where the first sum comes from the first term 
in the sum \eqref{md-star-ga1} and the second sum 
comes from the remaining terms in \eqref{md-star-ga1}\,.
  
The first sum in \eqref{md-md-star-ga} can be simplified as follows.
$$
\sum_{j=1}^{m} (\tau_{m, j} \,,\, \id)  \left(\,
\sum_{\si \in S_k} \sum_{\tau' \in S_{m-1} } 
(-1)^{|\si|}(\tau', \si)
\left(\sum_{p=1}^r   (\vs_{m,p}, \id) (\G)   \right)\, \right) = 
$$
$$
\sum_{\si \in S_k} \sum_{\tau \in S_m } 
(-1)^{|\si|}(\tau, \si)
\left(\sum_{p=1}^r   (\vs_{m,p}, \id) (\G)   \right)= 
$$
$$
r \, \sum_{\si \in S_k} \sum_{\tau \in S_m } 
(-1)^{|\si|}(\tau, \si) (\G) =  r\,  \ga\,.
$$
In other words, 
\begin{equation}
\label{first-sum}
\sum_{j=1}^{m} (\tau_{m, j} \,,\, \id)  \left(\,
\sum_{\si \in S_k} \sum_{\tau' \in S_{m-1} } 
(-1)^{|\si|}(\tau', \si)
\left(\sum_{p=1}^r   (\vs_{m,p}, \id) (\G)   \right)\, \right) = r\, \ga
\end{equation}

To simplify the  second sum in  \eqref{md-md-star-ga}
we notice that the subsets of $S_{k+1}$
$$
\{\si_{k+1,i}  \circ \si ~|~ \si \in  S_{\{2,3, \dots, k+1\}}, ~ 
2 \le i \le  k+1 \}
$$ 
and 
$$
\{\si \circ \vs_{k+1,i} ~|~ \si \in  S_{\{2,3, \dots, k+1\}}, ~ 
2 \le i \le  k+1 \}
$$ 
coincide.

Hence, 
$$
\sum_{i=2}^{k+1} \frac{(-1)^{i-1}}{k+1}  (\id, \si_{k+1,i}) (\ga'') =
$$
\begin{equation}
\label{for-second-sum}
\frac{1}{k+1} 
\sum_{\si \in S_{\{2,3, \dots, k+1\}}  } \sum_{\tau' \in S_{m-1}} 
(-1)^{|\si|}(\tau', \si)
\left(\sum_{p=1}^r 
\sum_{i=2}^{k+1} (-1)^{i-1} 
(\id, \vs_{k+1, i})
R_{\circ} \big( (\vs_{m,p}, \id) (\G) \big) \right)\,. 
\end{equation}

Next, we introduce operations $\big \{ \Cg^i_p \big\}_{1\le p \le r, 1 \le i \le k }$  
whose input is our graph $\G$ and whose outputs are graphs in $\dgra_{m,k}$
with the same properties, i.e. each white vertex of $\Cg^i_p(\G)$ 
has valency $1$ and $\Cg^i_p(\G)$ has exactly $r$ pikes. 
This operation is illustrated in figure \ref{fig:Cg}.
\begin{figure}[htp]
\centering 
\begin{tikzpicture}[scale=0.5, >=stealth']
\tikzstyle{ahz}=[circle, draw, fill=gray, minimum size=26, inner sep=1]
\tikzstyle{w}=[circle, draw, minimum size=4, inner sep=1]
\tikzstyle{b}=[circle, draw, fill, minimum size=4, inner sep=1]
\node[b] (b1) at (-0.5,4) {};
\draw (-0.5,4.5) node[anchor=center] {{\small $1$}};
\draw (0.5,4) node[anchor=center] {{\small $\dots$}};
\node[b] (b1p) at (1.5,4) {};
\draw (1.5,4.5) node[anchor=center] {{\small $p-1$}};
\node[b] (bp) at (3,4) {};
\draw (3,4.5) node[anchor=center] {{\small $p$}};
\node[b] (bp1) at (4.5,4) {};
\draw (4.5,4.5) node[anchor=center] {{\small $p+1$}};
\draw (5.5,4) node[anchor=center] {{\small $\dots$}};
\node[b] (br) at (6.5,4) {};
\draw (6.5,4.5) node[anchor=center] {{\small $r$}};
\node [ahz] (ahz) at (2.5,2) {};
\node[w] (w1) at (0,0) {};
\draw (0,-0.5) node[anchor=center] {{\small $1$}};
\node[w] (w2) at (1,0) {};
\draw (1,-0.5) node[anchor=center] {{\small $2$}};
\node[w] (w3) at (2,0) {};
\draw (2,-0.5) node[anchor=center] {{\small $3$}};
\draw (3.5,0) node[anchor=center] {{\small $\dots$}};
\node[w] (wk) at (5,0) {};
\draw (5,-0.5) node[anchor=center] {{\small $k$}};
\draw [->] (ahz) edge (w1);
\draw [->] (ahz) edge (w2);
\draw [->] (ahz) edge (w3);
\draw [->] (ahz) edge (wk);
\draw [->] (ahz) edge (b1);
\draw [->] (ahz) edge (b1p);
\draw [->] (ahz) edge (bp);
\draw [->] (ahz) edge (bp1);
\draw [->] (ahz) edge (br);
\draw (8,2) node[anchor=center] {{$\stackrel{\Cg^i_p ~~}{\longrightarrow}$}};
\node[b] (bb1) at (9.5,4) {};
\draw (9.5,4.5) node[anchor=center] {{\small $1$}};
\draw (10.5,4) node[anchor=center] {{\small $\dots$}};
\node[b] (bb1p) at (11.5,4) {};
\draw (11.5,4.5) node[anchor=center] {{\small $p-1$}};
\node[w] (ww1) at (13,4) {};
\draw (13,4.5) node[anchor=center] {{\small $1$}};
\node[b] (bbp) at (14.5,4) {};
\draw (14.5,4.5) node[anchor=center] {{\small $p$}};
\draw (15.5,4) node[anchor=center] {{\small $\dots$}};
\node[b] (bb1r) at (16.5,4) {};
\draw (16.5,4.5) node[anchor=center] {{\small $r-1$}};
\node [ahz] (ahz1) at (12.5,2) {};
\node[w] (ww2) at (10,0) {};
\draw (10,-0.5) node[anchor=center] {{\small $2$}};
\draw (11,0) node[anchor=center] {{\small $\dots$}};
\node[w] (wwi) at (12,0) {};
\draw (12,-0.5) node[anchor=center] {{\small $i$}};
\node[b] (bbm) at (13,0) {};
\draw (13,-0.5) node[anchor=center] {{\small $m$}};
\node[w] (wwi1) at (14.5,0) {};
\draw (14.5,-0.5) node[anchor=center] {{\small $i+1$}};
\draw (15.5,0) node[anchor=center] {{\small $\dots$}};
\node[w] (wwk) at (16.5,0) {};
\draw (16.5,-0.5) node[anchor=center] {{\small $k$}};
\draw [->] (ahz1) edge (bbm);
\draw [->] (ahz1) edge (ww2);
\draw [->] (ahz1) edge (wwi);
\draw [->] (ahz1) edge (wwi1);
\draw [->] (ahz1) edge (wwk);
\draw [->] (ahz1) edge (bb1);
\draw [->] (ahz1) edge (bb1p);
\draw [->] (ahz1) edge (ww1);
\draw [->] (ahz1) edge (bbp);
\draw [->] (ahz1) edge (bb1r);
\end{tikzpicture}
~\\[0.3cm]
\caption{The operation $\G \mapsto \Cg^i_p (\G) $\,. 
Gray regions denote subgraphs formed by black 
vertices which are not pikes} 
\label{fig:Cg}
\end{figure} 
More precisely, $\Cg^i_{p}(\G)$ is obtained from $\G$ via
these steps: 
\begin{itemize}

\item first, we replace the black vertex with label $p$ by
a white vertex and replace the white vertex with label $i$
by a black vertex;

\item second, we shift the labels on the black vertices which are 
$ > p$ down by $1$; 

\item third, we shift the labels on the white vertices which are 
$ < i$ up by $1$; 

\item finally, we assign label $1$ to the new white vertex 
and we assign label $m$ to the new black vertex. 

\end{itemize}

Using equation \eqref{for-second-sum}
and the graphs $\Cg^i_p(\G)$ we present
the second sum in  \eqref{md-md-star-ga}
in the following way. 
\begin{equation}
\label{second-sum}
 \sum_{j=1}^{m} (\tau_{m,j}, \id)
\left(
\sum_{i=2}^{k+1}  (-1)^{i-1}
\big( (\id, \si_{k+1,i}) (\ga'') \big) \circ_{1, \mo} \G^{\br}_0 
\right) = 
\end{equation}
$$
\sum_{j=1}^{m} (\tau_{m,j}, \id)
\sum_{\tau' \in S_{m-1}} 
\sum_{\si \in S_k }
(-1)^{|\si|} (\tau',  \si)
\left(
\sum_{p=1}^r
\sum_{i=2}^{k+1}  (-1)^{i-1}
 \Big(\Cg^{i-1}_{p} \big( \G \big) \Big)
\right) =
$$
$$
- \sum_{j=1}^{m} 
\sum_{\tau' \in S_{m-1}} 
\sum_{\si \in S_k }
(-1)^{|\si|} (\tau_{m,j} \tau',  \si)
\left(
\sum_{p=1}^r
\sum_{i=1}^{k}  (-1)^{i-1}
 \Big(\Cg^i_{p} \big( \G \big) \Big)
\right) =  
$$
$$
- \sum_{\tau \in S_{m}} 
\sum_{\si \in S_k }
(-1)^{|\si|} (\tau,  \si) \Big(  \sum_{p=1}^r
\sum_{i=1}^{k}  (-1)^{i-1}
\Cg^i_{p} \big( \G \big) \Big)\,.
$$

Combining this observation with equation \eqref{first-sum}, 
we conclude that 
\begin{equation}
\label{md-md-star-ga-new}
\md \md^* (\ga) =  r \,\ga  
\end{equation}
$$
- \sum_{\tau \in S_{m}} 
\sum_{\si \in S_k }
(-1)^{|\si|} (\tau,  \si) \Big( \sum_{p=1}^r
\sum_{i=1}^{k}  (-1)^{i-1}
\Cg^i_{p} \big( \G \big) \Big)
$$

To unfold $\md^* \md (\ga)$, we denote by $\om$ the vector $\md(\ga)$ 
 \eqref{md-ga}.  By discarding in $\om$  all graphs for 
which black vertex with label $1$ is not a pike we get 
the expression 
\begin{equation}
\label{om-pr}
\om' =  k \, \sum_{\tau\in S_{\{2,3, \dots, m+1\}}} 
\sum_{\si' \in S_{k-1}} (-1)^{|\si'|} (\tau, \si')\,
\left(
\sum_{i=1}^k (-1)^{i-1} (\tau_{m+1,1}, \id)(\G \circ_{i, \mo} \G^{\br}_0)
\right) 
\end{equation}
$$
+  k \, \sum_{\tau\in S_{\{2,3, \dots, m+1\}}} 
\sum_{\si' \in S_{k-1}} (-1)^{|\si'|} (\tau, \si')\,
\left(\sum_{p=1}^r
\sum_{i=1}^k (-1)^{i-1} (\vs_{m+1, p}, \id)(\G \circ_{i, \mo} \G^{\br}_0)
\right)\,.  
$$

Next, replacing the black vertices with label $1$ in each graph in 
$\om'$ by a white vertex with label $1$ and shifting the 
labels of the remaining vertices correspondingly, we get 
\begin{equation}
\label{om-prpr}
\om'' =  k \, \sum_{\tau\in S_m} 
\sum_{\si' \in S_{\{2,3, \dots, k\}}} (-1)^{|\si'|} (\tau, \si')\,
\left(
\sum_{i=1}^k (-1)^{i-1} (\id, \vs_{k, i})(\G)
\right) 
\end{equation}
$$
+   k \, \sum_{\tau\in S_m} 
\sum_{\si' \in S_{\{2,3, \dots, k\}}} (-1)^{|\si'|} (\tau, \si')\,
\left( \sum_{p=1}^r
\sum_{i=1}^k (-1)^{i-1} \Cg^i_p \big( \G \big)
\right)\,. 
$$

Thus 
\begin{equation}
\label{md-star-md-ga11}
\md^* \md(\ga) = 
\sum_{j=1}^k \frac{(-1)^{j-1}}{k} 
(\id, \si_{k,j}) (\om'') =
\end{equation}
$$ 
\sum_{\tau\in S_m} 
\sum_{\si \in S_k } (-1)^{|\si|} (\tau, \si)\,
\left(
\sum_{i=1}^k (-1)^{i-1} (\id, \vs_{k, i})(\G)
\right) 
$$
$$
+   \sum_{\tau\in S_m} 
\sum_{\si \in S_k} (-1)^{|\si|} (\tau, \si)\,
\left( \sum_{p=1}^r
\sum_{i=1}^k (-1)^{i-1} \Cg^i_p \big(  \G \big)
\right) = 
$$
$$
k \, \ga 
+   \sum_{\tau\in S_m} 
\sum_{\si \in S_k} (-1)^{|\si|} (\tau, \si)\,
\left( \sum_{p=1}^r
\sum_{i=1}^k (-1)^{i-1} \Cg^i_p \big(  \G \big)
\right)
$$

Combining \eqref{md-md-star-ga-new} with \eqref{md-star-md-ga11}
we immediately deduce equation \eqref{Hodge-decomp}.   
 
Lemma  \ref{lem:Koszul} is proved. 
\end{proof}
\begin{remark}
\label{rem:Koszul}
The cochain complex $\Hg$ \eqref{Hg} with the 
differential $\md$ \eqref{md} is very similar to Koszul 
complex for the exterior algebra. However, the author 
could not find an elegant way to reduce $\Hg$ to this 
well known complex. 
\end{remark}

We have the following corollary. 
\begin{cor}
\label{cor:Hg}
Let $\ga$ be a vector in 
$$
 \bs^{2m-2 +k}\, \big( \KGra(m,k)^{\mo} \big)^{S_m} 
$$
satisfying Properties  \ref{P:one}, \ref{P:anti-symm}.
If $k \ge 1$ and $\ga$ is $\md$-closed then 
there exists 
$$
\wt{\ga} \in  \bs^{2(m-1) -2  + k+1}\, \big( \KGra(m-1,k+1)^{\mo} \big)^{S_{m-1}} 
$$
which satisfies Properties  \ref{P:one}, \ref{P:anti-symm} and
such that 
\begin{equation}
\label{ga-md-exact}
\ga = \md(\,\wt{\ga}\,)\,.
\end{equation}
\end{cor}
\begin{proof}
Since $\ga$ is $\md$-closed, equation \eqref{Hodge-decomp}
implies that
\begin{equation}
\label{Hodge-for-ga}
\md \md^*(\ga) = k \ga + \sum_{r \ge 1} r \ga_r \,, 
\end{equation}
where $\ga_r$ is the linear combination in $\Hg$
which is obtained from $\ga$ by retaining 
the graphs with exactly $r$ pikes. 

Since each graph in the image of  $\md$ has 
at least one pike, equation   \eqref{Hodge-for-ga} 
implies that each graph in the linear combination 
$\ga$ has at least one pike. Hence,  
\begin{equation}
\label{ga-really}
\ga =  \sum_{r \ge 1}  \ga_r 
\end{equation}
and 
 \eqref{Hodge-for-ga} 
can be rewritten as 
\begin{equation}
\label{Hodge-for-ga-new}
\md \md^*(\ga) = \sum_{r \ge 1} (k+r)\ga_r \,. 
\end{equation}

Thus, setting 
\begin{equation}
\label{wt-ga}
\wt{\ga} =  \sum_{r \ge 1} \frac{1}{k+r}  \md^* (\ga_r) 
\end{equation}
we get the desired identity
$$
\ga = \md (\,\wt{\ga}\, )\,.
$$
\end{proof}

\section{Maurer-Cartan (MC) elements of filtered Lie algebras}
\label{app:MC}

Let $\cL$ be a Lie algebra in the category $\Ch_{\bbK}$ of 
unbounded cochain complexes of $\bbK$-vector spaces. 
Let us assume that $\cL$ is equipped with a descending  
filtration
\begin{equation}
\label{cL-filtr}
 \dots \supset \cF_{-1} \cL  \supset \cF_0 \cL  \supset \cF_1 \cL \supset   \cF_2 \cL \supset \cF_3 \cL \supset \dots
\end{equation}
which is compatible with the Lie bracket, and such that $\cL$ is complete 
and cocomplete with respect to this filtration, i.e. 
\begin{equation}
\label{cL-complete}
\cL = \lim_{k} \cL\, \big/ \, \cF_k \cL
\end{equation}
and 
\begin{equation}
\label{cL-cocomplete}
\cL = \bigcup_{k} \cF_k \cL\,.
\end{equation}
We call such Lie algebras {\it filtered}.

Condition \eqref{cL-complete} guarantees that the subalgebra 
$\cF_1 \cL^0$ of degree zero elements in $\cF_1 \cL$ is
a pro-nilpotent Lie algebra (in the category of $\bbK$-vector spaces). 
Hence,  $\cF_1 \cL^0$ can be exponentiated to a pro-unipotent 
group which we denote by 
\begin{equation}
\label{the-group}
\exp(\cF_1 \cL^0)\,.
\end{equation}

We recall that a {\it Maurer-Cartan (MC) element} of 
$\cL$ is a degree $1$ vector $\al \in \cL$ satisfying the equation
\begin{equation}
\label{eq:MC}
\pa \al + \frac{1}{2} [\al,\al] =0\,,
\end{equation}
where $\pa$ denotes the differential on $\cL$\,.

For a vector $\xi \in \cF_1\cL^0$ and a MC element $\al$ 
we consider the new degree $1$ vector $\wt{\al} \in \cL$
which is given by the formula 
\begin{equation}
\label{xi-acts}
\wt{\al} = \exp(\ad_{\xi})\, \al -
\frac{\exp(\ad_{\xi}) - 1}{\ad_{\xi}} \pa \xi\,,
\end{equation}
where the expressions 
$$
 \exp(\ad_{\xi}) \qquad \textrm{and} \qquad
\frac{\exp(\ad_{\xi}) - 1}{\ad_{\xi}}
$$
are defined in the obvious way using the Taylor 
expansions of the functions 
$$
e^x \qquad \textrm{and} \qquad   \frac{e^x -1}{x}
$$
around the point $x = 0$\,, respectively. 

Condition \eqref{cL-complete} guarantees that the
right hand side of equation \eqref{xi-acts} makes sense.
 
It is known (see, e.g. \cite[Appendix B]{BDW} or \cite{GM}) that, 
for every MC element $\al$ and for every degree zero 
vector $\xi \in \cF_1 \cL$, the vector $\wt{\al}$ in \eqref{xi-acts}
is also a MC element. Furthermore, formula \eqref{xi-acts} 
defines an action of the group \eqref{the-group} on 
the set of MC elements of $\cL$\,. 

The transformation groupoid corresponding to this 
action is called the {\it Deligne groupoid}  of the Lie 
algebra $\cL$\,. This groupoid and its higher versions 
were studied extensively in \cite{Berglund}, \cite{BGNT},
 \cite{B-Felix-M-Tanre},
\cite{GMtheorem}, \cite{Ezra} 
and \cite{Ezra-infty}, \cite{Henriques}, \cite{Yalin}.

\begin{example}
\label{ex:Conv} 
Let $\cC$ (resp. $\cO$) be a $\Xi$-colored pseudo-cooperad
(resp. $\Xi$-colored pseudo-operad) in $\Ch_{\bbK}$\,. 
The convolution Lie algebra $\Conv(\cC, \cO)$ described in 
Section \ref{sec:Conv} gives us an example of a filtered Lie 
algebra. Thus it makes sense to talk about the 
Deligne groupoid of $\Conv(\cC, \cO)$\,.  
\end{example}

\subsection{Differential equations on the Lie algebra $\cL \hotimes \bbK[t]$}
\label{app:diffura}

Given a filtered dg Lie algebra $\cL$,
we introduce the new dg Lie algebra\footnote{$t$ is an auxiliary degree zero variable.} 
\begin{equation}
\label{cL-adjoin-t}
\cL \hotimes \bbK[t] 
\end{equation}
where $\cL$ is considered with the topology coming from 
the filtration and $\bbK[t]$ is considered with the discrete topology. 

It is clear that $\cL \hotimes \bbK[t]$ consists of vectors 
\begin{equation}
\label{v-t}
v = \sum_{k=0}^{\infty} v_k\, t^k  ~~\in~~ \cL[[t]]
\end{equation}
satisfying the condition 
\begin{cond}
\label{cond:growth}
For every integer $m$, the image of $v$ in 
\begin{equation}
\label{cL-cFmcL}
\big( \cL \big/  \cF_m \cL \big) [[t]]
\end{equation}
is a polynomial in $t$. In other words, for every integer $m$ 
there exists $k_m$ such that $v_k \in \cF_m  \cL$ for all $k \ge k_m$\,.
\end{cond}

Let us also observe that $\cL \hotimes \bbK[t]$ comes with the obvious descending filtration:
\begin{equation}
\label{filtr-cL-t}
\cF_m \big( \cL \hotimes \bbK[t] \big)  : = (\cF_m \cL) \hotimes \bbK[t],
\end{equation}
the subspace $\cL \hotimes \bbK[t] \subset \cL[[t]]$ is closed with respect to the 
(formal) derivative $d/dt$ and 
$$
d/dt  \big( (\cF_m \cL) \hotimes \bbK[t] \big) ~~ \subset ~~ (\cF_m \cL) \hotimes \bbK[t].
$$

Condition \ref{cond:growth} and completeness of $\cL$ with respect to 
the filtration $\cF_{\bul}$ imply that the assignment 
\begin{equation}
\label{set-t-1}
v \mapsto v \Big|_{t=1}  
\end{equation}
defines a Lie algebra homomorphism from $ \cL \hotimes \bbK[t]$ to $\cL$. 
Furthermore, this homomorphism is compatible with the filtrations on 
$\cL \hotimes \bbK[t]$ and $\cL$. 

Next, we claim that 
\begin{claim} 
\label{cl:cL-t-complete}
The dg Lie algebra $\cL \hotimes \bbK[t]$ is complete and cocomplete with respect to 
the filtration \eqref{filtr-cL-t}. 
\end{claim}
\begin{proof}
The cocompleteness follows readily from Property \eqref{cL-cocomplete} 
and Condition \ref{cond:growth}.

To prove the completeness, we need to show that for every 
infinite sequence of vectors 
\begin{equation}
\label{sequence-v-r}
v^{(r)} = \sum_{k\ge 0} v^{(r)}_k t^k \in (\cF_{m_r} \cL) \hotimes \bbK[t], \qquad r \ge 1
\end{equation}
satisfying the condition  
\begin{equation}
\label{m-r-growth}
m_1 \le m_2 \le m_3 \le \dots, \qquad 
\lim_{r \to \infty} m_r = \infty
\end{equation}
the sum 
\begin{equation}
\label{sum-v-r}
\sum_{r \ge 1} v^{(r)} 
\end{equation}
belongs to the subalgebra $\cL \hotimes \bbK[t]$\,.

The sum \eqref{sum-v-r} can be rewritten as follows: 
\begin{equation}
\label{sum-v-r-new}
\sum_{r \ge 1} v^{(r)} =  \sum_{k \ge 0} w_k t^k\,,  
\end{equation}
where 
\begin{equation}
\label{w-k}
w_k = \sum_{r=1}^{\infty} v^{(r)}_k  
\end{equation}

Let us choose an integer $m$. Condition \eqref{m-r-growth} implies that 
there exist $r'$ such that 
$$
m_r \ge m \qquad \forall~~ r \ge r'\,.
$$
Hence 
\begin{equation}
\label{incl-cF-m-cL}
\sum_{r=r'}^{\infty} v^{(r)}_k  \in \cF_{m} \cL
\end{equation}
for all $k$\,. 

On the other hand,  $v^{(r)}  \in \cL \hotimes \bbK[t]$ for every $r$. 
So for every $r\ge 1$ there exists $k^r_m$ such that 
$$
v^{(r)}_k \in \cF_m \cL \qquad \forall ~~ k \ge k^r_m\,.
$$
Therefore, setting $k_m = \max\{k^1_m, k^2_m, \dots, k^{r'-1}_m\}$, we get 
the inclusion
$$
\sum_{r=1}^{r'-1} v^{(r)}_k  \in \cF_{m} \cL \qquad \forall~~ k \ge k_m\,.
$$

Combining this inclusion with \eqref{incl-cF-m-cL}, we conclude that 
\begin{equation}
\label{desired-incl}
\sum_{r=1}^{\infty} v^{(r)}_k  \in \cF_{m} \cL \qquad \forall~~ k \ge k_m\,.
\end{equation}

Claim \ref{cl:cL-t-complete} is proved.
\end{proof}

We will need the following proposition
\begin{prop}
\label{prop:diff-eq}
For every degree $1$ vector $\al \in \cL$ and $\eta(t) \in  \cF_1 \cL^0 \hotimes \bbK[t]$
the equation
\begin{equation}
\label{diffura}
\frac{d}{dt} \al(t) = - \pa \eta(t) + [\eta(t), \al(t)]
\end{equation}
with initial condition
\begin{equation}
\label{initial-cond}
 \al(t) \Big|_{t=0} = \al
\end{equation}
has a unique solution in $\cL \hotimes \bbK[t]$\,.
In addition, if $\al$ satisfies the MC equation
$$
\pa \al + \frac{1}{2}[\al, \al] = 0,
$$
then so does $\al(t)$\,. 
\end{prop}
\begin{proof}
Let us set up the following iterative procedure in $r \ge 0$
$$
\al^{(0)}(t) = \al 
$$
and 
\begin{equation}
\label{iter}
\al^{(r)}(t) =  \al -  \int_0^t \pa \eta(t_1)  \, d t_1 +  \int_0^t \,  [\eta(t_1), \al^{(r-1)}(t_1)] \, d t_1\,. 
\end{equation}

Since the differences $\al^{(r+1)}(t) - \al^{(r)}(t)$ and  $\al^{(r)}(t) - \al^{(r-1)}(t)$
satisfy the equation
$$
\al^{(r+1)}(t) - \al^{(r)}(t) =  \int_0^t \,  [\eta(t_1), \al^{(r)}(t_1) - \al^{(r-1)}(t_1)] \, d t_1 
$$
and $\eta(t) \in  \cF_1 \cL^0  \hotimes \bbK[t]$, this iterative procedure converges to 
a vector $\al(t) \in \cL \hotimes \bbK[t]$\,. Moreover, $\al(t)$ satisfies the integral equation
\begin{equation}
\label{int-eq}
\al(t) =  \al -  \int_0^t \pa \eta(t_1)  \, d t_1 +  \int_0^t \,  [\eta(t_1), \al(t_1)] \, d t_1
\end{equation}
and hence differential equation \eqref{diffura} with initial condition 
\eqref{initial-cond}.

To prove the uniqueness, let us assume that $\wt{\al}(t)$ is another 
solution of \eqref{diffura} with the initial condition \eqref{initial-cond}.
Then the difference: 
$$
\psi(t) =  \wt{\al}(t) - \al(t) 
$$
satisfies the differential equation
\begin{equation}
\label{diff-eq-psi}
\frac{d}{d t} \psi(t) = [\eta(t), \psi(t)]
\end{equation}
with the initial condition 
\begin{equation}
\label{initial-psi}
\psi(t) \Big|_{t=0} = 0\,.
\end{equation}

Using \eqref{diff-eq-psi} and \eqref{initial-psi} we conclude that 
\begin{equation}
\label{int-eq-psi}
\psi(t) = \int_0^t [\eta(t_1), \psi(t_1)]\, d t_1\,.
\end{equation}

Hence the inclusion $\eta(t) \in  \cF_1 \cL^0 \hotimes \bbK[t]$ implies that 
$$
\psi(t) \in \bigcap_{m} \cF_m \cL  \hotimes \bbK[t]\,.
$$
Therefore, by Claim \ref{cl:cL-t-complete}, $\psi(t) = 0$ and $\al(t) = \wt{\al}(t)$.

The first statement of Proposition \ref{prop:diff-eq} is proved.

To prove the second statement, we consider the following element
\begin{equation}
\label{Psi-t}
\Psi(t) = \pa \al(t) + \frac{1}{2}[\al(t), \al(t)] \in \cL^2 \hotimes \bbK[t].
\end{equation}

Taking the derivative $d/dt$ and using \eqref{diffura}, we get
\[
\frac{d}{d t} (\pa \al(t) + \frac{1}{2}[\al(t), \al(t)] ) = [\eta(t), \pa \al(t)
+ \frac{1}{2}[\al(t), \al(t)]]\,.
\]
In other words, the element $\Psi(t)$ satisfies the differential 
equation
\begin{equation}
\label{diff-eq-Psi}
\frac{d}{d t}\Psi(t) = [\eta(t), \Psi(t)]\,.
\end{equation}
Since $\al$ satisfies the MC equation, we conclude that 
\begin{equation}
\label{initial-Psi}
\Psi(t) \Big|_{t=0} = 0\,.
\end{equation}

Using \eqref{diff-eq-Psi} and \eqref{initial-Psi}, we deduce that 
\begin{equation}
\label{int-eq-Psi}
\Psi(t) = \int_0^t [\eta(t_1), \Psi(t_1)] \, d t_1\,.
\end{equation}

Equation \eqref{int-eq-Psi} implies that 
$$
\Psi(t) \in  \bigcap_{m} \cF_m \cL  \hotimes \bbK[t]\,.
$$
and hence $\Psi(t) =0$\,.

Proposition \ref{prop:diff-eq} is proved. 
\end{proof}

Proposition \ref{prop:diff-eq} implies that using 
an element $\eta(t) \in \cF_1 \cL^0 \hotimes \bbK[t] $ and a MC element 
$\al \in \cL$ we can produce another MC element $\al'$ by 
solving equation (\ref{diffura}) with initial condition
(\ref{initial-cond}) and setting
$$
\al' = \al(t) \Big|_{t=1}\,.
$$
Theorem \ref{thm:isom} below states that the MC elements $\al$ and $\al'$ are 
isomorphic.

To prove this statement we need the following technical lemma.
\begin{lem}
\label{nuzhna}
If $\al$ is a MC element of $\cL$,  $\eta(t) \in \cF_1 \cL^0  \hotimes \bbK[t] $, 
and $\al(t)$ is the unique
solution of \eqref{diffura} with initial condition
(\ref{initial-cond}), then for every $\ka \in \cF_{1} \cL^0$ and
every nonnegative integer $k$, the element
\begin{equation}
\label{beta-t}
\beta(t) = \exp\left(\frac{t^{k+1}}{k+1} \ad_{\ka} \right) \al(t) 
 - \frac{\exp\left(\frac{t^{k+1}}{k+1} \ad_{\ka}\right) - 1}{\ad_{\ka}} ~ \pa \ka
\end{equation}
satisfies the differential equation
\begin{equation}
\label{diffura-beta}
\frac{d}{d t} \beta(t) = [\ti{\eta}, \beta(t)] -  \pa \ti{\eta}\,, 
\end{equation}
where
\begin{equation}
\label{tilde-eta}
\ti{\eta} (t) = t^k \ka + \exp\left(\frac{t^{k+1}}{k+1}\ad_{\ka}  \right) \eta(t)\,.
\end{equation}
\end{lem}
\begin{proof}
First, we remark that, the infinite series in  \eqref{beta-t} and 
\eqref{tilde-eta} belong to $\cL  \hotimes \bbK[t] $ due to Claim \ref{cl:cL-t-complete}.

Second, we compute the derivative $\displaystyle \frac{d}{d t} \beta(t)$ using 
\eqref{diffura}
\begin{equation}
\label{deriv-beta}
\frac{d}{d t} \beta(t) = \exp\left(\frac{t^{k+1}}{k+1} \ad_{\ka} \right) \, [t^k \ka,  \al(t)]  +
\end{equation}
$$
+ \exp\left(\frac{t^{k+1}}{k+1} \ad_{\ka} \right) \frac{d}{d t}  \al(t) 
- t^k  \exp\left(\frac{t^{k+1}}{k+1} \ad_{\ka}\right)\, \pa \ka=
$$
$$
\exp\left(\frac{t^{k+1}}{k+1} \ad_{\ka} \right) \,   [t^k \ka,  \al(t)] 
+ \exp\left(\frac{t^{k+1}}{k+1} \ad_{\ka} \right) \, [\eta(t), \al(t)] 
$$
$$
-\exp\left(\frac{t^{k+1}}{k+1} \ad_{\ka} \right) 
\, \pa \eta(t) 
-  \exp\left(\frac{t^{k+1}}{k+1} \ad_{\ka}\right)\, \pa (t^k \ka)\,.
$$

Using the notation 
$$
U_{\ka} : =  \exp\left(\frac{t^{k+1}}{k+1} \ad_{\ka} \right)
$$
and the obvious identity $[\ka, \ka]=0$,
we rewrite the derivative $\displaystyle \frac{d}{d t} \beta(t)$ as follows: 
\begin{equation}
\label{deriv-beta1}
 \frac{d}{d t} \beta(t) =
   [t^k \ka + U_{\ka} (\eta(t)),  U_{\ka}(\al(t))]  
\end{equation}
$$
- U_{\ka} ( \pa \eta(t) ) - U_{\ka} (\pa (t^k \ka))\,.
$$

On the other hand, 
$$
\beta(t) = U_{\ka}( \al(t)) 
 - \frac{U_{\ka} - 1}{\ad_{\ka}} ~ \pa \ka\,.
$$
Hence
\begin{equation}
\label{deriv-beta11}
 \frac{d}{d t} \beta(t) =
  [t^k \ka + U_{\ka} (\eta(t)), \beta(t)]  
   +
   \left[ t^k \ka + U_{\ka} (\eta(t)), \frac{U_{\ka} - 1}{\ad_{\ka}} ~ \pa \ka \right] 
\end{equation}
$$
- U_{\ka} ( \pa \eta(t) ) - U_{\ka} (\pa (t^k \ka)) =
$$
$$
[t^k \ka + U_{\ka} (\eta(t)), \beta(t)]  - \pa (t^k \ka)
$$
$$
- U_{\ka} ( \pa \eta(t) ) - \left[  \frac{U_{\ka} - 1}{\ad_{\ka}} ~ \pa \ka \,,\, 
U_{\ka} (\eta(t)) \right]\,. 
$$

Thus, to prove Lemma \ref{nuzhna}, we need to verify that 
\begin{equation}
\label{does-it-hold}
(\pa \circ  U_{\ka} - U_{\ka} \circ \pa)\, (\eta(t)) =  \left[  \frac{U_{\ka} - 1}{\ad_{\ka}} ~ \pa \ka \,,\, 
U_{\ka} (\eta(t)) \right]\,.
\end{equation}

Clearly, it suffices to check that 
\begin{equation}
\label{does-it-hold-v}
(\pa \circ  U_{\ka} - U_{\ka} \circ \pa)\, (v) =  \left[  \frac{U_{\ka} - 1}{\ad_{\ka}} ~ \pa \ka \,,\, 
U_{\ka} (v) \right]
\end{equation}
for every vector $v \in \cL$\,.

Let us denote by $\Psi_1(t)$ (resp. $\Psi_2(t)$) the left (resp. right) hand side 
of \eqref{does-it-hold-v}. It is easy to see that 
\begin{equation}
\label{Psi-i-0}
\Psi_1(0) = \Psi_2(0) = 0\,.
\end{equation}

A direct computation shows that both $\Psi_1(t)$ and $\Psi_2(t)$ satisfy the 
same differential equation: 
\begin{equation}
\label{Psi-diffura}
\frac{d}{d t} \, \Psi_i (t) = [t^k \ka, \Psi_i(t) ]+ [t^k \pa \ka,  U_{\ka}(v)]\,. 
\end{equation}

Therefore, the difference $\Psi_2(t) - \Psi_1(t)$ satisfies the integral equation 
\begin{equation}
\label{int-eq-diff-Psi}
\Psi_2(t) - \Psi_1(t) = \int_0^t  t_1^k [\ka, \Psi_2(t_1) -  \Psi_1(t_1)] d t_1\,.
\end{equation}

Since $\ka \in \cF_1 \cL^0$, we have 
$$
\Psi_2(t) - \Psi_1(t) \in   \bigcap_{m} \cF_m \cL \hotimes \bbK[t]\,.
$$
Thus $\Psi_2(t) - \Psi_1(t) = 0$ and Lemma \ref{nuzhna} is proved. 
\end{proof}

Let us now prove a statement which is used in 
Section \ref{sec:the-action}.

\begin{thm}
\label{thm:isom}
Let $\mg$ be a Lie subalgebra of $\cL^0$, $n$ be an integer $\ge 2$, and
\begin{equation}
\label{eta-mg}
\eta(t) = \sum_{k\ge 0} \eta_k t^k, \qquad \eta_k \in \cF_{m_k} \mg 
\end{equation}
be a vector  in $\cF_1 \mg  \hotimes \bbK[t] $\,. If $\al$ is a MC element of $\cL$ 
and $\al(t)$ is the unique
solution of (\ref{diffura}) with initial condition
(\ref{initial-cond}), then there exists a vector $\xi  \in \cF_1\mg^0$
such that 
\begin{equation}
\label{eta-connects}
 \exp(\ad_{\xi})\, \al -
\frac{\exp(\ad_{\xi}) - 1}{\ad_{\xi}} \pa \xi ~ = ~   
\al(t) \Big|_{t=1} \,.
\end{equation}
Moreover if
\begin{equation}
\label{eta-t-eta-0}
\eta(t) - \eta_0  \in \cF_n \mg[[t]]
\end{equation}
then there exists $\xi  \in \cF_1\mg^0$ such that 
\eqref{eta-connects} holds and
\begin{equation}
\label{xi-eta-0}
\xi - \eta_0 \in  \cF_n \mg.
\end{equation}
\end{thm}
\begin{proof}
The statement of this theorem is very similar 
to  \cite[Proposition B.7]{BDW}. Unfortunately, theorem \ref{thm:isom} is not a Corollary 
of  \cite[Proposition B.7]{BDW}. So we give a separate proof.

Let us construct recursively the following sequence of vectors in $\cF_1 \mg  \hotimes \bbK[t]$: 
\begin{equation}
\label{eta-sequence}
\eta^{(k)} (t) =  \eta^{(k)}_k  t^k +  \eta^{(k)}_{k+1}  t^{k+1} +  \eta^{(k)}_{k+2}  t^{k+2} + \dots\,,
\end{equation}
\begin{equation}
\label{base}
\eta^{(0)}(t) : = \eta(t)\,, 
\end{equation}
\begin{equation}
\label{step}
\eta^{(k+1)} (t) : = \exp\left(-\frac{t^{k+1}}{k+1}\ad_{\eta^{(k)}_k}  \right) \eta^{(k)}(t) 
-t^k \eta^{(k)}_k\,.
\end{equation}

By Lemma \ref{nuzhna}, we get the sequence of MC elements in $\cL  \hotimes \bbK[t]$
\begin{equation}
\label{al-0}
 \al^{(0)}(t) := \al(t)
\end{equation}
\begin{equation}
\label{al-step}
 \al^{(k+1)}(t) :=  \exp\left(-\frac{t^{k+1}}{k+1} \ad_{\eta^{(k)}_k} \right) \al^{(k)}(t) 
 - \frac{\exp\left(-\frac{t^{k+1}}{k+1} \ad_{\eta^{(k)}_k}\right) - 1}{\ad_{\eta^{(k)}_k}} ~ \pa \eta^{(k)}_k\,,
\end{equation}
where $ \al^{(k)}(t)$ is the unique solution of the differential equation 
\begin{equation}
\label{diffura-al-kk}
\frac{d}{d t} \al^{(k)}(t) = [\eta^{(k)} (t), \al^{(k)}(t)] -  \pa \eta^{(k)}(t)
\end{equation}
with the initial condition 
\begin{equation}
\label{al-k-initial}
 \al^{(k)}(t) \Big|_{t = 0} = \al\,.
\end{equation}

Let us prove that the sequence of vectors  
$$
\{ \eta^{(k)}_k \}_{k \ge 0}
$$
satisfies the property
$$
\eta^{(k)}_k \in \cF_{n_k} \mg
$$
with 
\begin{equation}
\label{n-k-going-up}
n_0 \le n_1 \le n_2 \le n_3 \le \dots
\qquad \textrm{and}
\qquad 
\lim_{k \to \infty} n_k = \infty \,.
\end{equation}

Without loss of generality, we may assume that the 
sequence of number $\{m_k\}_{k \ge 0}$ in \eqref{eta-mg}
is increasing 
$$
m_0 \le m_1 \le m_2 \le m_3 \le \dots
$$
and $m_0 = 1$\,.

Hence the property
$$
n_0 \le n_1 \le n_2 \le n_3 \le \dots
$$
follows immediately from the construction. 

It remains to prove that for every $m$ there exists $k_m$ such that 
\begin{equation}
\label{xi-k-desired}
\eta^{(k)}(t) \in \cF_{m} \mg \hotimes \bbK[t] \qquad \forall ~~ k \ge k_m\,.
\end{equation}

Since $\eta(t) \in \cF_1\mg  \hotimes \bbK[t]$, there exists $r_1$ such that 
$$
\eta^{(0)}_k = \eta_k \in \cF_m \mg \qquad \forall ~~ k > r_1\,.
$$

In ``the worst case scenario'', $m_k=1$ for all $k \le r_1$\,. 

So after $r_1+1$ steps \eqref{step} we get 
$$
\eta^{(r_1+1)}(t) \in \cF_2 \mg \hotimes \bbK[t].
$$
 
Since $\eta^{(r_1+1)}(t) \in \cF_2 \mg \hotimes \bbK[t]$ there exists $r_2$ such that 
all coefficients, except for the first $r_2$ ones belong to $\cF_3 \mg$\,.

Hence, after additional $r_2$ steps \eqref{step} we get 
\begin{equation}
\label{after-r2}
\eta^{(r_1+1+r_2)}(t) \in \cF_3 \mg  \hotimes \bbK[t].
\end{equation}

Therefore, in finitely many steps \eqref{step}, we will arrive at 
$$
\eta^{(k_m)}(t) \in \cF_{m} \mg  \hotimes \bbK[t]. 
$$

On the other hand, if $\eta^{(k)}(t)$ belongs to  $\cF_{m} \mg  \hotimes \bbK[t]$ then so does 
$\eta^{(k+1)}(t)$. Thus, the desired inclusion \eqref{xi-k-desired} is proved. 

Property \eqref{n-k-going-up} implies that the sequence of vectors 
\begin{equation}
\label{the-sequence}
\CH\Big(-\eta^{(k)}_k /(k+1)~,~ \dots \CH \big(-\eta^{(2)}_2 /3~, ~
\CH(-\eta^{(1)}_1 /2~,~  -\eta^{(0)}_0) \big) \Big)
\end{equation}
converges in $\cF_1\mg$ and we denote the limiting vector 
by $\xi_{\infty}$\,. In addition, the sequence of vectors 
$\{\eta^{(k)}(t)\}_{k \ge 0}$ converges to zero and hence the sequence of vectors
\begin{equation}
\label{al-k-converges}
\{ \al^{(k)}(t) \}_{k \ge 0} 
\end{equation}
converges to the constant path $\al$. 

Therefore 
\begin{equation}
\label{al-al-1-connect}
\exp(\xi_{\infty}) \left( \al(t)\Big|_{t=1}  \right) = \al
\end{equation}
Hence, setting 
$$
\xi : = - \xi_{\infty}\,,
$$ 
we prove the first part of Theorem \ref{thm:isom}. 
 
According to \eqref{eta-t-eta-0}, $\eta_k \in \cF_n \mg$ for all $k \ge 1$ in \eqref{eta-mg}. 
Therefore, the sequence of vectors \eqref{eta-sequence} satisfies
$$
\eta^{(k)} (t) \in  t^k \cF_n \mg[[t]]
$$
for all $k \ge 1$. 

Hence 
$$
\xi_{\infty}  + \eta_0 \in \cF_n \mg
$$
and the desired inclusion in \eqref{xi-eta-0} follows.
\end{proof}


\noindent\textsc{Department of Mathematics,
Temple University, \\
Wachman Hall Rm. 638\\
1805 N. Broad St.,\\
Philadelphia PA, 19122 USA \\
\emph{E-mail address:} {\bf vald@temple.edu}}

\end{document}